\documentclass[11pt]{amsart}
\usepackage{amsmath}
\usepackage{amssymb}
\usepackage{amsthm}
\usepackage{tikz}
\usetikzlibrary{matrix,arrows}
\usepackage{graphics,graphicx}
\usepackage{verbatim}
\usepackage{longtable}
\usepackage{color}
\input{xy}
\xyoption{all}
\usepackage{hyperref}
\usepackage{enumerate}

\usepackage[left=1.2in,top=1in,right=1.2in,bottom=1in]{geometry}

\input{kmacros3.sty}
\newcommand{\ie}{{\itshape i.e.} }
\newcommand{\cf}{{\itshape cf.} }

\newenvironment{flushdesc}%
  {\begin{list}{}{\setlength{\labelwidth}{0pt}
   \setlength{\itemindent}{-12pt}
   \setlength{\leftmargin}{12pt}
   }}%
  {\end{list}}

\usepackage{setspace}
\usepackage{calc}

\renewcommand{\O}{\mathcal O}
\renewcommand{\m}{\mathfrak{m}}
\DeclareMathOperator{\length}{\ell}

\title{A survey of test ideals}
\author{Karl Schwede and Kevin Tucker}
\thanks{The first author was partially supported by NSF DMS 1064485/0969145}
\thanks{The second author was partially supported by a National Science Foundation Postdoctoral Fellowship}
\subjclass[2000]{14B05, 13A35}
\keywords{test ideal, multiplier ideal, tight closure, log pair, $F$-split, $F$-singularities, $F$-pure, $F$-regular, $F$-rational, $F$-injective, Hilbert-Kunz multiplicity, $F$-signature}
\address{Department of Mathematics\\ The Pennsylvania State University\\ University Park, PA, 16802, USA}
\email{schwede@math.psu.edu}
\address{Department of Mathematics\\ University of Utah\\ Salt Lake City, UT, 84112, USA}
\email{kevtuck@math.utah.edu}
\date{}

\renewcommand{\:}{\colon}

\begin{document}

\maketitle


\setcounter{tocdepth}{1}
\tableofcontents
\numberwithin{equation}{theorem}

\section{Introduction}

Test ideals were first introduced by Mel Hochster and Craig Huneke in their
celebrated theory of tight closure  \cite{HochsterHunekeTC1}, and
since their invention have been closely tied to the theory of Frobenius splittings
\cite{MehtaRamanathanFrobeniusSplittingAndCohomologyVanishing,
  RamananRamanathanProjectiveNormality}. Subsequently, test ideals
have also found application far beyond their original scope to
questions arising in complex analytic geometry.  In this paper we give
a contemporary
survey of test ideals and their wide-ranging applications.

The test ideal has become a fundamental tool in the study of positive
characteristic algebraic geometry and commutative algebra.
To each ring $R$ of prime characteristic $p > 0$, one can associate a test ideal $\tau(R)$
which reflects properties
of the singularities of $R$.
If $R$ is regular, then $\tau(R) = R$;
more generally, if the singularities of $R$ are mild, one expects that
$\tau(R)$ is close to or equal to $R$. Conversely, severe
singularities give rise to small test ideals.

While the name \emph{test} ideal comes from Hochster and Huneke's original description as the so-called test
elements in the theory of tight closure, we initially define them
herein somewhat more directly without any reference to tight closure.
Briefly, our approach makes use of pairs $(R, \phi)$ where $R$ is an integral domain and $\phi \: R^{1/p} \to R$ is an
$R$-module homomorphism, a familiar setting to readers comfortable with the
theory of Frobenius splittings.  This perspective has numerous advantages.
In addition to the relative simplicity of the definition of the test
ideal, this setting also provides a natural segue into the connection
between the test ideal and the multiplier ideal.  Nevertheless, we do
however include a short section on the classical definition via tight
closure theory.  In addition, we also include sections focusing on various
related measures of singularities in positive characteristic besides
the test ideal, such as $F$-rationality and Hilbert-Kunz multiplicity.

We have tried to make the sections of this paper modular, attempting to
minimize the reliance of each individual section on the previous sections.
Following the initial discussion on preliminaries and notation, the main
statements throughout the remainder of the document can all be read
independently.   In addition, we have included numerous exercises scattered throughout the text.  It should be noted that while many of the exercises are quite easy, some are substantially more difficult.  We denote these with a *.

We hope that this survey will be readable and useful to a wide variety of potential audiences.
In particular, we have three different audiences in mind:
\begin{itemize}
 \item[(i)]  Readers working in characteristic $p >
   0$ commutative algebra who wish to understand generalizations to ``pairs'' and
   connections between test ideals and algebraic geometry.  These
   individuals will probably be most interested in Sections \ref{sec.TestIdealInitialDefinitions}, \ref{sec.TestIdealIdealPairs}, and~\ref{sec.AlgebrasOfMaps}.
 \item[(ii)]  Readers familiar with Frobenius
   splitting techniques who wish to learn of the language and methods
   used by their counterparts studying tight closure (Section
   \ref{sec.TightClosure}) and connections to the
   minimal model program (Section \ref{sec.ConnectionsWithAG}).  These
   individuals may wish to skim Section \ref{sec.Preliminaries}.  They may additionally find
   sections \ref{sec.TestIdealInitialDefinitions}  and \ref{sec.OtherMeasures} useful.
 \item[(iii)]  Readers with a background in complex analytic and algebraic
   geometry working on notions related to the multiplier ideal or the minimal
   model program who wish to learn about characteristic $p > 0$
   methods.  The most useful material for these individuals is likely
   found in Sections \ref{sec.Preliminaries}, \ref{sec.TestIdealInitialDefinitions}, \ref{sec.ConnectionsWithAG} and \ref{sec.TestIdealIdealPairs}.
\end{itemize}

This survey is not, however, designed to be an introduction to tight
closure, as there are already several excellent surveys and resources on the subject, see
\cite{HunekeTightClosureBook,SmithTightClosure} and \cite[Chapter 10]{BrunsHerzog}.  In addition,
the reader interested in Frobenius splitting and related vanishing
cohomology theorems is referred to
\cite{BrionKumarFrobeniusSplitting}; we shall not have occasion to
discuss global vanishing theorems in this survey.

In the appendices to this paper, we very briefly review the notions of
Cohen-Macaulay and Gorenstein rings, as well as several forms of
duality which are used minimally throughout the body of the paper.
These include local duality, Matlis duality and Grothendieck
duality. Also included is a short summary of the formalism of divisors on normal
varieties from a very algebraic point of view.

As those experts already familiar with the technical language
surrounding the development of test ideas will
be quick to notice, the terminology and notation in this paper also differs from
that used historically in the following way:

\vspace{3mm}
\noindent
\fbox{\parbox{\textwidth - 3mm}{{\bfseries Convention.  }   When referring to the test ideal $\tau(R)$ throughout this paper, we always
 refer to the \emph{big (or non-finitistic) test ideal},
 often denoted by $\widetilde{\tau}(R)$ or $\tau_b(R)$ in the
 literature.  We will call the classical test ideal, \ie the test ideal
 originally introduced by Hochster and Huneke, the \emph{finitistic
   test ideal} and denote it rather by $\tau_{\textnormal{fg}}(R)$.}}


\vspace{3mm}
\noindent
We make this simplifying convention largely because we believe there
is now consensus that the \emph{big test ideal} is the most important
notion to study.  Of course, the two notions (finitistic and
non-finitistic) coincide in most contexts and are conjectured to be
equivalent (Conjecture \ref{conj.FinitisticTestIdealIsTestIdeal}).

\vspace{5mm}
\noindent{\it{Acknowledgements: } }  The  authors would like to thank Ian Aberbach, Manuel Blickle, Florian Enescu, Neil Epstein, Zsolt Patakfalvi, Shunsuke Takagi and Michael Von Korff for useful comments on previous drafts and also for discussions related to this project.  We would also like to thank the referee, Lance Miller and Mordechai Katzman for an extremely careful reading of a previous draft of this paper, with numerous useful comments.  Parts of this paper are also based on notes the first author wrote while teaching a course on $F$-singularities at the University of Utah in Fall 2010.

\section{Characteristic $p$ preliminaries}
\label{sec.Preliminaries}

\begin{setting*}
Throughout this paper all rings are integral domains essentially of finite type over a field
$k$.  In this context, the word ``essentially'' means that $R$ is
obtained from a finitely generated \mbox{$k$-algebra} by localizing at a
multiplicative system.
In this section, that field $k$ is always perfect of prime positive characteristic $p > 0$.
\footnote{Essentially all of the positive
  characteristic material in this paper
  can easily be generalized to the setting of reduced $F$-finite rings.  In addition,
  large portions of the theory extend to the setting of excellent local rings.}
\end{setting*}

\subsection{The Frobenius endomorphism}
\label{sec:frob-endom}

When working in characteristic $p>0$, the
Frobenius or $p$-th power endomorphism is a powerful tool which can be thought of in several equivalent ways.
First and foremost, it is the ring homomorphism $F \: R \to R$ given
by $r \mapsto r^{p}$.
However, in practice it is often convenient to
distinguish between the copies of $R$ serving as the source and
target.  To that end, consider the
set $R^{1/p}$ of all $p$-th roots of elements of $R$ inside a fixed algebraic closure of the
fraction field of $R$.  The set $R^{1/p}$ is closed under addition and
multiplication, and it
forms a ring abstractly isomorphic to $R$ itself (by taking $p$-th roots).  The
inclusion $R \subseteq R^{1/p}$ is naturally identified with the
Frobenius endomorphism of $R$ and gives $R^{1/p}$ the structure of an $R$-module.

More generally, denoting by
$R^{1/p^e}$ the set of $p^e$-th roots of elements of $R$ and
iterating the above procedure gives
\[
R \subseteq R^{1/p} \subseteq R^{1/p^{2}} \subseteq \cdots \subseteq
R^{1/p^{e}} \subseteq R^{1/p^{e+1}} \subseteq \cdots
\]
where each inclusion is identified with the Frobenius
endomorphism of $R$.  Thus, as before $R^{1/p^e}$ is a ring abstractly isomorphic
to $R$, and the inclusion $R \subseteq R^{1/p^{e}}$ is
identified with the $e$-th iterate $F^{e} \: R \to R$ of Frobenius
given by $r \mapsto r^{p^{e}}$.
For any ideal $I = \langle z_1, \dots, z_m \rangle \subseteq R$, we
write  $I^{1/p^e} =\langle z_1^{1/p^e}, \dots, z_m^{1/p^e} \rangle_{R^{1/p^e}}
$ to denote the ideal (in $R^{1/p^{e}}$) of $p^{e}$-th roots of elements of $I$.
Again, we have that $R^{1/p^{e}}$ is an $R$-module via the inclusion $R \subseteq R^{1/p^{e}}$.

\begin{exercise}
\label{ex.PolynomialRingFreeModule}
Consider the polynomial ring $S = k[x_1, \dots, x_d]$.  Show that $S^{1/p^e}$ is
a free $S$-module of rank $p^{ed}$ with $S$-basis $\{ x_1^{\lambda_1/p^e} \dots x_d^{\lambda_d/p^e} \}_{0 \leq \lambda_i \leq p^e - 1}$.
\end{exercise}

If $M$ is any $R$-module, the (geometrically motivated) notation $F^e_*
M$ is often used to denote the
corresponding $R$-module coming from
restriction of scalars for $F^{e}$.  Thus, $M$ and $F^{e}_{*}M$ agree
as both sets and Abelian groups.  However, if $F^{e}_{*}m$ denotes
the element of $F^{e}_{*}M$ corresponding to $m \in M$,  we  have $r \cdot
F^{e}_{*}m = F^{e}_{*}(r^{p^{e}} \cdot m)$ for $r \in R$ and $m \in
M$.  It is easy to see that $F^{e}_{*}R$ and $R^{1/p^{e}}$ are
isomorphic $R$-modules by identifying $F^{e}_{*}r$ with $r^{1/p^{e}}$
for each $r \in R$.  While we have taken preference to the use of
$R^{1/p^{e}}$ throughout, it can be very helpful to keep both
perspectives in mind.

\begin{remark}
  We caution the reader that the module $F^{e}_{*}M$ is quite
  different from that which is commonly denoted $F^{e}(M)$ originating
  in \cite{PeskineSzpiroDimensionProjective}.  This latter notation
  coincides rather with \mbox{$M \tensor_{R} F^{e}_{*}R$} considered as an
  $R$-$F^{e}_{*}R$ bimodule.
\end{remark}

\begin{exercise}
Show that $F^{e}_{*}(\blank)$ is an exact functor on the category
of \mbox{$R$-modules}.  Conclude that $F^{e}_{*}(R/I)$ and
$R^{1/p^{e}}/I^{1/p^{e}}$ are isomorphic $R$-modules (and
\mbox{($F^{e}_{*}R = R^{1/p^{e}}$)-modules}) for any ideal $I
\subseteq R$.
\end{exercise}

\begin{lemma}
$R^{1/p^e}$ is a
finitely generated $R$-module.
\end{lemma}
\begin{proof}
Because $R$ is essentially of finite type over $k$, we  may write $R =
W^{-1}(k[x_1, \dots, x_n] / I)$ where $I$ is an ideal in $S = k[x_1,
\dots, x_n]$ and $W$ is a multiplicative system in $S/I$.  First
notice that $(S/I)^{1/p^e} = S^{1/p^e} / I^{1/p^e}$ is certainly a
finite $S/I$-module by Exercise \ref{ex.PolynomialRingFreeModule}.
But then we have that $W^{-1} \left( (S/I)^{1/p^e} \right)$ is a
finitely generated $W^{-1} (S/I)$-module, and so the result is proven
after observing
$
W^{-1} \left( (S/I)^{1/p^e} \right) = \left( (W^{p^e})^{-1} (S/I) \right)^{1/p^e} = \left( W^{-1} (S/I) \right)^{1/p^e}$.
\end{proof}

Test ideals are measures of singularities of rings of characteristic
$p > 0 $, and will be defined initially through the use of a homomorphism
$\phi \in \Hom_{R}(R^{1/p^{e}},R)$.  The
following result which demonstrates that it is reasonable to use
properties of $R^{1/p^e}$ to quantify the singularities of $R$.

\begin{theorem}\cite{KunzCharacterizationsOfRegularLocalRings}
\label{thm.KunzCharacterizationsOfRegular}
$R$ is regular if and only if $R^{1/p^e}$ is a \emph{locally}-free $R$-module.
\end{theorem}
\begin{proof}
The forward direction of the proof follows by reducing to the case of
Exercise \ref{ex.PolynomialRingFreeModule}, while the converse direction is
more involved; see \cite{KunzCharacterizationsOfRegularLocalRings} and \cite{LechInequalitiesRelatedToCertainCouples}.
\end{proof}

\subsection{$F$-purity}
\label{sec:f-purity}
Rather than requiring that $R^{1/p^{e}}$ be a free $R$-module, one might
consider the weaker condition that $R$ is a direct summand of $R^{1/p^e}$.  To that end, recall that an inclusion of rings $A
\subseteq B$ is called \emph{split} if there is an $A$-module
homomorphism $s \: B \to A$ such that $s|_{A} = \id_{A}$ (in which
case $B$ is isomorphic as an $A$-module to $A \oplus \ker(s)$, and $s$
is called a \emph{splitting of $A \subseteq B$}).

\begin{definition}
$R$ is
\emph{$F$-pure}\footnote{A
  splitting of $R \subseteq R^{1/p}$ is referred to as an $F$-splitting. At times, $F$-pure rings are also known as
  $F$-split, but we caution the reader that (particularly when in a
  non-affine setting) these terms are not always interchangeable.  } if
the inclusions $R \subseteq R^{1/p^e}$ are split.
\end{definition}

\begin{exercise}
  Suppose that $R$ is $F$-pure.  Show that, for every $R$-module $M$
  and all $e \geq 1$,
  the natural map $M \to M \tensor_{R}R^{1/p^{e}}$ is injective.

  In the setting of this paper -- where $R$ is essentially of finite
  type over a perfect field $k$ -- the converse statement also holds
  \cite{HochsterCyclicPurity}, but may fail in general.   For an arbitrary ring, the
  injectivity of $M \to M \tensor_{R}R^{1/p^{e}}$ for all $M$ is taken
  to be the definition of $F$-purity.
\end{exercise}

\begin{exercise}
Show that if $R \subseteq R^{1/p^{e}}$ is split for some $e \geq 1$,
then it is split for all $e \geq 1$.
\end{exercise}

\begin{exercise}
Suppose that $\bq \in \Spec R$ is a point such that $R_{\bq}$ is $F$-pure.  Show that there exists an open neighborhood $U \subseteq \Spec R$ of $\bq$ such that $R_{\bp}$ is $F$-pure for every point $\bp \in U$.  \\
\emph{Hint:}  Prove $R \subseteq R^{1/p^e}$ splits if and only if ``evaluation at~1'' $\Hom_R(R^{1/p^e}, R) \to R$ is surjective.
\end{exercise}

\begin{exercise}
Suppose that for every maximal ideal $\bm \in \Spec R$, $R_{\bm}$ is $F$-pure.  Show that $R$ is also $F$-pure.
\end{exercise}

In  Theorem \ref{thm.FedderCriterion} below, we exhibit a simple way of determining whether $R$ is $F$-pure.

\begin{definition} [Frobenius power of an ideal]
\label{def.FrobeniusPowerOfIdeal}
Suppose $I =
\langle y_1, \dots, y_m \rangle \subseteq R$ is an ideal.  Then for
any integer $e > 0$, we set $I^{[p^e]}$ to be the ideal
$\langle y_1^{p^e}, \dots, y_n^{p^e} \rangle_R$.
\end{definition}

\begin{exercise}
Show that $(I^{[p^{e}]})^{1/p^{e}} = I R^{1/p^{e}}$, and conclude $I^{[p^e]}$ is independent of the choice of generators of $I$.
\end{exercise}

\begin{exercise}
\label{exer:testforbracketsinregrings}
  Suppose that $R$ is a regular local ring and $I \subset R$ is an
  ideal.  If $x \in R$, show that
$x \in I^{[p^{e}]}$ if and only if $\phi(x^{1/p^{e}}) \in I$ for
  all $ \phi \in \Hom_{R}(R^{1/p^{e}},R) $ .
\end{exercise}

\begin{theorem}[Fedder's Criterion] \cite[Lemma 1.6]{FedderFPureRat}
\label{thm.FedderCriterion}
Suppose that $S = k[x_1, \dots, x_n]$ and that $R = S / I$ is a
quotient ring.  Then for any point $\bq \in \Spec R = V(I)  \subseteq \Spec S$, the local ring
$R_{\bq}$ is $F$-pure if and only if $(I^{[p]} : I ) \not\subseteq \bq^{[p]}$.  (Notice we are abusing notation by identifying $\bq \in \Spec R$ with its pre-image in $\Spec S$).
\end{theorem}
\begin{proof}
We sketch the main ideas of the proof and leave the details to the reader.  First observe that every map $\phi \in \Hom_R(R^{1/p^e}, R)$ is the quotient of a map $\psi \in \Hom_S(S^{1/p^e}, S)$ (use the fact that $S^{1/p^e}$ is a projective $S$-module).
Next prove that $\Hom_S(S^{1/p^e}, S)$ is isomorphic to $S^{1/p^e}$ as an $S^{1/p^e}$-module.  Finally, show that $(I^{[p^e]} : I)^{1/p^e} \cdot \Hom_S(S^{1/p^e}, S)$ corresponds exactly to those elements of $\Hom_S(S^{1/p^e}, S)$ which come from $\Hom_R(R^{1/p^e}, R)$.
Once this correspondence is in hand, show that the elements $\phi \in \Hom_R(R^{1/p^e}, R)$ that are surjective at $R_{\bq}$ are in bijective correspondence with the elements $x \in (I^{[p^e]} : I)$ which are not contained in $\bq^{[p^e]}$.
\end{proof}

\begin{exercise}[Coordinate Hyperplanes are $F$-pure]
Suppose that $S = k[x_1, \dots, x_n]$ and $I =  \langle x_1
\cdots x_n \rangle$.   Show that $S/I$ is $F$-pure.
\end{exercise}

\begin{starexercise}[Elliptic curves]
Show that $R = \bF_p[x,y,z]/ \langle x^3 + y^3 + z^3 \rangle$ is not $F$-pure if $p = 2, 3, 5, 11$, but is $F$-pure if $p = 7, 13$.  Generally, show that $R$ is $F$-pure if and only if $p \equiv 1 \mod 3$, in which case the associated elliptic curve is ordinary (see \cite[Page 332]{Hartshorne}).
\end{starexercise}

\begin{exercise}
Suppose that $S = \bF_p[x, y, z]$ and that $f = xy - z^2$ and $g = x^4
+ y^4 + z^4$.  Show that, for any choice of $p$,  $S/\langle f \rangle
$ is always $F$-pure while $S/\langle g \rangle$ is never $F$-pure.
\end{exercise}

\begin{exercise} \cite[Proposition 5.31]{HochsterRobertsFrobeniusLocalCohomology}
A reduced ring $R$ of characteristic $p > 0$ with total quotient ring
$K$ is called weakly normal if it satisfies the following property: for every $x \in K$, $x^p \in R$ automatically implies $x \in R$ as well. Prove that any $F$-pure ring is weakly normal.  \\
{\emph{Hint: }}  First check that a splitting of $R \subseteq R^{1/p}$
can be extended to a splitting of $K \subseteq K^{1/p}$.
\end{exercise}

\section{The test ideal}
\label{sec.TestIdealInitialDefinitions}

\begin{setting*}
In this section as before, all rings are integral domains of essentially finite type over a perfect field of characteristic $p > 0$.
\end{setting*}

While the test ideal was first described as an auxiliary component of
tight closure theory, we give a description of the
test ideal without reference to tight closure in this section.  This description has roots in
\cite{SmithFRatImpliesRat, LyubeznikSmithCommutationOfTestIdealWithLocalization, HaraTakagiOnAGeneralizationOfTestIdeals}; see also \cite[Theorem
6.3]{SchwedeCentersOfFPurity} for further statements and details.

\subsection{Test ideals of map-pairs}
We begin by introducing test ideals for pairs $(R, \phi)$, where
the addition of a homomorphism $\phi \in \Hom_{R}(R^{1/p^{e}},R)$ in
fact helps to simplify the definition.

\begin{definition}
\label{def.TestIdealDefn}
 Fix an
integer $e > 0$ and a non-zero $R$-linear map $\phi \: R^{1/p^e}
\to R$ (for example, a splitting of $R \subseteq R^{1/p^{e}}$).
We define the \emph{test ideal $\tau(R, \phi)$ } to be the unique smallest non-zero ideal $J \subseteq R$ such that $\phi(J^{1/p^e}) \subseteq J$.
\end{definition}

We make two initial observations about this definition:
\begin{itemize}
\item[(1)]  It is in no way clear that there is such a smallest ideal!  (More on this soon.)
\item[(2)]  The choice of $\phi$ can wildly change the test ideal, as in
  Exercise \ref{ex.ChoiceOfPhiMatters} below. In particular, $\tau(R,
  \phi)$ doesn't just reflect properties of $R$, but rather incorporates those of
  $\phi$ as well.
\end{itemize}

\begin{remark}
 If $\phi \: R^{1/p^e} \to R$ is as above, and $J \subseteq R$ is an
 ideal such that $\phi(J^{1/p^e}) \subseteq J$, then $J$ is said to be
 \emph{$\phi$-compatible}.  Thus $\tau(R, \phi)$ is the unique smallest non-zero $\phi$-compatible ideal.
\end{remark}

\begin{exercise}  \cf \cite{BlickleBoeckle.MinimalCartier}
 With notation as above, show that $\phi( \tau(R, \phi)^{1/p^e}) = \tau(R, \phi)$.  \\
 \emph{Hint: } Show that $\phi( \tau(R, \phi)^{1/p^e})$ is $\phi$-compatible
\end{exercise}

\begin{exercise}  \cite[Proposition 2.5]{FedderWatanabe}, \cite[Theorem 3.3]{VassilevTestIdeals}, \cite[Theorem 7.1]{SchwedeCentersOfFPurity}
Suppose that in addition $\phi \: R^{1/p^{e}} \to
R$ is surjective (for example, a splitting of $R \subseteq
R^{1/p^{e}}$).  Show that $\tau(R, \phi)$ is a radical ideal. Furthermore, prove that $R / \tau(R, \phi)$ is an $F$-pure ring.
\end{exercise}

\begin{exercise}
\label{ex.ChoiceOfPhiMatters}
Suppose that $R = \bF_2[x,y]$ and recall that $R^{1/2}$ is a free
$R$-module (Exercise \ref{ex.PolynomialRingFreeModule}) with basis $1,
x^{1/2}, y^{1/2}, (xy)^{1/2}$ .  Consider the $R$-linear three maps $\alpha, \beta, \gamma \: R^{1/2} \to R$ defined as follows:
\[
\xymatrix@R=0pt{
R^{1/2} \ar[r]^{\alpha} & R & & R^{1/2} \ar[r]^{\beta} & R  & & R^{1/2} \ar[r]^{\gamma} & R\\
1 \ar@{|->}[r] & 0          & & 1 \ar@{|->}[r] & 0          & & 1 \ar@{|->}[r] & 1 \\
x^{1/2} \ar@{|->}[r] & 0    & & x^{1/2} \ar@{|->}[r] & 1    & & x^{1/2} \ar@{|->}[r] & 0 \\
y^{1/2} \ar@{|->}[r] & 0    & & y^{1/2} \ar@{|->}[r] & 0    & & y^{1/2} \ar@{|->}[r] & 0 \\
(xy)^{1/2} \ar@{|->}[r] & 1 & & (xy)^{1/2} \ar@{|->}[r] & 0 & & (xy)^{1/2} \ar@{|->}[r] & 0 \\
}
\]
Prove that $\tau(R, \alpha) = R$, $\tau(R, \beta) = \langle y \rangle$ and $\tau(R, \gamma) = \langle xy \rangle$.
\end{exercise}

Now we turn our attention to the question of existence.  We make use
of the following somewhat technical lemma, which has its origins in tight closure theory.

\begin{lemma}\cite[Section 6]{HochsterHunekeTC1}, \cite[Proposition 3.21]{SchwedeTestIdealsInNonQGor}
\label{lem.TestElementsExist}
Suppose
that $\phi \: R^{1/p^e} \to R$ is a \mbox{non-zero} $R$-linear map.  Then there exists a non-zero $c \in R$ satisfying the following property:  for every element $0 \neq d \in R$, there exists an integer $n > 0$ such that $c \in \phi^n( (dR)^{1/p^{ne}} )$.  Here $\phi^n$ is defined to be the composition map
\[
\xymatrix{
R^{1/p^{ne}} \ar[rr]^{\phi^{1/p^{(n-1)e}}}  & & R^{1/p^{(n-1)e}} \ar[rr]^{\phi^{1/p^{(n-2)e}}} & & R^{1/p^{(n-2)e}} \ar[r] & \dots \ar[r] & R^{1/p^e} \ar[r]^{\phi} & R.
}
\]
\end{lemma}
\begin{proof}
The proof is involved, and so we omit it here and refer the
interested reader to \cite[Theorem 5.10]{HochsterHunekeFRegularityTestElementsBaseChange}.  However, let us remark that
if $b \in R$ is such that $R_b := R[b^{-1}]$ is regular and also
$\Hom_{R_b}(R^{1/p^e}_b, R_b)$ is generated by $\phi_{b}$ as an
$R^{1/p^e}_b$-module, then $c = b^l$ will suffice for some $l \gg 0$.  In fact, if additionally $b \in \phi(R^{1/p^e})$, then $c = b^3$ will work.
\end{proof}

\begin{remark}
\label{rem.CompletelyStableTestElement}
 The element $c \in R$ constructed above in Lemma \ref{lem.TestElementsExist} is an example something called a \emph{test element}.  It's construction implies that $c$ remains a test element after localization and completion (this condition is also sometimes called being a \emph{completely stable test element}).
\end{remark}

\begin{theorem}
\label{thm.TestIdealPairsExist}
With the notation of Definition \ref{def.TestIdealDefn},  fix any $c \in R$ satisfying the condition of Lemma~\ref{lem.TestElementsExist}.  Then
\[
\tau(R, \phi) = \sum_{n \geq 0} \phi^n \left( (cR)^{1/p^{ne}} \right).
\]
Here $\phi^0$ is defined to be the identity map $R \to R$.
\end{theorem}
\begin{proof}
Certainly the sum $\sum_{n \geq 0} \phi^n \left( (cR)^{1/p^{ne}} \right)$ is the smallest ideal $J \subseteq R$ both containing $c$ and such that $\phi(J^{1/p^e}) \subseteq J$.
On the other hand, if $I \subseteq R$ is any nonzero ideal such that $\phi(I^{1/p^e}) \subseteq I$, then Lemma \ref{lem.TestElementsExist} implies that $c \in I$.
This completes the proof.
\end{proof}

\begin{starexercise} \cite[Proposition 4.8]{SchwedeFAdjunction}
\label{ex.TestIdealOfPhiM}
Prove that $\tau(R, \phi) = \tau(R, \phi^m)$ for any integer $m > 0$.  \\
{\emph Hint: } The containment $\subseteq$ should is easy.  For the other containment, use a clever choice of an element from Lemma \ref{lem.TestElementsExist}.
\end{starexercise}

\begin{exercise}
\label{ex.LocalizationForMapPairs}
Suppose that $W$ is a
multiplicative system in $R$.  Let $\phi \in \Hom_R(R^{1/p^e}, R)$ and
consider the induced map $(W^{-1}\phi) \in \Hom_{W^{-1} R}(
(W^{-1}R)^{1/p^e}, W^{-1}R)$.  Then show that
$W^{-1} \tau( R, \phi) = \tau(W^{-1}R, W^{-1}\phi)$.
\\
\emph{Hint: } Suppose that $c \in R$ comes from Lemma \ref{lem.TestElementsExist}.  Prove that $c/1 \in W^{-1} R$ also satisfies the condition of Lemma \ref{lem.TestElementsExist} for $W^{-1} \phi$.
\end{exercise}

We conclude with an algorithm for computing the test ideal of a pair $(R, \phi)$.

\begin{exercise} \cite{KatzmanParameterTestIdealOfCMRings}
 Choose $c$ satisfying Lemma \ref{lem.TestElementsExist} for a
 non-zero $\phi \in \Hom_R(R^{1/p^e}, R)$ (finding such a $c$ can be quite easy, as explained in the proof of Lemma \ref{lem.TestElementsExist}).  Consider the following chain of ideals.  $J_0 = cR$, $J_1 = J_0 + \phi(J_0^{1/p^e})$, and in general $J_n = J_{n-1} + \phi(J_{n-1}^{1/p^e})$.  Show that $J_n = \tau(R, \phi)$ for $n \gg 0$.
\end{exercise}

\subsection{Test ideals of rings}

As noted above, $\tau(R, \phi)$ is depends heavily  on the choice of
$\phi$.  To remove this dependence, one simply considers all possible $\phi$ simultaneously.

\begin{definition}
We define the \emph{test
  ideal\footnote{Strictly speaking, if we follow the literature,
    $\tau(R)$ is traditionally called the \emph{big test ideal} or the
    \emph{non-finitistic test ideal} and often denoted by $\tau_b(R)$ or $\tld \tau(R)$.} $\tau(R)$} to be the unique smallest non-zero ideal $J \subseteq R$ such that $\phi(J^{1/p^e}) \subseteq J$ for all $e > 0$ and all $\phi \in \Hom_R(R^{1/p^e}, R)$.
\end{definition}

It follows from the definition that $\tau(R, \phi) \subseteq \tau(R)$ for any choice of $\phi \in \Hom_R(R^{1/p^e}, R)$.

\begin{exercise}  \cite[Theorem 4.4]{HochsterHunekeTC1}
\label{ex.TestIdealOfPolynomialRing}
Suppose that $S = k[x_1, \dots, x_n]$.  Prove that $\tau(S) = S$.\\
\emph{Hint:}  Use the fact that $S^{1/p^e}$ is a free $S$-module to
show the following: for any $d \in S$, there exists an integer $e > 0$ and $\phi \in \Hom_S(S^{1/p^e}, S)$ such that $\phi(d^{1/p^{e}}) = 1$.
\end{exercise}

Again, it is not clear that $\tau(R)$ exists.

\begin{theorem} \cite[Lemma 2.1]{HaraTakagiOnAGeneralizationOfTestIdeals}
\label{thm.TestIdealExist}
Fix any nonzero $\phi \in \Hom_R(R^{1/p}, R)$ and any $c \in R$ satisfying the condition of Lemma \ref{lem.TestElementsExist} for $\phi$.  Then
\[
\tau(R) = \sum_{e \geq 0} \sum_{\psi} \psi \left( (cR)^{1/p^{e}} \right).
\]
where the inner sum runs over all $\psi \in \Hom_R(R^{1/p^e}, R)$.
\end{theorem}
\begin{proof}
The proof is essentially the same as in Theorem \ref{thm.TestIdealPairsExist} and so is left to the reader.
\end{proof}

\begin{exercise} \cf \cite[Theorem 7.1(7)]{LyubeznikSmithCommutationOfTestIdealWithLocalization}
 Prove that for any given multiplicative system $W$, $W^{-1} \tau(R) = \tau(W^{-1} R)$.\\
\emph{Hint: } Mimic the proof of Exercise \ref{ex.LocalizationForMapPairs}.
\end{exercise}

\begin{remark}
The result of the above exercise holds in much more general settings than we consider here.  See \cite{LyubeznikSmithCommutationOfTestIdealWithLocalization, AberbachEnescuTestIdealsAndBaseChange}.
\end{remark}

\begin{exercise} \cite{FedderWatanabe}, \cite{VassilevTestIdeals}, \cite{SchwedeCentersOfFPurity}
Suppose that $R$ is an $F$-pure ring. Prove that $\tau(R)$ is a
radical ideal and that $R/\tau(R)$ is also $F$-pure.
\end{exercise}

\begin{starexercise}  \cite[Exercise 1.2.E(4)]{BrionKumarFrobeniusSplitting}
\label{ex.TestIdealInCondcutor}
Suppose that $R$ is a reduced (possibly non-normal) ring and $R^{\textnormal{N}}$ is its normalization.  The conductor ideal $\bc \subseteq R$ is the largest ideal of $R^{\textnormal{N}}$ which is also simultaneously an ideal of $R$ (it can also be described as $\Ann_R(R^{\textnormal{N}}/ R)$).  Show that $\tau(R) \subseteq \bc$.\\
\emph{Hint: } Show that $\phi(\bc^{1/p^e}) \subseteq \bc$ for all $\phi \in \Hom_R(R^{1/p^e}, R)$ and all $e \geq 0$.
\end{starexercise}

We conclude this section with a theorem which characterizes when $\tau(R) = R$.

\begin{theorem}
\label{thm.StrongFRegEquiv}
 Suppose $R$ is a domain  essentially of finite type over a perfect
 field $k$.  Then we have $\tau(R) = R$ if and only if for every $0 \neq c \in R$, there exists an $e > 0$ and an $R$-linear map \mbox{$\phi \: R^{1/p^e} \to R$} which sends $c^{1/p^e}$ to $1$.
\end{theorem}
\begin{proof}
We leave it to the reader to reduce to the
case where $R$ is a local ring with maximal ideal $\bm$.
First suppose that $\tau(R) = R$.  Choose a non-zero $c \in R$ and consider the
ideal $\sum_{e \geq 0} \sum_{\psi} \psi \left( (cR)^{1/p^{e}} \right)$
where the inner sum runs over $\psi \in \Hom_R(R^{1/p^e}, R)$.  Since
$\tau(R) = R$, this sum must equal $R$ as the sum is clearly
compatible under all $\psi$.  Therefore, since $R$ is local, there
exists an $e$ with $\psi \left( (cR)^{1/p^{e}} \right) \nsubseteq \bm$
and so $1 \in \psi \left( (cR)^{1/p^{e}} \right)$.  Thus $1 = \psi(
(cd)^{1/p^e})$ for some $d \in R$ and so by setting $\phi(\blank) = \psi(d^{1/p^e} \cdot \blank)$ we have $1 = \phi(c^{1/p^e})$ as desired.

Conversely, suppose that the condition of the theorem is satisfied.
It quickly follows that every non-zero ideal $J$ which is $\phi$-compatible for
all $\phi : R^{1/p^e} \to R$ and all $e > 0$, satisfies $1 \in J$.  Thus
$\tau(R) = R$.
\end{proof}

\begin{definition}
\label{def.strongFReg}
A ring $R$ for which $\tau(R) = R$ is called \emph{strongly $F$-regular}.
\end{definition}

\begin{theorem}\cite{HochsterHunekeTightClosureAndStrongFRegularity}
A regular ring $R$ is strongly $F$-regular.
\end{theorem}
\begin{proof}Left as an exercise to the reader  (\cf Exercise \ref{ex.TestIdealOfPolynomialRing}).
\end{proof}

\begin{exercise} \cite{HochsterHunekeTC1, HochsterHunekeFRegularityTestElementsBaseChange}
Prove that a strongly $F$-regular ring is Cohen-Macaulay.\\
\emph{Hint: } Reduce to the case of a local ring $(R, \bm)$ and find a
non-zero element $c \in R$ which annihilates $H^i_{\bm}(R)$ for all $i
< \dim R$.  Now apply the functors $H^i_{\bm}(\cdot)$ to the
homomorphism $R \to R^{1/p^e}$ which sends $1 \to c^{1/p^{e}}$.  Finally, apply the same functors to a map $\phi \: R^{1/p^e} \to R$ which sends $c^{1/p^{e}}$ to $1$.  Finally, use the criterion for checking whether a ring Cohen-Macaulay found in Appendix \ref{sec.appendixDuality}, fact (iv).
\end{exercise}

\begin{exercise} \cite{HochsterHunekeFRegularityTestElementsBaseChange}
Suppose that $R \subseteq S$ is a split inclusion of normal domains
where $S$ is strongly $F$-regular (\textit{e.g.} if $S$ is regular).  Show that $R$ is also strongly $F$-regular and in particular Cohen-Macaulay.
\end{exercise}

\subsection{Test ideals in Gorenstein local rings}
\label{subsec.TestIdealsInGorRings}

Consider now that the ring $R$ has a canonical module $\omega_R$. Applying
the functor $\Hom_R(\blank, \omega_R)$ to the natural inclusion $R
\subseteq R^{1/p^{e}}$, yields a map
\[
\Hom_R(R^{1/p^e}, \omega_R) \to \omega_R \, \, .
\]
Now, by Theorem \ref{thm.TransformationRuleCanonicalModulesFiniteMaps}, we have $\Hom_R(R^{1/p^e}, \omega_R) \cong \omega_{R^{1/p^e}} = (\omega_R)^{1/p^e}$.  Thus the map above may be viewed as a homomorphism
\[
\Phi_R^e \: \omega_R^{1/p^e} \to \omega_R \,\, .
\]
In a Gorenstein local ring  $\omega_R \cong R$, and so we have a \emph{nearly} canonical map \mbox{$\Phi_R \: R^{1/p^e} \to R$}.

\begin{setting*}
Throughout the rest of this subsection, we will assume that $R$ is a
Gorenstein\footnote{Everything in this subsection can be immediately generalized to any ring satisfying $\omega_R \cong R$, a condition sometimes called quasi-Gorenstein, or 1-Gorenstein.  It is possible to generalize many of the results in this setting to the \mbox{$\bQ$-Gorenstein} setting as well.  See Appendix \ref{sec.appendixDuality} for additional definitions.} local domain essentially of finite type over a perfect field $k$,
and the map $\Phi_{R}^{e} \: R^{1/p^{e}} \to R$ is as described above.
\end{setting*}

\begin{lemma} \cite[Lemma 7.1]{SchwedeFAdjunction} \cf \cite[Example 3.6]{LyubeznikSmithCommutationOfTestIdealWithLocalization}
The $R$-linear map $\Phi_R^e \: R^{1/p^e} \to R$ generates $\Hom_R(R^{1/p^e}, R)$ as an $R^{1/p^e}$-module.
\end{lemma}
\begin{proof}
Left to the reader.
\end{proof}

\begin{remark}
\label{rem.UniqueUpToUnit}  When one identifies $R$ with $\omega_R$,
there is a choice to be made.  In particular,
using the above definition, $\Phi_R^e$ is \emph{not}
canonically determined in an absolute sense.    Rather, $\Phi_R^e \in \Hom_R(R^{1/p^e}, R)$
is uniquely determined up to multiplication by a unit in
$R^{1/p^e}$.  Using the Cartier operator, however, it is possible to
give a more canonical construction of $\Phi_R^{e}$. See, for example
\cite{BrionKumarFrobeniusSplitting}, where $\Phi_R^{e}$ is called the \emph{trace}.
\end{remark}

\begin{exercise}
Suppose that $S = k[x_1, \dots, x_n]$ where $k$ is a perfect field and
consider the $S$-linear map $\Psi \: S^{1/p^e} \to S$ which sends $(x_1
\dots x_n)^{(p^e - 1)/p^e}$ to $1$ and all other monomials of the free
basis $\{ x_1^{\lambda_1/p^e} \dots x_n^{\lambda_n/p^e} \}_{0 \leq \lambda_i \leq p^e - 1}$ to zero.  Show that $\Psi$ generates $\Hom_S(S^{1/p^e}, S)$ as an $S$-module, and thus that $\Psi$ may be identified with $\Phi_{S}^{e}$.
\end{exercise}


At first glance, writing
$\Phi_R^e$ might seem in conflict with the exponential notation
introduced in Lemma \ref{lem.TestElementsExist}.  However, it is not
difficult to verify that -- up to multiplication by a unit as in
Remark~\ref{rem.UniqueUpToUnit} -- one has $(\Phi_R^1)^e = \Phi_R^e$.  See \cite[Appendix F]{KunzKahlerDifferentials} or \cite[Lemma 3.9, Corollary 3.10]{SchwedeFAdjunction} for further details.

\begin{theorem}
\label{thm.TauForGorensteinRing}
 Suppose that $R$ is Gorenstein and local and that $\Phi_R^e$ is as
 above (for any $e > 0$).  Then $\tau(R) = \tau(R, \Phi_R^e)$.
\end{theorem}
\begin{proof}
Certainly $\tau(R, \Phi_R^e) \subseteq \tau(R)$ since $\tau(R)$ is certainly $\Phi_R^e$-compatible.  For the converse inclusion, first note that by Exercise \ref{ex.TestIdealOfPhiM}, $\tau(R, \Phi_R^e) = \tau(R, \Phi_R^d)$ for any integer $d > 0$.  So consider now some $\phi : F^d_* R \to R$.  We know we can write $\phi(\blank) = \Phi_R^d(c^{1/p^d} \cdot \blank)$ for some element $c \in R^{1/p^d}$.  Thus
\[
\phi( \tau(R, \Phi_R^e)^{1/p^d}) = \Phi_R^d(c^{1/p^d} \tau(R, \Phi_R^d)^{1/p^d}) \subseteq \Phi_R^d(\tau(R, \Phi_R^d)^{1/p^d}) \subseteq \tau(R, \Phi_R^d) = \tau(R, \Phi_R^e)
\]
and so $\tau(R) \subseteq \tau(R, \Phi_R^e)$ as desired.
\end{proof}

\begin{phil}
 The previous theorem motivates the study of test ideal pairs $\tau(R,
 \phi)$.  For example, consider the following situation.
Suppose that $R$ is a non-normal Gorenstein domain and that $R^{\textnormal{N}}$ is its normalization.  By applying  Exercise \ref{ex.TestIdealInCondcutor},
 it can be shown that every $R$-linear map $\phi \: R^{1/p^e} \to R$
 extends (uniquely) to a $R^{\textnormal{N}}$-linear map
 $\overline{\phi} \: (R^{\textnormal{N}})^{1/p^e} \to
 R^{\textnormal{N}}$ (for details, see \cite[Exercise 1.2.E(4)]{BrionKumarFrobeniusSplitting}).

 In particular, $\Phi_R^{e} : R^{1/p^{e}} \to R$ extends to a map
 $\overline{\Phi_R^{e}}$ on the normalization as asserted above.
 However, even in the case where $R^{\textnormal{N}}$ is Gorenstein, $\overline{\Phi_R^{e}}$ is almost certainly not equal to $\Phi_{R^N}^{e}$.  Nevertheless, it may still be advantageous to work on $R^{\textnormal{N}}$, and the following exercise shows that $\tau(R)$ can be computed on the normalization.

\end{phil}

\begin{starexercise}
 Suppose that $R$ is Gorenstein and $R^{\textnormal{N}}$ is its
 normalization (which is not assumed Gorenstein).  Fix $\overline{\Phi_R^{e}}$ as above,  show that
$\tau(R^{\textnormal{N}}, \overline{\Phi_R^{e}}) = \tau(R)$.
\end{starexercise}

\section{Connections with algebraic geometry}
\label{sec.ConnectionsWithAG}

In this section we explain the connection between the test ideal and
the multiplier ideal, a construction which first appeared in complex
analytic geometry.  We assume that the reader already has some
familiarity with constructions such as divisors on normal algebraic
varieties, a resolution of singularities, and the canonical divisor.
Note that we have provided a brief review of divisors in
Appendix~\ref{sec.divisors} aimed at those mainly familiar with
algebraic techniques.
Further references for this section include
\cite{KollarSingularitiesOfPairs}, \cite{LazarsfeldPositivity2} or
\cite{BlickleLazarsfeldMultiplier} (the latter giving a particularly
satisfying introduction  to multiplier ideals).  Again, we remind the
reader that the material in this section is \emph{not} required to
understand the sections that follow.  See also Section \ref{sec.TestIdealIdealPairs}.

\begin{setting*}
  Throughout this section, let $R_0$ be a normal domain of finite type
  over $\mathbb{C}$, and let $X_0 = \Spec R_0$ denote the corresponding affine
  algebraic variety.
\end{setting*}

\subsection{Characteristic 0 preliminaries}
\label{sec:char-0-prel}

Before defining the multiplier ideal, we say a brief word about the type of resolution of singularities we consider.
\begin{definition}
Suppose that $Y$ is a variety defined over $\bC$.
 A proper birational map $\pi : Y' \to Y$ is called a \emph{log
   resolution of singularities for $Y$} if $\pi$ is proper and birational and
 $Y'$ is smooth and $\exc(\pi)$, the exceptional set of $\pi$, is a
 divisor with simple normal crossings (see Definition \ref{def.PropertiesOfWeilDivisors} in the appendix).   Given a closed subscheme $Z
 \subseteq Y$, $\pi$ is called a
 \emph{log resolution of singularities for $Z \subseteq Y$} if $\pi$
 is a log resolution of singularities for $Y$ and if both $\pi^{-1}(Z)$ and
 $\pi^{-1}(Z) \cup \exc(\pi)$ are divisors with simple normal crossings.
\end{definition}

We first define the multiplier ideal of $R_0$ in the case that $R_0$
is Gorenstein.  Let $\pi : \tld X_0 \to X_0$ be a log resolution of
singularities, and choose a canonical divisor $K_{\tld X_0}$ on $\tld X_0$, in other words, we choose a divisor $K_{\tld X_0}$ such that $\O_{\tld X_0}(K_{\tld X_0}) = \wedge^{\dim \tld X_0} \Omega_{\tld X_0/\bC}$.  We then obtain a \emph{canonical divisor} on $X_0$ as follows.
 Set $K_{X_0}$ to be the (unique) divisor on $X_0$ which agrees with
 $K_{\tld X_0}$ wherever $\pi$ is an isomorphism.  We now can define the multiplier ideal of $X_0$.

\begin{definition}
Consider the module $\Gamma\left( \tld X_0, \O_{\tld X_0}(\lceil K_{\tld X_0} - \pi^* K_{X_0} \rceil) \right)$.  This module is called the \emph{multiplier ideal} and is denoted by $\mJ(X_0)$.  It is independent of the choice of resolution.
\end{definition}

Of course, it is natural to ask why this module is an ideal.  However, set $U = \tld X_0 \setminus \exc(\pi)$ which is an open subset of both $\tld X_0$ and $X_0$ (in fact, $X_0 \setminus U$ has codimension at least $2$).  We have the natural inclusion
\[
\Gamma\left( \tld X_0, \O_{\tld X_0}(K_{\tld X_0} - \pi^* K_{X_0}) \right) \subset \Gamma\left(U, \O_{\tld X_0}(K_{\tld X_0} - \pi^* K_{X_0}) \right)
\]
But $(K_{\tld X_0} - \pi^* K_{X_0})|_U$ is zero, so the right side is just $\Gamma\left(U, \O_{{X_0}} \right) = R_0$ because $R$ is S2; see \cite[Chapter III, the method of Exercise 3.5]{Hartshorne} and \cite[Proposition 1.11]{HartshonreGeneralizedDivisorsAndBiliaison}.

The multiplier ideal has been discovered and re-discovered in many
contexts.   At least as early as \cite{GRVanishing}, it was noted that
$\mJ({X_0})$ is independent of the choice of resolution and might be an
interesting object to study.  Variants of the multiplier ideals
described also appeared throughout the work of Joseph Lipman and
others in the 1970's, see for example
\cite{LipmanTwoDimensionalDesingularization}.  However, multiplier
ideals have been most useful in the context of pairs
(definitions will be provided below) and first appeared independently
in the works of Nadel \cite{NadelMultiplierIdealSheavesAndExistence},
from the analytic perspective, as well as Lipman \cite{LipmanAdjointsOfIdeals}, from the perspective of commutative algebra.  However, the fundamental algebro-geometric theory of multiplier ideals was worked out even earlier without the formalism of multiplier ideals by Esnault-Viehweg in relation to Kodaira-vanishing and its generalization, Kawamata-Viehweg vanishing; see \cite{KawamataVanishing}, \cite{ViehwegVanishingTheorems} and \cite{EsnaultViehwegLecturesOnVanishing}.

\begin{remark}
 Smooth varieties have multiplier ideal $\mJ({X_0}) = \O_{X_0}$.  The easiest way to see this is to simply take $\pi$ as the identity (in other words, take ${X_0}$ as its own resolution).  In general, the more severe the singularities of ${X_0}$, the smaller the ideal $\mJ({X_0})$ is.
\end{remark}

\begin{example}
\label{ex.MultiplierIdealOfEllipticCurve}
 Consider the following ${X_0} = \Spec \bC[x,y,z]/\langle x^3 + y^3 + z^3 \rangle = \Spec R_0$.  Because this is a cone over a smooth variety (an elliptic curve), it has a resolution $\tld X_0 \to X_0$ obtained by blowing up the origin.  We embed $X_0 \subseteq \bC^3$ in the obvious way and blow up the origin in $\bC^3$ to obtain a log resolution $\pi : Y \to \bC^3$ of $X_0$ inside $\bC^3$.
 \[
 \xymatrix{
 E \ar@{^{(}->}[r] \ar[d]  & \tld X_0 \ar@{^{(}->}[r] \ar[d] & Y \ar[d]^{\pi} \\
 \{ \text{pt} \} \ar@{^{(}->}[r] & X_0 \ar@{^{(}->}[r] & \bC^3.
 }
 \]
Here $E$ is the elliptic curve obtained by blowing up the cone point in $X_0$.  We know $K_Y = 2F$ where $F \cong \bP^2_{\bC}$ is the exceptional divisor of $\pi$ by \cite[Chapter 8.5(b)]{Hartshorne}, thus we set $K_{\bC^3} = 0$.  It follows that $K_Y + \tld{X}_0 = 2F + \tld X_0$.  On the other hand, we know $\pi^* X_0 = \tld X_0 + 3F$, where the $3$ comes from the fact that $x^3 + y^3 + z^3$ vanishes to order 3 at the origin (the point being blown-up).  Therefore,
 \[
 K_Y + \tld{X}_0 = 2F + \tld X_0 = \pi^*X_0 - F.
 \]
 Now, $X_0 \sim 0$ in $\Pic(\bC^3) = 0$, so $\pi^*X_0 \sim 0$ also.  Thus $K_Y + \tld{X}_0 \sim -F$ and so by the adjunction formula (in the form of \cite[Chapter II, Ex. 8.20]{Hartshorne}),
 \[
 K_{\tld X_0} \sim (K_Y + \tld{X}_0)|_{\tld{X}_0} \sim (-F)|_{\tld{X}_0} \sim -E.
 \]
 So we set $K_{\tld X_0} = -E$ and then see that the corresponding $K_{X_0} = 0$ (since that is the divisor that agrees with $\sim -E$ outside of exceptional locus).

 Therefore, $\O_{\tld X_0}(K_{\tld X_0} - \pi^* K_{X_0}) = \O_{\tld X_0}(-E)$.  This sheaf can be thought of as the sheaf of functions in the fraction field of $R_0$ which vanish to order $1$ along $E$ and have no poles.  It is then clear that $\Gamma \left(U, \O_{X_0} \right)$ is just the maximal ideal of the origin in $R_0$.
\end{example}

\begin{exercise}
Compute the multiplier ideal of $\Spec \bC[x,y,z]/\langle x^n + y^n + z^n\rangle$ for arbitrary $n > 1$.
\end{exercise}

\subsection{Reduction to characteristic $p > 0$ and multiplier ideals}
\label{subsec.ReductionToCharP}

We now relate the multiplier ideal and the test ideal.  We need to briefly describe \emph{reduction to characteristic $p$,} a method of translating varieties in characteristic zero to characteristic $p > 0$.  We make many simplifying assumptions and so we refer the reader to \cite{SmithTightClosure}, \cite{HochsterHunekeTightClosureInEqualCharactersticZero}, or \cite{HunekeTightClosureBook} for a more detailed description of the reduction to positive characteristic process in this context.

Suppose that ${X_0} = \Spec R_0 = \Spec \bC[x_1, \dots, x_n] / I \subseteq
\bC^n$.  We write $I = \langle f_1, \dots, f_m \rangle$ where the
$f_i$ are polynomials.  For simplicity, we assume that all of the
coefficients of the $f_i$ are integers.  We set $R_{\bZ}$ to be the
ring $\bZ[x_1, \dots, x_n] / \langle f_1, \dots, f_m \rangle$.  For
each prime integer $p$, consider the ring $R_p := R_{\bZ}/pR_{\bZ}
\cong (\bZ/p\bZ)[x_1, \dots, x_n]/\langle f_1 \mod p, \dots, f_m \mod p\rangle$ and the associated scheme $X_p = \Spec R_p$.  The scheme $X_p$ is called a characteristic $p > 0$ model for $X_0$, and for large $p \gg 0$, $X_p$ and ${X_0}$ share many properties.  For example, $R_p$ is regular for large $p \gg 0$ if and only if $R_0$ is regular, \cite{HochsterHunekeTightClosureInEqualCharactersticZero}.  Given an ideal $J_0 \subseteq R_0$, we may also reduce it to positive characteristic by viewing a set of defining equations modulo $p$, to obtain an ideal $J_p$ (of course, $J_p$ might depend on the particular generators of $J$ chosen in small characteristics).

\begin{remark}
If the $f_i$ are not defined over $\bZ$, instead of working with
$R_{\bZ}$, one should work with $R_{A} = A[x_1, \dots, x_n] / \langle f_1, \dots, f_m\rangle$ where $A$ is the $\bZ$-algebra generated by the coefficients of the $f_i$ (and the coefficients of any other ideals one wishes to reduce to characteristic $p > 0$).  Instead of working modulo prime integers, one should quotient out by maximal ideals of $A$.
\end{remark}

Containments and equality (or non-containments and non-equality) of ideals are preserved after reduction to characteristic $p \gg 0$.  By viewing finitely generated $R_0$-modules as quotients of $R_0^{\oplus n}$, one can likewise reduce finitely generated modules to positive characteristic.  Maps between such modules can then be represented as matrices, which themselves can be reduced to characteristic $p > 0$, and properties of those maps, such as injectivity, non-injectivity, surjectivity and non-surjectivity are also be preserved for $p \gg 0$.  In particular, if a map of modules is an isomorphism after reduction to characteristic $p \gg 0$, then it is an isomorphism in characteristic zero as well.

Now, if $\pi : \tld X_0 \to X_0$ is a resolution of the singularities of $X_0$ obtained by blowing up an ideal $J_0$, then we may reduce $J_0$ to $J_p$, and then blow that up to obtain $\pi_p : \tld X_p \to X_p$, which is also a resolution of singularities for all $p \gg 0$ (of course, the existence of resolutions of singularities for arbitrary varieties in characteristic $p > 0$ is still an open question \cite{AbhyankarResolutionOfSingularitiesOfEmbedded, CossartPiltantResolutionOfThreefolds1, CossartPiltantResolutionOfThreefolds2, CutkoskyResolutionOfSingularitiesFor3Folds}).  Finally, the multiplier ideal $\mJ(X)_p$ (the multiplier ideal reduced to characteristic $p$) coincides with the characteristic $p > 0$ multiplier ideal $\mJ(X_p) := \Gamma(X_p, \O_{\tld X_p}(K_{\tld X_p} - \pi_p^* K_{X_p}))$ for $p \gg 0$.  Again, we suggest the reader see \cite{SmithTightClosure}, \cite{HochsterHunekeTightClosureInEqualCharactersticZero}, or \cite{HunekeTightClosureBook} for a much more detailed description of the reduction to characteristic $p > 0$ process.

\begin{theorem}\cite{SmithMultiplierTestIdeals, HaraInterpretation}
Suppose that $R_0$ is a Gorenstein ring in characteristic $0$ with $X_0 = \Spec R_0$.  Then $\mJ(X_0)_p = \tau(X_p)$ for all $p \gg 0$.
\end{theorem}
\begin{proof}
We will only prove the $\supseteq$ containment.  As above, we choose $\pi : \tld X_0 \to X_0$ to be a log resolution of singularities which we reduce to a positive characteristic resolution of singularities $\pi_p : \tld X_p \to X_p$.  We have the following commutative diagram of schemes in characteristic $p > 0$
\[
\xymatrix{
\tld X_p \ar[d]_{\pi_{p}} \ar[r]^{F} & \tld X_p \ar[d]^{\pi_p} \\
X_p \ar[r]_{F} & X_p
}
\]
where the maps labeled $F$ are the Frobenius maps.  By duality, see Corollary \ref{cor.FunctorialityOfOmega}, we have the following diagram of obtained from canonical modules.
\[
\xymatrix{
\Gamma(\tld X_p, \omega_{\tld X_p}) \ar[d] \ar[r]^{\Phi_{\tld X_p}} & \Gamma(\tld X_p, \pi_* \omega_{\tld X_p}) \ar[d] \\
\Gamma(X_p, \omega_{X_p}) \ar[r]_{\Phi_{X_p}} & \Gamma(X_p, \omega_{X_p}).
}
\]
By working on a sufficiently small affine chart, because $X_0$ and thus $X_p$ is Gorenstein, we may assume that $\Gamma(X_p, \omega_{X_p}) \cong R_p$ and thus assume that $\Phi_{X_p}$ is the map $\Phi_{R_p}$ discussed in Subsection~\ref{subsec.TestIdealsInGorRings}.

The image of the vertical maps is the multiplier ideal $\mJ(X_p)$ and so it follows from the diagram that the multiplier ideal is $\Phi_R$-compatible. Thus $\mJ(X_p) \supseteq \tau(R_p)$ as long as $\mJ(X_p) \neq 0$ by Theorem \ref{thm.TauForGorensteinRing}.  But $\mJ(X_p)$ is non-zero because $\pi$ is an isomorphism at the generic points of $X_p$ and $\tld X_p$.
\end{proof}

\begin{starexercise}
 While the result above holds for $p \gg 0$, it does not necessarily hold for small $p > 0$.  Consider the ring $R = \bF_2[x,y,z]/\langle z^2+xyz+xy^2 + x^2y \rangle$.  Verify the following:
\begin{itemize}
 \item[(1)]  $R$ is $F$-pure (use Fedder's criterion).
 \item[(2)]  $R$ is not strongly $F$-regular (show that $\tau(R) = \langle x,y,z \rangle$).
 \item[(3)]  The singularities of $R$ can be resolved in characteristic 2 and furthermore, $\mJ(R) = R$ (use the method of Example \ref{ex.MultiplierIdealOfEllipticCurve}).  This is more involved.
\end{itemize}
Also see \cite{ArtinWildlyRamifiedZ2Actions} and \cite[Example 7.12]{SchwedeTuckerTestIdealFinite}.
\end{starexercise}

\subsection{Multiplier ideals of pairs}

We have so far only defined the multiplier ideal for a Gorenstein ring. We now consider a more general setting.

\begin{definition}
 Suppose that $X$ is a normal variety of any characteristic.  Then a \emph{$\bQ$-divisor $\Delta$} is a formal sum of prime Weil divisors with rational coefficients (in other words, a $\bQ$-divisor is just a divisor where we allow rational coefficients).  A $\bQ$-divisor $\Delta$ is called \emph{effective} if all its coefficients are positive.  Two $\mathbb{Q}$-divisors $\Delta_1$ and $\Delta_2$ are said to be \emph{$\bQ$-linearly equivalent}, denoted $\Delta_1 \sim_{\bQ} \Delta_2$, if there exists an integer $n > 0$ such that $n \Delta_1$ and $n \Delta_2$ are linearly equivalent Weil divisors.  We say that a $\bQ$-divisor $\Gamma$ is \emph{${\bQ}$-Cartier} if there exists an integer $n$ such that $n \Gamma$ is an integral Cartier divisor.  In that case, the \emph{index of $\Gamma$} is the smallest positive integer $n$ such that $n \Gamma$ is an integral Cartier divisor.  See Appendix \ref{sec.divisors} for a more detailed discussion from an algebraic perspective.
\end{definition}

Instead of working with an arbitrary variety, one often works with a pair.

\begin{definition}
\label{def.logQGorenstein}
 A \emph{log $\bQ$-Gorenstein pair} (or simply a \emph{pair} if the context is understood) is the data of a normal variety $X$ of any characteristic and an effective $\bQ$-divisor $\Delta$ such that $K_X + \Delta$ is $\bQ$-Cartier.  A pair is denoted by $(X, \Delta)$.  The \emph{index of $(X, \Delta)$} is defined to be the index of $K_X + \Delta$.
\end{definition}

\begin{remark}
There are many reasons why one should consider pairs.  Of course, you might be interested in a divisor inside an ambient variety, and pairs are natural in that context.  Also, not all varieties are Gorenstein, and log $\bQ$-Gorenstein pairs have associated multiplier ideals (as we'll see shortly).  Another reason that pairs occur is if one changes the variety.  In particular, suppose that $Y \to X$ is a morphism of varieties; for example, a closed immersion, a blow-up, a finite map, or a fibration.  Then in many cases properties of $X$ (respectively $Y$) can be detected by studying an appropriate pair $(X, \Delta)$ (respectively $(Y, \Delta)$), see for example \cite{KawakitaInversion} or \cite{KawamataSubadjunctionOne}.
However, many of the deepest applications of multiplier ideals of pairs are revealed by observing the behavior of the multiplier ideal as the coefficients of $\Delta$ vary.  This is not a topic we will explore in this article.  We invite the reader to see \cite{LazarsfeldPositivity2, EinMultiplierIdealsVanishing, EinLazSmithVarJumpingCoeffs, SiuMultiplierIdealSheavesSurvey, SiuDynamicMultiplierSurvey} for more background.
\end{remark}

Before we define the multiplier ideal, we first we state how to pull-back $\bQ$-Cartier divisors.  Suppose that $\Gamma$ is a $\bQ$-Cartier divisor on $X$ and $\pi : Y \to X$ is a birational map from a normal variety $Y$.  Choose $n$ such that $n \Gamma$ is Cartier and define $\pi^* \Gamma$ to be ${1 \over n} \pi^* (n \Gamma)$.

\begin{definition}
Suppose that $(X_0, \Delta_0)$ is a log $\bQ$-Gorenstein pair in characteristic zero.  Set $\pi : \tld X_0 \to X_0$ to be a log resolution of singularities of a pair $(X_0, \Delta_0)$ (in other words, we also assume that $\Supp(\pi^{-1} \Delta_0) \cup \exc(\pi)$ is a simple normal crossings divisor).  Consider the module $\Gamma\left( \tld X_0, \O_{\tld X_0}(K_{\tld X_0} - \pi^* (K_{X_0} + \Delta_0) ) \right)$.  This module is called the \emph{multiplier ideal} and is denoted by $\mJ(X_0, \Delta_0)$.  It is independent of the choice of log resolution.

 We say that $(X_0, \Delta_0)$ has \emph{log terminal singularities} if $\mJ(X_0, \Delta_0) = \O_{X_0}$.  For more about log terminal singularities, see \cite{KollarSingularitiesOfPairs} and \cite{KollarMori}.
\end{definition}

\begin{exercise}
 Suppose that $(X_0, \Delta_0)$ is a pair where $X$ is smooth and $\Delta_0$ has simple normal crossings support.  Show that $\mJ(X_0, \Delta_0) = \O_{X_0}( - \lfloor \Delta_0 \rfloor)$.
\end{exercise}

\begin{exercise}
 Suppose that $X_0$ is smooth and that $D$ is an effective Cartier divisor on $X_0$.  Prove that $\mJ(X_0, D) = \O_{X_0}(-D)$.  \\
\emph{Hint:}  Use the projection formula, \cite[Chapter II, Exercise 5.1(d)]{Hartshorne}
\end{exercise}

\subsection{Multiplier ideals vs test ideals of divisor pairs}

Previously we considered test ideals of pairs $(R, \phi)$ where $R$ is a ring of characteristic $p$ and $\phi : R^{1/p^e} \to R$ is an $R$-linear map.  We will see that this pair is essentially the same data as a log $\bQ$-Gorenstein pair $(X, \Delta)$.

\begin{setting*}
Throughout this subsection, $R$ is used to denote a \emph{normal} ring essentially of finite type over a perfect field of characteristic $p > 0$.  Furthermore, $X = \Spec R$.
\end{setting*}

Suppose we are given a $\phi \in \Hom_R(R^{1/p^e}, R)$.  The module $\Hom_R(R^{1/p^e}, R)$ is S2 both as an $R$-module and as an $R^{1/p^e}$-module.  Therefore it is determined by its localizations outside a set $Z \subseteq X$ of codimension 2 (the singular locus), see \cite[Theorem 1.12]{HartshorneGeneralizedDivisorsOnGorensteinSchemes}.  We set $U$ to be the smooth locus of $X$ and consider the sheaf $\sHom_{\O_U}(\O_U^{1/p^e}, \O_U)$.  Tensoring with $\omega_U = \O_U(K_U)$ and using the projection formula, we see that this module is isomorphic to
\begin{align*}
& \sHom_{\O_U}\left( \O_U^{1/p^e} \tensor_{\O_U} \O_U(K_U), \O_U(K_U)\right) \\ \cong & \sHom_{\O_U}( (\O_U(p^e K_U))^{1/p^e}, \O_U(K_U) ) \\ \cong & (\sHom_{\O_U}( \O_U(p^e K_U), \O_U(K_U)))^{1/p^e} \\ \cong & (\O_U( (1 - p^e) K_U ) )^{1/p^e}.
\end{align*}
Here the first isomorphism is due to the projection formula and the fact that $(F^e)^* \sL = \sL^{p^e}$ for any line bundle $\sL$.  The second isomorphism is Theorem \ref{thm.TransformationRuleCanonicalModulesFiniteMaps} from the Appendix.  The last isomorphism is just \cite[Chapter II, Exercise 5.1(b)]{Hartshorne}.

Because $\Hom_R(R^{1/p^e}, R)$ is S2, it is determined on $U$, see \cite[Theorem 1.12]{HartshorneGeneralizedDivisorsOnGorensteinSchemes}.  Therefore,
\[
\begin{array}{rl}
& \Hom_R(R^{1/p^e}, R) \\
\cong & \Hom_{\O_U}(\O_U^{1/p^e}, \O_U) \\
\cong & \Gamma(U, (\O_U( (1-p^e)K_U))^{1/p^e}) \\
\cong & \Gamma(X, (\O_X( (1-p^e) K_X))^{1/p^e}).
\end{array}
\]
Of course, $(\O_X( (1-p^e) K_X))^{1/p^e}$ is abstractly isomorphic to $\O_X( (1-p^e) K_X)^{1/p^e}$.
Therefore, $\phi$ may be viewed as a global section of $\O_X( (1-p^e)K_X)$.  In particular, by \cite[Proposition 7.7]{Hartshorne} which also works for reflexive rank-1 sheaves on normal varieties, $\phi$ determines an effective divisor $D_{\phi}$ linearly equivalent to $(1 - p^e)K_X$.  Set
\[
 \Delta_{\phi} := {1 \over p^e - 1} D_{\phi}.
\]

\begin{exercise}
 If $R = \bF_p[x,y]$, find $\phi$ such that $\Delta_{\phi}$ is the sum of the two coordinate axes (in other words, that $\Delta_{\phi} = \Div(xy)$.  Find a $\phi$ such that $\Delta_\phi = 0$.
\end{exercise}

It is straightforward check that there is a bijection between the following two sets, see \cite[Theorems 3.11, 3.13]{SchwedeFAdjunction}.
\begin{equation}
\label{eq:divmaplbcorr}
\left\{ \parbox{2in}{
\begin{center}
Non-zero $R$-linear maps $\phi:
F^{e}_{*} R \to R$ up to  pre-multiplication
by units.
\end{center}
} \right\}
\quad \longleftrightarrow \quad
\left\{ \parbox{1.8in}{
\begin{center}
Effective $\Q$-divisors $\Delta$ on $X = \Spec R$ such that
$(1-p^{e})(K_{X}+\Delta) \sim 0$.
\end{center}
} \right\}.
\end{equation}

\begin{starexercise} \cite[Proof of Theorem 6.7]{SchwedeCentersOfFPurity},  \cf \cite[Proof of the main theorem]{HaraWatanabeFRegFPure},
\label{ex.MapExtensionAlongBlowup}
 Suppose that $\phi : R^{1/p^e} \to R$ is a non-zero $R$-linear map and that $\pi : \tld X \to X = \Spec R$ is a birational map where $\tld X$ is normal.  Prove that $\phi$ induces a map
\[
\left(\O_{\tld X}(\lceil K_{\tld X} - \pi^* (K_X + \Delta_{\phi}) \rceil)\right)^{1/p^e} \to  \O_{\tld X}(\lceil K_{\tld X} - \pi^* (K_X + \Delta_{\phi}) \rceil)
\]
which agrees with $\phi$ wherever $\pi$ is an isomorphism.
\end{starexercise}

\begin{theorem} \cite[Theorem 3.2]{TakagiInterpretationOfMultiplierIdeals}
\label{thm.MultiplierIdealBecomesTestIdeal}
 Suppose that $X_0 = \Spec R_0$ is a variety of finite type over $\bC$ and that $(X_0, \Delta_0)$ is a log $\bQ$-Gorenstein pair.  Then, after reduction to characteristic $p \gg 0$, $\left( \mJ(X_0, \Delta_0) \right)_p = \tau(X_p, \phi_{\Delta_p})$ where $\phi_{\Delta_p}$ is a map corresponding to $\Delta_p$ as in \eqref{eq:divmaplbcorr} above.
\end{theorem}
\begin{proof}
 We only briefly sketch the proof.  By working on a smaller affine chart if necessary, we may assume that $n(K_{X_0} + \Delta_0) \sim 0$ for some integer $n$.  We reduce both $X_0$ and $\Delta_0$ (and a log resolution) to characteristic $p$ and notice that if $n(K_{X_0} + \Delta_0) \sim 0$, then that property is preserved after reduction to characteristic $p \gg 0$.  Thus we may always assume that there exists an $e > 0$ such that $(1 - p^e)(K_{X_p} + \Delta_p)$ is Cartier.  We set $\phi_{\Delta_p}$ to be a map corresponding to $\Delta_p$ via \eqref{eq:divmaplbcorr} above.

Now use the exercise \ref{ex.MapExtensionAlongBlowup} above to show that $\mJ(X_0, \Delta_0)_p = \mJ(X_p, \Delta_p)$ is $\phi$-compatible.  Thus the inclusion $\supseteq$ is rather straightforward.  The converse inclusion requires additional techniques that we won't cover here.
\end{proof}

\begin{remark}
While multiplier ideals are quite closely to test ideals, many basic properties which hold for multiplier ideals fail spectacularly for test ideals.  For example, it follows immediately from the definition that every multiplier ideal is integrally closed (we suggest the reader prove this as an exercise).  However, not every test ideal is integrally closed \cite{McDermottTestIdealsInDiagonalHypersurfaces} and furthermore, every ideal in a regular ring is the test ideal of an appropriate pair $\tau(R, \langle f \rangle^t)$, see \cite{MustataYoshidaTestIdealVsMultiplierIdeals} (here $(R, \langle f \rangle ^t)$ is a pair as discussed in Section \ref{sec.TestIdealIdealPairs} below).
\end{remark}

\section{Tight closure and applications of test ideals}
\label{sec.TightClosure}

In this section we survey the test ideal's historic connections with tight closure theory.  It is not necessary to read this section in order to understand later sections.  We should also mention that as a survey of tight closure, this section is completely inadequate.  Some important aspects of tight closure theory are completely missing (for example, phantom homology).  Again, we refer the reader to the book \cite{HunekeTightClosureBook} or \cite[Chapter 10]{BrunsHerzog} for a more complete account.

\begin{setting*}
In this section, all rings are assumed to be integral domains essentially of finite type over a perfect field of characteristic $p > 0$.
\end{setting*}

Suppose that $R \subseteq S$ is an extension of rings.  Consider an ideal $I \subseteq R$ and its extension $I S$.  We always have that $(I S) \cap R \supseteq I$, however:

\begin{lemma}
\label{lem.ExtendAndContractFromASplitting}
With $R \subseteq S$ as above and further suppose the extension splits as a map of $R$-modules.  Then
\[
(IS) \cap R = I.
\]
\end{lemma}
\begin{proof}
Fix $\phi : S \to R$ to be the splitting given by hypothesis.  Suppose that $z \in (IS) \cap R$, in other words, $z \in IS$ and $z \in R$.  Write $I = (x_1, \dots, x_n)$, we know that there exists $s_i \in S$ such that $z = \sum s_i x_i$.  Now, $z = \phi(z) = \phi\left( \sum s_i x_i \right) = \sum x_i \phi(s_i) \in I$ as desired.
\end{proof}

A converse result holds too.

\begin{theorem}\cite{HochsterCyclicPurity}
\label{thmCyclicPureVsPure}
Suppose that $R \subseteq S$ is a finite extension of approximately Gorenstein\footnote{Nearly all rings in geometry satisfy this condition.  Explicitly, a local ring $(R, \bm)$ is called approximately Gorenstein if for every integer $N > 0$, there exists $I \subseteq \bm^N$ such that $R/I$ is Gorenstein.} rings, a condition which every ring in this section automatically satisfies.  If for every ideal $I \subseteq R$, we have $IS \cap R = S$, then $R \subseteq S$ splits as a map of $R$-modules.
\end{theorem}
\begin{proof} See, \cite{HochsterCyclicPurity} \end{proof}

Consider now what happens if the extension $R \subseteq S$ is the Frobenius map.  Recall from Definition \ref{def.FrobeniusPowerOfIdeal} that if $I = (x_1, \dots, x_m)$, then $I^{[p^e]} = (x_1^{p^e}, \dots, x_m^{p^e})$.  It is an easy exercise to verify that this is independent of the choice of generators $x_i$.

\begin{exercise}
 Notice that $R$ is abstractly isomorphic to $R^{1/p^e}$ as a ring.  Show that under this isomorphism, $I^{[p^e]}$ corresponds to the extended ideal $I (R^{1/p^e})$ coming from $R \subseteq R^{1/p^e}$.
\end{exercise}

\begin{definition}
Given an ideal $I \subseteq R$, the \emph{Frobenius closure of $I$} (denoted $I^F$) is the set of all elements $z \in R$ such that $z^{p^e} \in I^{[p^e]}$ for some $e > 0$.  Equivalently, it is equal to the set of all elements $z \in R$ such that $z \in (I R^{1/p^e})$ for some $e > 0$.
\end{definition}

\begin{remark}
The set $I^F$ is an ideal.  Explicitly, if $z_1, z_2 \in I^F$, then $z_1^{p^a} \in I^{[p^a]}$ and $z_2^{p^b} \in I^{[p^b]}$.  Notice that we may assume that $a = b$.  Thus $z_1 + z_2 \in I^F$.  On the other hand, clearly $h z_1 \in I^F$ for any $h \in R$.
\end{remark}

We point out a several basic facts about $I^F$ mostly for comparison with tight closure (defined below).

\begin{proposition}
Fix $R$ to be a domain and $(x_1, \dots, x_n) = I \subseteq R$ an ideal.
\begin{itemize}
\item[(i)]  $(I^F)^F = I^F$.
\item[(ii)]  For any multiplicative set $W$, $(W^{-1} I)^F = W^{-1} (I^F)$.
\item[(iii)]  $R$ is $F$-split if and only if $I = I^F$ for all ideals $I \subseteq R$.
\end{itemize}
\end{proposition}
\begin{proof}
The proof of property (i) is left to the reader.
For (ii), we note that $(\supseteq)$ is obvious.  Conversely, suppose that $z \in (W^{-1} I)^F$, thus $z^{p^e} \in (W^{-1} I)^{[p^e]} = W^{-1}(I^{[p^e]})$.  Therefore, for some $w \in W$, $wz^{p^e} \in I^{[p^e]}$, which implies that $(wz)^{p^e} \in I^{[p^e]}$ and the converse inclusion holds.
Part (iii) is obvious by Theorem \ref{thmCyclicPureVsPure}.
\end{proof}

Now we define tight closure.

\begin{definition}\cite{HochsterHunekeTC1}
Suppose that $R$ is an $F$-finite domain and $I$ is an ideal of $R$, then the \emph{tight closure of $I$} (denoted $I^*$) is defined to be the set
\[
\{ \; z \in R \; | \; \exists \; 0 \neq c \in R\text{ such that } c
z^{p^e} \in I^{[p^e]} \text{ for all } e \geq 0 \;  \}.
\]
\end{definition}

It should be noted that tight closure is notoriously difficult to
compute.
For a survey on computations of tight closure (using highly geometric methods) we suggest reading \cite{BrennerHerzogVillamayorThreeLectures}.  Also see \cite{SinghComputationofTightClosureDiagonalHypersurfaces,KatzmanTheComplexityofFrobPowers,BrennerTightClosureOfProjectiveBundles,BrennerComputingTightClosureInDim2}.

\begin{proposition}
Suppose we have an ideal $(x_1, \dots, x_n) = I \subseteq R$ where $R$ is an $F$-finite domain.
\begin{itemize}
\item[(i)]  $I^*$ is an ideal containing $I$, \cite[Proposition 4.1(a)]{HochsterHunekeTC1}.
\item[(ii)]  $(I^*)^* = I^*$, \cite[Proposition 4.1(e)]{HochsterHunekeTC1}.
\item[(iii)]  It is known that the formation of $I^*$ does \emph{NOT} commute with localization, \cite{BrennerMonsky}.
\item[(iii')]  If $I$ is generated by a system of parameters, then the formation of $I^*$ does commute with localization, \cite{AberbachHochsterHunekeLocalofTCandModsofFinProjDim,SmithTightClosureParameter}.
\item[(iv)]  If $\tau(R) = R$, then $I^* = I$ for all ideals $I$, \cite[Theorem 3.1(d)]{HochsterHunekeTightClosureAndStrongFRegularity}.
\item[(v)]  We always have the containment $I^* \subseteq {\overline{I}}$ where $\overline{I}$ is the integral closure of $I$, \cite[Theorem 5.2]{HochsterHunekeTC1}.
\end{itemize}
\end{proposition}
\begin{proof}
For (i), suppose that $c z^{p^e} \in I^{[p^e]}$ and $d y^{p^e} \in I^{[p^e]}$ for all $e \geq 0$ for certain $c, d \in R \setminus \{ 0\}$.  Then $cd (z+y)^{p^e} \in I^{[p^e]}$ for all $e \geq 0$.  Of course, clearly $I^*$ contains $I$ (choose $c = 1$).  Property (ii) is left to the reader.  The proof of (iii) is beyond the scope of this survey, see \cite{BrennerMonsky}.  The proof of (iii') can be found in \cite{SmithTightClosureParameter} where it is actually shown that tight closure coincides with plus-closure.

For (iv), suppose that $z \in I^*$ and $\tau(R) = R$.  Choose $c \neq 0$ such that $c z^{p^e} \in I^{[p^e]}$ for all $e \geq 0$.  We know that there exists an $e > 0$ and $\phi : R^{1/p^e} \to R$ which sends $c$ to $1$.  Write $cz^{p^e} = \sum a_i x_i^{p^e}$.  Then $z = \phi(cz^{p^e}) = \sum x_i \phi(a_i) \in I$.  For (v) we give a hint in the form of a characterization of $\overline{I}$.  One has $x \in \overline{I}$ if and only if there exists $0 \neq c \in R$ such that $cx^n \in I^n$ for all $n > 0$.
\end{proof}

We now state some important additional more subtle properties of tight closure.

\begin{theorem} (\cf \cite{HunekeTightClosureBook}, \cite[Chapter 10]{BrunsHerzog}, \cite[Section 1.3]{SmithTightClosure})
\begin{description}
\item[Persistence]  Given any map of rings $R \to S$, $I^*S \subseteq (IS)^*$.
\item[Finite extensions]  If $R \subseteq S$ is a finite extension of rings, then $(IS) \cap R \subseteq I^*$ for all ideals $I \subseteq R$.  Also see \cite{SmithTightClosureParameter}.
\item[Colon Capturing]  If $R$ is local and $x_1, \dots, x_d$ is a system of parameters for $R$, then we have $(x_1, \dots, x_{i}) :_R x_{i+1} \subseteq (x_1, \dots, x_i)^*$.
\end{description}
\end{theorem}

Perhaps the most important open problem in tight closure theory is the following.

\begin{conjecture}
\label{conj:strong=weak}
 $R$ is strongly $F$-regular if and only if $I^* = I$ for all ideals $I$.
\end{conjecture}

\begin{remark}
A number of special cases of this conjecture are known; see \cite[Theorem 12.2]{HunekeTightClosureBook}, \cite{LyubeznikSmithStrongWeakFregularityEquivalentforGraded}, \cite{LyubeznikSmithCommutationOfTestIdealWithLocalization}, \cite{AberbachMacCrimmonSomeResultsOnTestElements} and \cite{StubbsThesis}.  One should note that the method of Lemma \ref{lem.ExtendAndContractFromASplitting} immediately yields the ($\Rightarrow$) implication.
\end{remark}

In fact, one can also simply use tight closure of ideals to define a slightly different variant of test ideals.

\begin{definition}
Suppose that $R$ is a domain essentially of finite type over a perfect field.  Define $\tau_{\textnormal{fg}}(R)$ to be $\bigcap_{I \subseteq R} (I :_R I^*)$.  This ideal is called the \emph{finitistic test ideal} or sometimes the \emph{classical test ideal}.
\end{definition}

\begin{remark}
The definition of the test ideal in this article is non-standard.  Normally $\tau_{\textnormal{fg}}(R)$ is called the \emph{test ideal}, while the ideal we denoted by $\tau(R)$ is called the \emph{non-finitistic test ideal}, or sometimes the \emph{big test ideal} and is commonly denoted by $\tld \tau(R)$ or $\tau_b(R)$.  It is hoped that these two potentially different ideals always coincide, see the conjecture below.  However, even if they do not, there now seems to be consensus that the \emph{non-finitistic test ideal} is the better notion.
\end{remark}

\begin{conjecture}
\label{conj.FinitisticTestIdealIsTestIdeal}
The ideals $\tau_{\textnormal{fg}}(R)$ and $\tau(R)$ coincide.
\end{conjecture}

\begin{exercise}  \cf \cite[Theorem 7.1(4)]{LyubeznikSmithCommutationOfTestIdealWithLocalization}
Prove that $\tau(R) \subseteq \tau_{\textnormal{fg}}(R)$.
\end{exercise}

\begin{exercise} \cite[Proposition 2.5]{FedderWatanabe}
Suppose that $R$ is $F$-pure, show that $\tau_{\textnormal{fg}}(R)$ is a radical ideal.
\end{exercise}

\subsection{The {B}rian\c con-{S}koda theorem}

Now we move on to one of the classical applications of tight closure theory, a very simple proof of the {B}rian\c con-{S}koda theorem.

\begin{theorem}\cite[Theorem 5.4]{HochsterHunekeTC1}
\label{thm:brianconskoda}
Let $R$ be an $F$-finite domain, and $(u_1, \dots, u_n) = I \subseteq R$ an ideal.  Then for every natural number $m$,
\[
\overline{I^{m+n}} \subseteq \overline{I^{m+n - 1}} \subseteq (I^{m})^*
\]
and so
\[
\tau(R)  \overline{I^{m+n}} \subseteq \tau_{\textnormal{fg}}(R) \overline{I^{m+n}} \subseteq I^m.
\]
\end{theorem}

\noindent
This gives a particularly nice statement in the case that $R$ is strongly $F$-regular (because $\tau(R) = R$).

\begin{proof}[This proof is taken from \cite{HochsterFoundations}]
For any $y \in \overline{I^{m+n - 1}}$, we know that there exists $0 \neq c \in R$ such that $cy^l \in (I^{m+n-1})^l$ for all $l \geq 0$, see \cite[Exercise 1.5]{HunekeSwansonIntegralClosure}.  Consider a monomial $u_1^{a_1} \dots u_n^{a_n}$ where $a_1 + \dots + a_n = l(m+n-1)l$.  Write each $a_i = b_il + r_i$ where $0 \leq r_i \leq l-1$.  We claim that the sum of the $b_i$ is at least $m$, which will imply that the monomial is contained in $(I^m)^{[l]}$ for all $l$ such that $l = p^e$.  However, if the sum $b_1 + \dots + b_m \leq m-1$, then $l(m+n-1) = \sum a_i \leq l(m-1) + n(l-1) = l(m+n - 1) - n < l(m+n-1)$, which implies the claim.

Thus $cy^{p^e} \in I^{[p^e]}$ and so $y \in (I^m)^*$ as desired.
\end{proof}

\begin{exercise}
\label{exerc:pulloffbracketpower}
  Following the method in the above proof, if $\mathfrak{a}$ is an ideal generated by $r$
  elements, show that $\mathfrak{a}^{rp^{e}} \subseteq
  \mathfrak{a}^{[p^{e}]}\mathfrak{a}^{(r-1)p^{e}}$ for all $e \geq 0$.
\end{exercise}

\subsection{Tight closure for modules and test elements}

\begin{definition}
Suppose that $R$ is a domain and that $M$ is an $R$-module.  We define the \emph{tight closure of $0$ in $M$}, denoted $0^*_M$ as follows.
\[
0^*_M = \{ \; m \in M \; | \; \exists \; 0 \neq c \in R, \text{ such
  that } 0 = m \tensor c^{1/p^e} \in M \tensor_R R^{1/p^e} \text{ for
  all $e \geq 0$.} \; \}
\]
If $0^*_M = 0$, then we say that \emph{$0$ is tightly closed in $M$}.
\end{definition}

\begin{remark}
\label{rem.TightClosureOfModules}
More generally, given a submodule $N \subseteq M$, one can define $N^*_M \subseteq M$, the \emph{tight closure of $N$ in $M$}.  However, this submodule is just the pre-image of $0^*_{M/N} \subseteq M/N$ under the natural surjection $M \to M/N$, see \cite[Remark 8.4]{HochsterHunekeTC1}.
\end{remark}

It is known that a ring is strongly $F$-regular if and only if $0$ is tightly closed in every module, \cite{HochsterFoundations}.  By Remark \ref{rem.TightClosureOfModules}, note that $I^* = I$, if and only if $0^*_{R/I} = 0$.

We conclude with one more definition.

\begin{definition}
An element $0 \neq c \in R$ is called a \emph{finitistic test element} if for every ideal $I \subseteq R$ and every $z \in I^*$, we have
\[
c z^{p^e} \in I^{[p^e]}.
\]
An element $0 \neq c \in R$ is called a \emph{(big) test element} if for every $R$-module $M$ and every $z \in 0^*_M$, we have
\[
0 = z \tensor c^{1/p^e} \in M \tensor_R R^{1/p^e}.
\]
\end{definition}

\begin{theorem}
Suppose that $c \in R$ is chosen as in Lemma \ref{lem.TestElementsExist} for some non-zero $R$-linear map $\phi : R^{1/p^e} \to R$.  Then $c$ is a finitistic test element and a big test element.
\end{theorem}
\begin{proof}
First consider the finitistic case.

Suppose that $z \in I^*$.  Then there exists a $0 \neq d \in R$ such that $d z^{p^e} \in I^{[p^e]}$ for all $e \geq 0$.  Fix $\phi : R^{1/p} \to R$.  It follows from Lemma \ref{lem.TestElementsExist} that there exists an integer $e_0 > 0$ such that $\phi^{e_0}(d^{1/p^{e_0}}) = c$.  Applying $\phi^{e_0}$ to the equation $d z^{p^e} \in I^{[p^e]}$ for $e \geq e_0$ yields
\[
c z^{p^{e - e_0}} \in \phi^{e_0}(I^{[p^e]}) \subseteq I^{[p^{e - e_0}]}.
\]
Since this holds for all $e \geq e_0$, we see that $c$ is indeed a finitistic test element.

We leave the non-finitistic case to the reader.  It is essentially the same argument but instead one considers the map $E \tensor_R R^{1/p^e} \to E \tensor_R R = E$ which defined by $z \tensor d^{1/p^e} \mapsto z \tensor \phi(d^{1/p^e}) = \phi(d^{1/p^e})z$ where $\phi : R^{1/p^e} \to R$ is the map given in the hypothesis.
\end{proof}

\begin{starexercise} \cite{HochsterHunekeTC1}, \cite{HochsterFoundations}
\label{ex.TestElementsMakeUpTestElements}
Show that $\tau_{\textnormal{fg}}(R)$ is generated by the set of finitistic test elements.  Even more,
\[
\tau_{\textnormal{fg}}(R) = \{ \text{ the set of all of the finitistic test elements of $R$} \} \cup \{ 0 \}.
\]
Furthermore, show that $\tau(R)$ is likewise generated by the big test elements of $R$.\\
\emph{Hint: } Suppose that $z \in I^*$, show that for every $e > 0$, $z^{p^e} \in (I^{[p^e]})^*$.
\end{starexercise}

\begin{remark}
\label{rem.TestIdealEtymology}
Test ideals are made up of test elements, or those elements which can be used to \emph{test} tight closure containments.  This is the etymology of the name ``test ideals''.
\end{remark}

\begin{exercise}
Suppose that $(R, \bm)$ is a local domain and $E$ is the injective hull of the residue field $R/\bm$.  Show that $0^*_E$ is the Matlis dual of $R/\tau(R)$.
\\ \emph{Hint: } Choose an element $0 \neq c \in R$ satisfying the conclusion of Lemma \ref{lem.TestElementsExist}.  Show that
\[
0^*_E = \bigcap_{e \geq 0} \ker \left( E \to E \tensor R^{1/p^e} \right)
\]
where the maps in the intersection send $z \mapsto z \tensor c^{1/p^e}$.  Show that this intersection is the Matlis dual of the construction of the test ideal found in Theorem \ref{thm.TestIdealExist}.  See Appendix \ref{sec.appendixDuality} for Matlis Duality.
\end{exercise}

\section{Test ideals for pairs $(R, \ba^t)$ and applications}
\label{sec.TestIdealIdealPairs}

\renewcommand{\a}{\mathfrak{a}}
\renewcommand{\b}{\mathfrak{b}}

As mentioned before, many of the most important applications of multiplier ideals in characteristic zero were for multiplier ideals of pairs.  Another variant of pairs not discussed thus-far in this survey is the pair $(R_0, \ba_0^t)$ where $R_0$ is a normal $\bQ$-Gorenstein domain of finite type over $\bC$, $\ba_0$ is a non-zero ideal and $t \geq 0$ is a real number.  The associated multiplier ideals $\mJ(\Spec R_0, \ba_0^t)$ are important in many applications and have themselves become objects of independent interest, see for example \cite{LazarsfeldPositivity2} and \cite{LazarsfeldLeeLocalSyzygiesOfMultiplierIdeals}.  Inspired by this relation, N. Hara and K.-i. Yoshida defined test ideals for such pairs.  They also proved the analog of Theorem \ref{thm.MultiplierIdealBecomesTestIdeal} showing that the multiplier ideal coincides with the test ideal after reduction to characteristic $p \gg 0$.

\begin{setting*}
In this section, unless otherwise specified, all rings are assumed to be integral domains essentially of finite type over a perfect field of characteristic $p > 0$.
\end{setting*}

\subsection{Initial definitions of $\ba^t$-test ideals}

We now show how to incorporate an ideal $\a$ and coefficient $t \in
\Q_{\geq 0}$ into the test ideal.  An important motivating case is
when $R$ is in fact regular; in this situation, one should think of
this addition as roughly measuring the singularities of (a multiple
of) the closed subscheme of $\Spec(R)$ defined by $\a$.

\begin{definition} \cite{HaraYoshidaGeneralizationOfTightClosure}, \cite{SchwedeSharpTestElements}
Suppose that $R$ is a ring, let $\a \subseteq R$ be a non-zero ideal, and $t \in
\Q_{\geq 0}$.  We define the \emph{test ideal $\tau(R, \a^{t})$} (or
simply $\tau(\a^{t})$ when confusion is unlikely to arise) to be the
unique smallest non-zero ideal $J \subseteq R$ such that
$\phi((\a^{\lceil t (p^{e}-1) \rceil}J)^{1/p^e}) \subseteq J$ for all $e > 0$ and all $\phi \in \Hom_R(R^{1/p^e}, R)$.
\end{definition}

\begin{remark}
  In other words, $\tau(R,\a^{t})$ is in fact the unique smallest non-zero
  ideal which is \mbox{$\phi$-compatible} for all $\phi \in (\a^{\lceil
  t(p^{e}-1)\rceil})^{1/p^{e}}\cdot \Hom_{R}(R^{1/p^{e}},R)$ and all $e \geq
0$.  Again, it is unclear that a smallest such non-zero ideal exists.
\end{remark}

\begin{remark}
 As in previous sections, N. Hara and K.-i. Yoshida's original definition was the \emph{finitistic test ideal} of a pair.  In particular, they defined $\tau_{\textnormal{fg}}(R, \a^t)$ to be $\bigcap_{I \subseteq R} (I : I^{*\ba^t})$ where $I^{*\ba^t}$ is the $\ba^t$-tight closure of $I$, see Definition \ref{def.atTightClosure} below.  These two ideals are known to coincide in many cases including the case that $R$ is $\bQ$-Gorenstein.
\end{remark}

\begin{theorem} \cite{HaraTakagiOnAGeneralizationOfTestIdeals}
\label{thm:attestexists}
 Suppose that $R$ is a ring, $\a \subseteq R$ is a non-zero ideal, and $t \in
\Q_{\geq 0}$.
Then,
for any non-zero $c \in \tau(R, \psi)$, we have
\[
\tau(R, \a^{t}) = \sum_{e \geq 0} \sum_{\phi} \phi((c\a^{\lceil t (p^{e}-1)\rceil})^{1/p^{e}})
\]
where the inner sum runs over $\phi \in \Hom_{R}(R^{1/p^{e}}, R)$.
More generally, the above equality remains true if $c$ is replaced by
any non-zero element of $\tau(R, \a^{t})$.
\end{theorem}

\begin{proof}
This is left as an exercise to the reader.  For a hint, reduce to the case of Lemma \ref{lem.TestElementsExist}.
\end{proof}

\begin{exercise} \cite[Remark 6.6]{HaraYoshidaGeneralizationOfTightClosure}
   Suppose that $R$ is a ring, $\a_{1}, \ldots, \a_{k} \subseteq R$ are  non-zero ideals, and
$t_{1}, \ldots, t_{k} \in
\Q_{\geq 0}$.  Imitating the above result and its proof, show that one can define a test
ideal $\tau(\a_{1}^{t_{1}}\a_{2}^{t_{2}} \cdots \a_{k}^{t_{k}})$ as the unique smallest non-zero ideal $J \subseteq R$ such that
$$\phi\left(((\prod_{i=1}^{k}\a_{i}^{\lceil t_{i} (p^{e}-1)) \rceil})J\right)^{1/p^e}) \subseteq J$$ for all $e > 0$ and all $\phi \in \Hom_R(R^{1/p^e}, R).$
\end{exercise}

The following property of test ideals was inspired by analogous statement for multiplier ideals.

\begin{theorem} \cite[Remark 2.12]{MustataTakagiWatanabeFThresholdsAndBernsteinSato}, \cite[Corollary 2.16]{BlickleMustataSmithDiscretenessAndRationalityOfFThresholds},  \cite[Lemma 3.23]{BlickleSchwedeTakagiZhang}
Suppose that $R$ is a ring, $\a \subseteq R$ is a non-zero ideal, and $s,t \in
\Q_{\geq 0}$.  If $s \geq t$, then $\tau(\a^{s}) \subseteq \tau(\a^{t})$.
  Furthermore, there exists $\epsilon > 0$ such that $\tau(\a^{s}) =
  \tau(\a^{t})$ for all $s \in [t, t+\epsilon]$.
\end{theorem}

\begin{proof}
 If $s \geq t$, then $(\a^{\lceil s(p^{e}-1)
    \rceil})^{1/p^{e}} \subseteq (\a^{\lceil t(p^{e}-1)
    \rceil})^{1/p^{e}}$
and the first statement is left as an exercise to for the reader.  For the
second statement,
choose non-zero elements
  $c \in \tau(R, \a^{t+1})$ and $x \in \a$.  Using the Noetherian
  property of $R$ with Theorem~\ref{thm:attestexists} above, there
  exists an $N \geq 0$ such that
\[
\tau(R, \a^{t}) = \sum_{e = 0}^{N} \sum_{\phi} \phi((cx\a^{\lceil t (p^{e}-1)\rceil})^{1/p^{e}})
\]
where the inner sum runs over all $\phi \in \Hom_{R}(R^{1/p^{e}},R)$.
Let $\epsilon = \frac{1}{p^{N}-1}$, so that $x  \a^{\lceil t
  (p^{e}-1)\rceil} \subseteq \a^{\lceil (t+\epsilon) (p^{e}-1)\rceil}$
for all $0 \leq e \leq N$.  Thus, we have
\[
\tau(R, \a^{t}) \subseteq \sum_{e \geq 0} \sum_{\phi}
\phi((c\a^{\lceil (t +\epsilon)(p^{e}-1)\rceil})^{1/p^{e}}) = \tau(R, \a^{t+\epsilon})
\]
as desired.
\end{proof}

\begin{definition}  \cite{TakagiWatanabeFPureThresh}, \cite{MustataTakagiWatanabeFThresholdsAndBernsteinSato}, \cite{BlickleMustataSmithDiscretenessAndRationalityOfFThresholds}
\label{def.FJumpingNumber}
  A positive real number $\xi$ is called an \emph{$F$-jumping
    number} of the ideal $\a$ if $\tau(\a^{\xi}) \neq \tau(\a^{\xi -
    \epsilon})$ for all $\epsilon > 0$.  If $R$ is strongly
  $F$-regular, then the smallest $F$-jumping number of $\a$ is called
  the \emph{$F$-pure threshold} of $\a$.
\end{definition}

\begin{remark}
 The $F$-jumping numbers were introduced as characteristic $p > 0$ analogs of jumping numbers of multiplier ideals in characteristic zero, \cite{EinLazSmithVarJumpingCoeffs}.  We point out that a great interest in $F$-jumping numbers has revolved around proving that the set of $F$-jumping numbers form a discrete set of rational numbers.  See \cite{HaraMonskyFPureThresholdsAndFJumpingExponents, BlickleMustataSmithDiscretenessAndRationalityOfFThresholds, BlickleMustataSmithFThresholdsOfHypersurfaces, KatzmanLyubeznikZhangOnDiscretenessAndRationality, TakagiTakahashiDModulesOverRingsWithFFRT, SchwedeTakagiRationalPairs, BlickleSchwedeTakagiZhang, SchwedeDiscretenessQGorenstein}.
\end{remark}

\begin{theorem} \cite{HaraYoshidaGeneralizationOfTightClosure}
\label{thm:atproperties}
  Suppose that $R$ is a ring, $\a, \b \subseteq R$ are non-zero ideals, and $s,t \in
\Q_{\geq 0}$.
\begin{enumerate}[(i.)]
\item If $\a \subseteq b$, then
  $\tau(\a^{t}) \subseteq \tau(\b^{t})$.  Furthermore, if $\a$ is a
  reduction of $\b$ ({i.e. } $\bar{\a} = \bar{\b}$), then
  $\tau(\a^{t}) = \tau(\b^{t})$.
\item
We have $\a \; \tau(\b^{s}) \subseteq \tau(\a \b^{s})$ with equality if
$\a$ is principal.  In particular, if $R$ is strongly $F$-regular,
then $\a \subseteq \tau(\a)$.
\item (Skoda)
If $\a$ is generated by $r$
elements, then $
\tau(\a^{r}\b^{s}) = \a \; \tau(\a^{r-1} \b^{s})$.
\end{enumerate}
\end{theorem}

\begin{proof}
The proof is left as an exercise to the reader.  As a hint for {\it (iii.)}, using Exercise~\ref{exerc:pulloffbracketpower}, we have
\[
c\a^{r} \a^{r(p^{e}-1)} = c \a^{rp^{e}} = c \a^{[p^{e}]}
\a^{(r-1)p^{e}} = \a^{[p^{e}]} c \a^{r-1} \a^{(r-1)(p^{e}-1)} \, \, .
\]
Then manipulate
\[
\tau(\a^{r}\b^{s}) = \sum_{e\geq 0} \sum_{\phi} \phi( ( c
  \a^{r}\a^{r(p^{e}-1)}\b^{\lceil t(p^{e}-1) \rceil}
  )^{1/p^{e}} ).
\]
\end{proof}

\begin{starexercise}  \cite[Theorem 2.1]{HaraYoshidaGeneralizationOfTightClosure} \cite[Theorem 4.1]{HaraTakagiOnAGeneralizationOfTestIdeals}
\label{exer:atskoda}
  If $\a$ has a reduction generated by at most $r$ elements, show that
  $\tau(\a^{t}) = \a \, \tau(\a^{t-1})$ for any $t \geq r$.  In particular,
  $\tau(\a^{h}) = \a^{h-r+1} \, \tau(\a^{r-1}) \subseteq \a^{h -r + 1}
  \subseteq \a$ for any
  integer $h \geq r$.
\end{starexercise}

\begin{exercise} \cite[Proposition 3.1]{HaraTakagiOnAGeneralizationOfTestIdeals}, \cf \cite{SchwedeCentersOfFPurity}
\label{exer:atlocalizes}
For any multiplicative system~$W$, prove that $W^{-1} \tau(R,
  \a^{t}) = \tau( W^{-1}R, (W^{-1}\a)^{t})$.
\end{exercise}

\begin{remark}
 While the above exercise was first stated as Proposition 3.1 in \cite{HaraTakagiOnAGeneralizationOfTestIdeals}, the proof provided therein is not sufficient.  In particular, one needs the existence of a ``test element'' that remains a test element after localization.  See Remark \ref{rem.CompletelyStableTestElement}.  However, once one uses Lemma \ref{lem.TestElementsExist} to construct such a test element, the proof in \cite{HaraTakagiOnAGeneralizationOfTestIdeals} goes through without substantial change.
\end{remark}

\subsection{$\ba^t$-tight closure}

\begin{definition}
\label{def.atTightClosure}
Suppose that $R$ is a ring $\ba \subseteq R$ is an ideal and $t \in \bQ_{\geq 0}$.  For any ideal $I$ of $R$, the \emph{$\a^{t}$-tight closure of $I$} (denoted $I^{*\a^{t}}$) is defined to be the set
\[
\{ \; z \in R \; | \; \exists \;  0 \neq c \in R\text{ such that } c
\a^{\lceil t(p^{e}-1)\rceil}z^{p^e} \in I^{[p^e]} \text{ for all } e
\geq 0 \;  \}.
\]
See Definition \ref{def.FrobeniusPowerOfIdeal} for the definition of $I^{[p^e]}$.
\end{definition}

\begin{exercise} \cite{HaraYoshidaGeneralizationOfTightClosure}
  Show that $I^{*\a^{t}}$ is an ideal containing $I$, and that $\tau(R,
  \a^{t}) \, I^{*\a^{t}} \subseteq I$ for all $I$.
\end{exercise}

\begin{exercise} \cite[Proposition 1.3(4)]{HaraYoshidaGeneralizationOfTightClosure}
  If $\a$ is a reduction\footnote{Again, an ideal $\a \subseteq \b$ such that $\overline{\a} = \overline{\b}$.} of $\b$, prove $I^{*\a^{t}} =
  I^{*\b^{t}}$ for all $t \in \Q_{\geq 0}$.
\end{exercise}

\begin{exercise}
 In the definition of $I^{*\a^t}$-tight closure, demonstrate that the containment $c \a^{\lceil t(p^{e}-1)\rceil}z^{p^e} \in I^{[p^e]} $ may be replaced by the containment $c \a^{\lceil tp^e \rceil}z^{p^e} \in I^{[p^e]}$ or the containment $c \a^{\lfloor t(p^{e}-1)\rfloor}z^{p^e} \in I^{[p^e]}$ without change to $I^{*\a^t}$.  In fact, the original definition of $I^{*\a^t}$ was the former of these, see \cite{HaraYoshidaGeneralizationOfTightClosure}.
\end{exercise}

\begin{exercise}  \cite{SchwedeSharpTestElements}
\label{ex.SharpTestElements}
 Show that $\tau_{\textnormal{fg}}(\ba^t) := \bigcap_{I \subseteq R} (I^{* \ba^t} : I)$ coincides with the set of $c \in R$ satisfying the following condition: whenever $z \in I^{*\ba^t}$, then $c \ba^{\lceil t(p^e - 1)\rceil} z^{p^e} \subseteq I^{[p^e]}$ for all $e \geq 0$.  \\
{\emph{Hint: }} Show that if $z \in I^{*\ba^t}$, then for every $e \geq 0$, $\ba^{\lceil t(p^e - 1) \rceil} z^{p^e} \subseteq (I^{[p^e]})^{*\ba^t}$.
\end{exercise}

\begin{remark}
Note when $\a = R$, we recover the original definition of tight
closure.  In general, while there are many similarities between tight closure and $\a^{t}$-tight
  closure, there are some very important differences as well.  In
  fact, $\a^{t}$-tight closure fails to be a closure
  operation at all: in many cases $(I^{*\a^{t}})^{*\a^{t}}$ is
  strictly larger than $I^{*\a^{t}}$ (in other words, the operation is not idempotent).  However, if $\ba$ is primary to a maximal ideal, A. Vraciu has developed an alternate version of $\ba^t$-tight closure which shares many aspects of the same theory but which is idempotent, see \cite{VraciuNewatTightClosure}.
\end{remark}

\begin{exercise}
  If $R$ is strongly $F$-regular and $\a$ is a non-zero principal ideal, show that $I^{*\a} = I : \a $ for all ideals $I$.  Use this to
  produce an example where $(I^{*\a^{t}})^{*\a^{t}}$ is strictly larger than $I^{*\a^{t}}$.
\end{exercise}

\subsection{Applications}

The test ideal $\tau(\ba^t)$ of a pair was introduced because of the connection between $\tau(R)$ and $\mJ(X_0 = \Spec R_0)$ discussed in Section \ref{sec.ConnectionsWithAG}.  In particular, working in characteristic zero, many of the primary applications of multiplier ideals involved pairs of the form $(X_0, \ba_0^t)$, or more generally $(X_0, \ba_0^t \bb_0^s)$.  One such formula is the subadditivity formula, see \cite{DemaillyEinLazSubadditivity, MustataMultiplierIdealofASum}.  In \cite{TakagiFormulasForMultiplierIdeals}, S. Takagi proved analogous results for the test ideal.  In fact, Takagi was able to prove subadditivity formula on singular varieties (for simplicity, we only handle the smooth case below).  Using reduction to characteristic $p > 0$, one can then obtain the same formula for multiplier ideals, thus obtaining a new result in characteristic zero algebraic geometry.  Very recently, E. Eisenstein obtained a geometric characteristic-zero proof of the results for singular varieties \cite{EisensteinGeneralizationsOfRestrictionFormula}.

\begin{starexercise} \cite[Proposition 2.1]{TakagiFormulasForMultiplierIdeals}
\label{exer:forsubadd}
  Given ideals $\a, \b \subseteq R$ and numbers $s, t \in \bR_{\geq 0}$, prove that
  $I^{*\a^{t}} \tau(\a^{t}\b^{s}) \subseteq I \, \tau(\b^{s})$ for any ideal
  $I \subseteq R$.
\end{starexercise}

\begin{theorem}
\label{thm.subadditivity}
  \cite[Theorem 2.7]{TakagiFormulasForMultiplierIdeals} (Subadditivity) Suppose $R$ is a regular
  local ring, and $\a, \b$ are ideals in $R$.  For any $s, t \in
  \Q_{\geq 0}$, we have
\[
\tau(\a^{t}\b^{s}) \subseteq \tau(\a^{t})\tau(\b^{s}) \, \, .
\]
If instead $R$ is of finite type over a perfect field $k$ but is not assumed to be regular, then still:
\[
 \mathfrak{J}(R/k) \tau(\a^{t}\b^{s}) \subseteq \tau(\a^{t})\tau(\b^{s}) \, \, ,
\]
where $\mathfrak{J}(R/k)$ is the Jacobian ideal of $R$ over $k$, see \cite[Section 16.6]{EisenbudCommAlgWithAView}.
\end{theorem}

\begin{proof}

We only prove the first statement.
Fix $0 \neq c \in \tau(\a^{t}\b^{s}) \subseteq \tau(\a^{t}) \cap \tau(\b^{s})$.
From Exercise~\ref{exer:forsubadd}, we have
$\tau(\a^{t})^{*\a^{t}} \tau(\a^{t}\b^{s}) \subseteq \tau(\a^{t})
\tau(\b^{s})$.
To finish the proof, it suffices to show that $\tau(\a^{t})^{*\a^{t}}
= R$.  For all $\psi \in
\Hom_{R}(R^{1/p^{e}},R)$, we know $\psi((c \a^{\lceil t(p^{e}-1)\rceil})^{1/p^{e}}
  ) \subseteq \tau(\a^{t})$.  Thus, using
Exercise~\ref{exer:testforbracketsinregrings}, we see $c \a^{\lceil
  t(p^{e}-1) \rceil} \subseteq \tau(\a^{t})^{[p^{e}]}$ for all $e \geq
0$.  In particular, this shows $1 \in \tau(\a^{t})^{*\a^{t}}$ as desired.
\end{proof}

\begin{corollary} \cite[Theorem 0.1]{TakagiFormulasForMultiplierIdeals}
 Let $X_0 = \Spec R_0$ be a normal $\bQ$-Gorenstein variety over $\bC$ and let $\mathfrak{J}(R_0/\bC)$ be the Jacobian ideal sheaf of $X$ over $\bC$.  Let $\ba_0, \bb_0 \in \O_{X_0}$ be two non-zero ideal sheaves and fix real numbers $s, t \geq 0$.  Then
\[
  \mathfrak{J}(R_0/\bC) \mJ(X_0, \a_0^{t}\b_0^{s}) \subseteq \mJ(X_0, \a_0^{t})\mJ(X_0, \b_0^{s}) \, \, .
\]
\end{corollary}
\begin{proof}
 This follows via reduction to characteristic $p > 0$, see Subsection \ref{subsec.ReductionToCharP}.  Apply Theorem \ref{thm.subadditivity} and an analog of Theorem \ref{thm.MultiplierIdealBecomesTestIdeal} which can be found in \cite[Theorem 6.8]{HaraYoshidaGeneralizationOfTightClosure}.
\end{proof}

We now discuss another application of the subadditivity formula for test ideals:  the growth of symbolic and ordinary powers of an ideal.
Recall that the $n$-th symbolic power of an ideal $\a \subseteq R$ is
given by $\a^{(n)} = (\a^{n} W^{-1}R) \cap R$, where $W \subseteq R$ is
the compliment of the union of the associated primes of $\a$.  In the case that $\a$ is a prime ideal, $\a^{(n)}$ coincides with the $\ba$-primary component of $\a^n$.

\begin{theorem} \cite{HochsterHunekeComparisonOfSymbolic},  \cf \cite{EinLazSmithSymbolic, HochsterHunekeFineBehaviorOfSymbolicPowers, TakagiYoshidaGeneralizedTestIdealsAndSymbolicPowers}
\label{thm:symbpowers}
  Let $(R, \m)$ be a regular local ring with infinite residue field $R/\m$.  Let $\a$ be any non-zero
  ideal of $R$ and let $h$ be the maximal height of any associated
  prime ideal of $\a$.  Then
$\a^{(hn)} \subseteq \a^{n}$
for all integers $n \geq 1$.
\end{theorem}

\begin{proof}
  Since $R$ is regular (and hence also strongly $F$-regular), using
  subadditivity (and Theorem~\ref{thm:atproperties} \textit{(iii.)}) we have
\[
\a^{(hn)} \subseteq
  \tau(\a^{(hn)})\subseteq \left(
    \tau((\a^{(hn)})^{1/n})\right)^{n}
\]
for all integers $n \geq 0$.
 Thus, it suffices to check
  that $ \tau((\a^{(hn)})^{1/n}) \subseteq \a^{(1)} = \a$, which may
  be done after localizing at each associated prime $\bp$ of $\a$.
  Since $R_{\bp}$ has dimension at most $h$ and infinite residue
  field, every nonzero ideal of $R_{\bp}$ has a reduction generated by no more than
  $h$ elements, \cite[Proposition 8.3.7, Corollary 8.3.9]{HunekeSwansonIntegralClosure}.  Thus, since $\a^{(hn)}R_{\bp} = \a^{hn}R_{\bp}$, we
  have (using Exercises~\ref{exer:atskoda} and \ref{exer:atlocalizes})
\[
\tau(R, (\a^{(hn)})^{1/n})R_{\bp} = \tau(R_{\bp},
(\a^{hn}R_{\bp})^{1/n}) = \tau(R_{\bp}, \a^{h}R_{\bp}) \subseteq
\a R_{\bp}
\]
and we conclude that $\a^{(hn)} \subseteq \a^{n}$ for all $n \geq 0$.
\end{proof}

\begin{starexercise}
  Modify the proof of Theorem~\ref{thm:symbpowers} to show the
  stronger statement
  $\a^{(kn)} \subseteq (\a^{(k-h+1)})^{n}$ for all $k \geq n$.
\end{starexercise}

\section{Generalizations of pairs: algebras of maps}
\label{sec.AlgebrasOfMaps}

In this short section we discuss a common generalization of the pairs $(R, \phi)$ and $(R, \ba^t)$ previously introduced.  In fact this generalization encompasses all studied types of pairs, triples, etc.  This idea has also been generalized to modules in \cite{BlickleAlgebras} although we will not work in that generality.

\begin{setting*}
In this section, unless otherwise specified, all rings are assumed to be integral domains essentially of finite type over a perfect field of characteristic $p > 0$.
\end{setting*}

Fix a ring\footnote{In fact, M. Blickle has shown that the theory below can be extended $F$-finite rings without which may or may not be reduced, \cite{BlickleAlgebras}.} $R$ and set $\sC^e = \Hom_R(R^{1/p^e}, R)$.  The test ideals and related notions such as $F$-purity and $F$-regularity are detected by looking for Frobenius splittings and similar special elements of $\sC^e$ for various $e \geq 0$.  Fundamentally, all pairs we have previously considered restrict the potential elements of $\sC^e$.  We abstract the idea of restricting potential elements of $\sC^e$ as follows.

Consider now the Abelian group
\[
\sC = \bigoplus_{e \geq 0} \sC_e = \bigoplus_{e \geq 0} \Hom_R(R^{1/p^e}, R)
\]
We can turn this into a non-commutative $\bN$-graded algebra by the following multiplication rule.  For $\alpha \in \sC_e$ and $\beta \in \sC_d$ we define
\[
\alpha \cdot \beta := \left( \alpha \circ \beta^{1/p^e} : R^{1/p^{d+e}} \to R \right) \in \sC_{d+e}.
\]
Explicitly, $\beta^{1/p^e}$ is the $R^{1/p^e}$-linear map $R^{1/p^{d+e}} \to R^{1/p^e}$ defined by the rule
\[
\beta^{1/p^e}\left( x^{1/p^{d+e}}\right) = \left(\beta(x^{1/p^d}) \right)^{1/p^e}.
\]
We then compose with a map $\alpha : R^{1/p^e} \to R$ to obtain $\alpha \cdot \beta$.

We call $\sC$ the \emph{complete algebra of maps on $R$}.  Notice that $\sC$ is not commutative (and $\sC_0$ is not even central).  Even more, this algebra is not generally finitely generated \cite{KatzmanANonFGAlgebraOfFrob}.

\begin{remark}
 Suppose for simplicity that $(R, \bm)$ is local.  The algebra $\sC$ is, up to some choices of isomorphism, Matlis dual to $\sF(E)$, the algebra of (iterated-)Frobenius actions on $E$, the injective hull of the residue field $R/\bm$. See \cite{LyubeznikSmithCommutationOfTestIdealWithLocalization}.
\end{remark}

\begin{definition} \cite[Section 3]{SchwedeTestIdealsInNonQGor}
An (algebra-)pair $(R, \sD)$ is the combined information of $R$ and a graded-subalgebra $\sD \subseteq \sC$ such that $\sD_0 = \sC_0 = \Hom_R(R, R) \cong R$.
\end{definition}

\begin{example}
 Suppose $R$ is a ring and $\ba \subseteq R$ is an ideal.  Then for any real number $t \geq 0$, we can construct the submodule
\[
\sD_e := (\ba^{\lceil t(p^e - 1) \rceil})^{1/p^e} \cdot \sC_e = \bigoplus_{e \geq 0} \left( (\ba^{\lceil t(p^e - 1) \rceil})^{1/p^e} \cdot \Hom_R(R^{1/p^e}, R) \right).
\]
One can verify that $\bigoplus_{e \geq 0} \sD_e$ forms a graded subalgebra which we denote by $\sC^{\ba^t}$.
\end{example}

\begin{exercise}[Different roundings and algebras]
 Prove that $\sC^{\ba^t}$ is indeed a graded subalgebra of $\sC$.  Give an example to show that $\bigoplus_{e \geq 0} (\ba^{\lfloor t(p^e - 1) \rfloor})^{1/p^e} \cdot \sC_e$ is not a graded subalgebra but $\bigoplus_{e \geq 0} (\ba^{\lceil tp^e \rceil})^{1/p^e} \cdot \sC_e$ is (although its first graded piece is not necessarily isomorphic to $R$).
\end{exercise}

\begin{example}
 If one fixes a homogeneous element $\phi \in \sC_{e}$ for $e > 0$, then one can form the algebra $\langle \phi \rangle = \bigoplus_{n \geq 0} \phi^{n} R^{1/p^{ne}} \subseteq \sC$ which is just the subalgebra generated by $\sC_0$ and $\phi$.
\end{example}

\begin{definition}
 Given a pair $(R, \sD)$, an ideal $I \subseteq R$ is called $\sD$-compatible if $\phi(I^{1/p^e}) \subseteq I$ for all $\phi \in \sD_e$ and all $e \geq 0$.
\end{definition}

\begin{definition} \cite[Definition 3.16]{SchwedeTestIdealsInNonQGor}
 The big test ideal $\tau(R, \sD)$ of a
pair $(R,\sD)$, if it exists, is the unique smallest ideal $J$ that satisfies two conditions:
\begin{itemize}
 \item[(1)] $J$ is $\sD$-compatible, and
 \item[(2)] $J \neq \{0\}$.
\end{itemize}
\end{definition}

\begin{exercise} \cite{SchwedeTestIdealsInNonQGor}
 Suppose that $\sD$ has a non-zero homogeneous element $\phi \in \sD_e$, $e > 0$.  Prove that $\tau(R, \sD)$ exists by using Lemma \ref{lem.TestElementsExist}.
\end{exercise}

\begin{exercise} \cite{SchwedeTestIdealsInNonQGor}
Suppose that $(R, \ba^t)$ is a pair as in Section \ref{sec.TestIdealIdealPairs}.  Prove that $\tau(R, \sC^{\ba^t}) = \tau(R, \ba^t)$.  Further show that if $\phi \in \sC_e$ is non-zero, then $\tau(R, \langle \phi \rangle) = \tau(R, \phi)$.
\end{exercise}

Algebras of maps appear very naturally.  For example, suppose that $R$ is a ring and $\sC_R$ is the complete algebra of maps on $R$.  Suppose that $I \subseteq R$ is a $\sC_R$-compatible ideal (such as the test ideal $\tau(R)$ or the splitting prime\footnote{In an $F$-pure local ring, the splitting prime is the unique largest $\sC_R$-compatible ideal not equal to the whole ring.} $\sP$ of \cite{AberbachEnescuStructureOfFPure}.).  One can then restrict each element of $\sC_R$ to $R/I$.  This yields an algebra of maps $\sD_{R/I} = \sC_R|_{R/I}$ which may or may not be equal to $\sC_{R/I}$.

\begin{exercise}\cite{SchwedeFAdjunction}
 With the notation above, suppose that $R$ is Gorenstein and local.  Prove that the algebra $\sD_{R/I}$ is equal to $\langle \phi \rangle$ for some $\phi \in \sC_{R/I}$.
\end{exercise}

One can define $F$-purity for algebras as well.

\begin{definition}
 Suppose that $(R, \sD)$ is a pair.  Then the pair is called \emph{sharply $F$-pure} (or sometimes just \emph{$F$-pure})  if there exists a homogenous element $\phi \in \sD_{e}$, $e > 0$ such that $\phi(R^{1/p^e}) = R$, \ie $\phi$ is surjective.
\end{definition}

\begin{exercise}\cite{SchwedeTestIdealsInNonQGor}, \cf \cite{BlickleBoeckleCartierModulesFiniteness, BlickleAlgebras}
 Prove that for a sharply $F$-pure pair $(R, \sD)$, $\tau(R, \sD)$ defines an $F$-pure subscheme and so in particular is a radical ideal.
\end{exercise}

The following theorem is an application of this approach of algebras of pairs.

\begin{theorem} \cite{SchwedeTestIdealsInNonQGor} (\cf \cite{SchwedeSmithLogFanoVsGloballyFRegular, deFernexHaconSingularitiesOnNormal})
Suppose that $(R, \sD)$ is a pair (\ie $\sD = \sC$).  Then \[
\tau(R, \sD) = \sum_{e > 0} \sum_{\phi \in \sD_e} \tau(R, \phi).
\]
If $R$ is additionally normal, then this also equals $\sum_{e > 0} \sum_{\phi \in \sD_e} \tau(R, \Delta_\phi)$.
\end{theorem}

For non-$\bQ$-Gorenstein normal varieties $X_0$ over $\bC$, de Fernex and Hacon have defined a multiplier ideal $\mJ(X_0)$ \cite{deFernexHaconSingularitiesOnNormal}.  Furthermore,
\[
\mJ(X_0) = \sum_{\begin{array}{c}  K_{X_0} + \Delta_0 \\ \text{is
      $\bQ$-Cartier} \end{array} } \mJ(X_0, \Delta_0).
\]
It is therefore natural to conjecture following.

\begin{conjecture}
Given a variety $X_0 = \Spec R_0$ in characteristic zero, we have
$\tau(R_p) = \mJ(X_0)_p$ for all $p \gg 0$ (here the subscript $p$
denotes reduction to characteristic $p>0$ as in Subsection \ref{subsec.ReductionToCharP}).
\end{conjecture}

Work of M. Blickle implies that this conjecture holds for toric rings, \cite{BlickleMultiplierIdealsAndModulesOnToric}.

\begin{remark}
 For other applications, it is likely important that one has a good measure of the finiteness properties of the given algebra $\sD$.  One very useful such property is the \emph{Gauge-Bounded} property introduced in \cite{BlickleAlgebras}, \cf \cite{AndersonElementaryLFunctions}.  This property is quite useful for proving questions related to the discreteness of $F$-jumping numbers, see Definition \ref{def.FJumpingNumber}.
\end{remark}

\section{Other measures of singularities in characteristic $p$}
\label{sec.OtherMeasures}

So far we have talked about test ideals, $F$-regularity and $F$-purity.  In this section we introduce several other ways to measure singularities in positive characteristic.  For a more complete list, please see the appendix.

First we introduce two other classes of singularities.  $F$-rationality and $F$-injectivity.

\subsection{$F$-rationality}

\begin{setting*}
In this subsection, unless otherwise specified, all rings are assumed to be integral domains essentially of finite type over a perfect field of characteristic $p > 0$.
\end{setting*}

\begin{definition}
Suppose that $R$ is a normal Cohen-Macaulay ring and that $\Phi_R : \omega_{R^{1/p}} = F_* \omega_R \to \omega_R$ is the canonical dual of Frobenius, see Theorem \ref{thm.TransformationRuleCanonicalModulesFiniteMaps}.  We say that $R$ has \emph{$F$-rational singularities} if there are no non-zero proper submodules $M \subseteq \omega_R$ such that $\Phi_R(F_* M) \subseteq M$.
\end{definition}

\begin{starexercise}
The hypothesis that $R$ is normal implied by the other hypotheses.  Prove it.
\end{starexercise}

\begin{exercise} \cite{FedderWatanabe}
Prove that a strongly $F$-regular ring is $F$-rational and that a Gorenstein $F$-rational ring is strongly $F$-regular.\\
\emph{Hint: } For the first part, use Theorem \ref{thm.StrongFRegEquiv} and apply the functor $\Hom_R(\blank, \omega_R)$.
\end{exercise}

\begin{starexercise}\cite{SmithFRatImpliesRat}
Recall that an integral domain $R_0$ of finite type over $\bC$ is said to have \emph{rational singularities} if for a resolution of singularities $\pi : \tld X_0 \to X_0 = \Spec R_0$, $\pi_* \omega_{\tld X_0} = \omega_{X_0}$ and $X_0$ is Cohen-Macaulay\footnote{This isn't normally the definition of rational singularities, but is instead a criterion often attributed to Kempf, \cite[Page 50]{KempfToroidalEmbeddings}.}.
Now suppose we are given an integral domain $R_0$ of finite type over $\bC$.  Show that if $R_p$ has $F$-rational singularities after reduction to characteristic $p \gg 0$ (see Subsection \ref{subsec.ReductionToCharP}), then $R_0$ has rational singularities in characteristic zero.
\end{starexercise}

\begin{remark}
The converse of the above exercise also holds, but the proof is more involved.  See \cite{HaraRatImpliesFRat} and \cite{MehtaSrinivasRatImpliesFRat}.
\end{remark}

We briefly mention the original definition of $F$-rationality.

\begin{theorem}
A local ring $(R, \bm)$ has $F$-rational singularities if and only if some ideal $I = (x_1, \dots, x_n)$ generated by a full system of parameters satisfies $I = I^*$ (here $I^*$ denotes the tight closure of $I$, see Section \ref{sec.TightClosure}).
\end{theorem}
\begin{proof}
See \cite{FedderWatanabe} or \cite[Chapter 10]{BrunsHerzog}.
\end{proof}

\subsection{$F$-injectivity}
Now we move on to $F$-injectivity.

\begin{setting*}
In this subsection, unless otherwise specified, all rings are assumed to be reduced and essentially of finite type over a perfect field of characteristic $p > 0$.
\end{setting*}

\begin{definition}
A local ring $(R, \bm)$ is called \emph{$F$-injective} if for every integer $i > 0$, the natural map $H^i_{\bm}(R) \to H^i_{\bm}(R^{1/p})$ is injective.  An arbitrary ring $R$ is called $F$-injective if all of its localizations at prime ideals are $F$-injective.
\end{definition}

\begin{exercise}\label{ex.DualFInjectiveStatement}
Suppose that $R$ is Cohen-Macaulay and local.  Prove that $R$ is $F$-injective if and only if the canonical dual to Frobenius $F_* \omega_R \to \omega_R$ is surjective.
\end{exercise}

\begin{exercise} \cite{FedderFPureRat}
Prove that a Gorenstein ring is $F$-injective if and only if it is $F$-pure and that an $F$-pure ring is always $F$-injective.
\end{exercise}

\begin{exercise}  \cite{FedderWatanabe}
\label{ex.CMFinjectivityDeforms}
Suppose that $(R, \bm)$ is a Cohen-Macaulay local ring and $f \in R$ is a regular element.  Prove that if $R/f$ is $F$-injective, then $R$ is $F$-injective.\\
\emph{Hint: } Using the criterion in Exercise \ref{ex.DualFInjectiveStatement}, consider the diagram:
\[
\xymatrix{
0 \ar[r] & F_* \omega_R \ar[d] \ar[r]^{\times F_* f^{p}} & F_* \omega_R \ar[r] \ar[d] & F_* \omega_{R/f^p} \ar[d]_{\beta} \ar[r] & 0\\
0 \ar[r] & \omega_R \ar[r]^{\times f} & \omega_R \ar[r] & F_* \omega_{R/f} \ar[r] & 0
}
\]
Show that $\beta$ surjects, by considering the map $F_* \omega_{R/f} \to F_* \omega_{R/f^p}$.  Now take the cokernels of the left and middle vertical maps and use Nakayama's lemma.
\end{exercise}

\begin{remark}
It is an open question whether Exercise \ref{ex.CMFinjectivityDeforms} holds without the Cohen-Macaulay assumption.  It is however known that the analog of Exercise \ref{ex.CMFinjectivityDeforms} does not hold for $F$-pure rings in general, see \cite{FedderFPureRat} and also \cite{SinghDeformationOfFPurity}.  One can ask the same question for strongly $F$-regular and $F$-rational singularities, and the answers are no and yes respectively; see \cite{FedderWatanabe} and \cite{SinghFregularityDoesNotDeform}.
\end{remark}

\begin{starexercise} \cite{SchwedeCentersOfFPurity, KovacsSchwedeSmithLCImpliesDuBois, KovacsSchwedeDuBoisSurvey}
\label{ex.FInjectiveImpliesDuBois}
A normal Cohen-Macaulay ring $R_0$ of finite type over $\bC$ is called Du Bois if for some (equivalently any) log resolution of singularities $\pi : \tld X_0 \to X_0 = \Spec R_0$ with simple normal crossings exceptional divisor $E_0$, $\pi_* \omega_{\tld X_0}(E_0) \cong \omega_{X_0}$.  Prove that if $X_p$ has $F$-injective singularities after reduction to characteristic $p \gg 0$, then $X_0$ has Du Bois singularities in characteristic zero.  \\
{\emph{Hint: }} Consider the diagram  \[
  \xymatrix{ (\pi_* \omega_{\tld{X_p}}(p^e E_p))^{1/p^e} \ar[d]_{\rho} \ar[r] & \pi_*
    \omega_{\tld{X_p}}(E_p) \ar[d]^{\beta} \\
    \omega_{X_p}^{1/p^e} \ar[r]^{\Phi} & \omega_{X_p} \\
  }
  \]
  where the horizontal arrows are the dual of the Frobenius map (see Subsection \ref{subsec.TestIdealsInGorRings}).
\end{starexercise}

\begin{remark}
The converse implication of the above exercise is false as stated, in fact that singularity $\bF_p[x,y,z]/(x^3 + y^3 + z^3)$ is $F$-injective if and only if $p = 1 \mod 3$.  However, it is an important open question whether a Du Bois singularity is $F$-injective after reduction to characteristic $p > 0$ for infinitely many primes $p$ (technically, a Zariski-dense set of primes), this condition is called \emph{dense $F$-injective type} whereas the original condition is called \emph{open $F$-injective type}.  Likewise, it is an open question whether a log canonical singularity is $F$-pure after reduction to characteristic $p > 0$ for infinitely many primes $p$.
\end{remark}

\begin{starexercise} \cite{HaraWatanabeFRegFPure}
\label{ex.FpureImpliesLC}
A log $\bQ$-Gorenstein pair $(X, \Delta)$ of any characteristic is called \emph{log canonical} if for every proper birational map $\pi : \tld X \to X$ with $\tld X$ normal, all the coefficients of $K_{\tld X} - \pi^* (K_{X} + \Delta)$ are $\geq -1$.  This can be checked on a single log resolution of $(X, \Delta)$, if it exists (it does in characteristic zero).  For more about log canonical singularities, see \cite{KollarSingularitiesOfPairs} and \cite{KollarMori}.

Use the method of the above exercise to show the following.
 If $X = \Spec R$ is a ring of characteristic $p$ and $\phi : R^{1/p^e} \to R$ is a divisor corresponding to $\Delta$ as in (\ref{eq:divmaplbcorr}), then if $\phi$ is surjective (\ie if $(R, \phi)$ is $F$-pure) show that $(X, \Delta)$ is log canonical.  Conclude by showing that log $\bQ$-Gorenstein pairs $(X_0, \Delta_0)$ over $\bC$ of dense $F$-pure type (\ie such that $\phi_{\Delta_p}$ is surjective for infinitely many $p \gg 0$) are always log canonical.
\end{starexercise}

\subsection{$F$-signature and $F$-splitting ratio}
Recall that a ring $R$ was said to be
$F$-pure if the all of Frobenius inclusions $R \to R^{1/p^{e}}$ split as a maps
of $R$-modules.  In this section, we consider local numerical
invariants -- the $F$-signature and \mbox{$F$-splitting ratio} -- which characterize the
asymptotic growth of the number of splittings of the iterates of
Frobenius.

\begin{setting*}
In this subsection, unless otherwise specified, all rings are assumed to be \emph{local} integral domains essentially of finite type over a perfect field of characteristic $p > 0$.
\end{setting*}

\begin{definition}
 Let $(R, \m, k)$ be a local ring.
 For each $e \in \N$, the \emph{$e$-th Frobenius splitting
   ($F$-splitting) number of $R$} is
the maximal rank $a_{e} = a_{e}(R)$ of a free $R$-module appearing in a direct sum
decomposition of $R^{1/p^{e}}$.  In other words, we may write $R^{1/p^{e}} =
R^{ \oplus a_{e}} \oplus M_{e}$ where $M_{e}$ has no free
direct summands.
\end{definition}

\begin{exercise}
  Show that $R$ is $F$-pure if and only if $a_{e} > 0$ for some $e \in
  \N$, in which case $a_{e} >0$ for all $e \in \N$.
\end{exercise}

\begin{exercise} \cite{AberbachLeuschke}
  For any prime ideal $\mathfrak{p}$ in $R$, show that
  $a_{e}(R_{\mathfrak{p}}) \geq a_{e}(R)$.
\end{exercise}

Using the following Proposition, it is easy to see that the
$F$-splitting numbers are independent of the chosen direct sum decomposition
of $R^{1/p^{e}}$.

\begin{proposition} \cite{AberbachEnescuStructureOfFPure}
\label{prop:Ie}
Assume that $k = k^{p}$ is perfect.
Consider the sets
\[
I_{e} := \{ r \in R \, | \, \phi(r^{1/p^{e}}) \in \m \mbox{ for all }
\phi \in \Hom_{R}(R^{1/p^{e}}, R) \} \; .
\]  Then $I_{e}$ is an ideal in
$R$ with $\length_{R}(R/I_{e}) = a_{e}$.
\end{proposition}

\begin{starexercise} \cite{AberbachEnescuStructureOfFPure}
  Check that $I_{e}$ is, in fact, an ideal.  Then prove the proposition.
\end{starexercise}


\begin{theorem} \cite{TuckerFSignatureExists}
\label{thm:sigexists}
   Let $(R, \m, k)$ be a local ring of dimension $d$.  Assume $k = k^{p}$
   is perfect.  Then the limit
\[
s(R) := \lim_{e \to \infty} \frac{a_{e}}{p^{ed}}
\]
exists and is called the $F$-signature of $R$.
\end{theorem}

The $F$-signature was first explicitly\footnote{Implicitly, the
  $F$-signature first appeared in \cite{SmithVanDenBerghSimplicityOfDiff}.} defined by C. Huneke and G. Leuschke
\cite{HunekeLeuschkeTwoTheoremsAboutMaximal} and captures delicate information about the
singularities of $R$.
For example, the $F$-signature of
the two-dimensional rational double-points\footnote{Here it is
  necessary to assume that $p \geq 7$ to avoid pathologies in low characteristic.} ($A_{n}$), ($D_{n}$),
($E_{6}$), ($E_{7}$), ($E_{8}$) is the reciprocal of the order of the group
defining the quotient singularity \cite[Example
18]{HunekeLeuschkeTwoTheoremsAboutMaximal}.  However, a positive answer to a conjecture of Monsky implies the existence of local rings with irrational $F$-signature, see \cite{MonskyRationalityOfHKMultCounterExampleLikely}.

The heart of the proof of Theorem~\ref{thm:sigexists} lies in the
following technical lemma.

\begin{lemma} \cite{TuckerFSignatureExists}
\label{lem:idealsand limits}
  Let $(R, \m, k)$ be a local ring of dimension $d$.  If $\{J_{e}\}_{e
     \in \N}$ is any sequence of $\m$-primary ideals such that
   $J_{e}^{[p]} \subseteq J_{e+1}$ and $\m^{[p^{e}]} \subseteq J_{e}$
   for all $e$, then $\lim_{e \to \infty} \frac{1}{p^{ed}}
   \length_{R}(R/J_{e})$ exists.
\end{lemma}

\begin{exercise} \cite{TuckerFSignatureExists}
  Show that the ideals $I_{e}$ from Proposition~\ref{prop:Ie} satisfy
  $I_{e}^{[p]} \subseteq I_{e+1}$ and $\m^{[p^{e}]} \subseteq I_{e}$,
  and use the previous lemma to conclude the existence of the
  $F$-signature limit.
\end{exercise}

It is quite natural to expect the
\mbox{$F$-signature} to measure the singularities of $R$.  Indeed,
when $R$ is regular, $R^{1/p^{e}}$ itself is a free $R$-module of rank
$p^{ed}$.  Thus, for general $R$, the \mbox{$F$-signature} asymptotically
compares the number of direct summands of $R^{1/p^{e}}$ isomorphic to $R$ with
the number of such summands one would expect from a regular local ring
of the same dimension.

\begin{theorem}
  \cite[Theorem 0.2]{AberbachLeuschke}
Let $(R, \m, k)$ be a local ring of dimension $d$.  Assume $k = k^{p}$
   is perfect.  Then $s(R) > 0$ if and only if $R$ is strongly $F$-regular.
\end{theorem}

\begin{definition} \cite{AberbachEnescuStructureOfFPure}
\label{def.SplittingPrime}
  Suppose $R$ is $F$-pure.  If  $I_{e}$ is as in Proposition \ref{prop:Ie}, the ideal $P = \cap_{e \in \N}
  I_{e}$ is called the \emph{$F$-splitting prime} of $R$.
\end{definition}

\begin{starexercise} \cite{AberbachEnescuStructureOfFPure}, \cite{SchwedeCentersOfFPurity}
  Check that $P$ is a prime ideal, and that $\phi (P^{1/p^{e}})
  \subseteq P$ for all $\phi \in \Hom_{R}(R^{1/p^{e}},R)$.  In
  particular, conclude that $\tau(R) \subseteq P$.  Show
  that $P = \langle 0 \rangle$ if and only if $R$ is strongly $F$-regular.  More generally, show that $R/P$ is strongly $F$-regular.
\end{starexercise}

\begin{theorem} \cite{BlickleSchwedeTuckerFSignaturePairs} \cite{TuckerFSignatureExists}
 Let $(R, \m, k)$ be an $F$-pure local of dimension $d$.  Assume $k = k^{p}$
   is perfect.  Let $P$ be the $F$-splitting prime of $R$. Then the limit
\[
r_{F}(R) := \lim_{e \to \infty} \frac{a_{e}}{p^{e \dim(R/P)}}
\]
exists and is called the $F$-splitting ratio of $R$.  Furthermore, we
have $r_{F}(R) > 0$.
\end{theorem}

\subsection{Hilbert-Kunz(-Monsky) multiplicity}

Our goal in this section is to explore
a variant, introduced by P. Monsky, of the Hilbert-Samuel multiplicity of a ring.    Recall that the Hilbert-Samuel multiplicity of a local
ring $(R,\m,k)$ along an $\m$-primary ideal $I$ is simply
\[
e(I) := \lim_{n \to \infty} \frac{d!}{n^{d}} \length_{R}(R/I^{n}) \,
\, .
\]
The existence of the above limit follows easily from the fact that,
for sufficiently large $n$, $\length_{R}(R/I^{n})$ agrees with a
polynomial in $n$ of degree $d$.  Since this polynomial maps $\Z \to \Z$, it is easy to see that $e(I) \in \Z$.  When $I =
\m$, $e(R) := e(\m)$ is called the Hilbert-Samuel multiplicity of $R$.

Roughly speaking, the idea behind Hilbert-Kunz multiplicity is to
use the Frobenius powers $I^{[p^{e}]}$ of an ideal $I$, see Definition \ref{def.FrobeniusPowerOfIdeal}, in place of the ordinary powers $I^{n}$
in the definition of multiplicity.  For a somewhat different introduction to the Hilbert-Kunz multiplicity, see \cite[Chapter 6]{HunekeTightClosureBook}.

\begin{setting*}
In this subsection, unless otherwise specified, all rings are assumed to be \emph{local} integral domains essentially of finite type over a perfect field of characteristic $p > 0$.
\end{setting*}

\begin{theorem} \cite{MonskyHKFunction} \cite{KunzOnNoetherianRingsOfCharP}
\label{thm:hkexists}
Suppose $(R,\m,k)$ is a local ring of dimension
$d$ and characteristic $p>0$.  If $I$ is any $\m$-primary ideal, then
the limit
\[
e_{HK}(I) := \lim_{e\to\infty}\frac{1}{p^{ed}} \length_{R}(R/I^{[p^{e}]})
\]
exists and is called the \emph{Hilbert-Kunz multiplicity of $R$ along
  $I$}.  When $I = \m$, we write $e_{HK}(R) := e_{HK}(\m)$ and we refer to this number as the \emph{Hilbert-Kunz multiplicity of $R$.}
\end{theorem}

Many basic properties of Hilbert-Kunz multiplicity (see \cite{MonskyHKFunction} or \cite{HunekeTightClosureBook}) mirror those for
Hilbert-Samuel multiplicity, such as the following:

\begin{itemize}
\item
If $I \subseteq J$ are $\m$-primary ideals, then $e_{HK}(I) \geq
e_{HK}(J)$.
\item
We always have $e_{HK}(R) \geq 1$ with equality when $R$ is regular.
\item
$e_{HK}(R) = \sum_{\mathfrak{p} \in \mathrm{Assh}(R)}
e_{HK}(R/\mathfrak{p})$
where $\mathrm{Assh}(R)$ denotes the set of prime ideals
$\mathfrak{p}$ of $R$ with $\dim(R/\mathfrak{p}) = \dim(R)$.
\item
\cite{WatanabeYoshidaHKMultAndInequality, HunekeYaoMinimalHKMultiplicity}
If $R$ is equidimensional, then $e_{HK}(R) = 1$ if and only if $R$ is
regular.
\item
If $I$ is generated by a regular sequence, then $e_{HK}(I) =
\length_{R}(R/I)$.
\end{itemize}

\begin{exercise}
  If $R$ is regular, show that
  $e_{HK}(I) = \length_{R}(R/I)$ for every $\m$-primary ideal $I$.
\end{exercise}

\begin{exercise}
\cite[Lemma 6.1]{HunekeTightClosureBook}
  Show that $e_{HK}(I) \geq \frac{1}{d!}e(I)$.  Note that this
  inequality is known to be sharp if $d \geq 2$ by \cite{HanesNotesOnHKFunction}.
\end{exercise}

Monsky's proof of the existence of Hilbert-Kunz multiplicity, however, bears little
resemblance to the proof of the existence of Hilbert-Samuel
multiplicity.  Indeed, the function $e \to \length_{R}(R/I^{[p^{e}]})$
frequently exhibits non-polynomial behavior, and $e_{HK}(R)$ need not be an integer.

\begin{example}
  Consider the characteristic 5 local ring
\[
R = \left( \mathbb{F}_{5}[w,x,y,z]/\langle w^{4} + x^{4} + y^{4} + z^{4}
  \rangle \right)_{\langle w, x, y, z \rangle} \, \, .
\]
C. Han and P. Monsky in have computed in
\cite{HanMonskySomeSurprisingHKFunctions} that
\[
\length_{R}(R/\m^{[p^{e}]}) = \frac{168}{61}(5^{e})^{3} -
\frac{107}{61} 3^{e}
\]
and, in particular, we have $e_{HK}(R) = \frac{168}{61}$.
\end{example}

In a sense, the proof of Theorem~\ref{thm:hkexists} is not
constructive: Monsky proceeds to show that $\{ \frac{1}{p^{ed}}
\length_{R}(R/I^{[p^{e}]}) \}_{e \in \N}$ is a Cauchy sequence.  As
such, the limit is only known to be a real number (and, in particular, not
necessarily even rational).  The computation of Hilbert-Kunz
multiplicity is widely considered to be a difficult problem.

\begin{exercise}
  Use Lemma~\ref{lem:idealsand limits} to show that $e_{HK}(I)$ exists when $R$ is a domain.
\end{exercise}

\begin{conjecture}
  \cite{MonskyTranscendenceOfSomeHKMultiplicities}, \cf \cite{BrennerRationalityOfHKMultiplicityInDim2} The Hilbert-Kunz multiplicity of the local ring
\[
\left(\mathbb{F}_{2}[x,y,z,u,v]/\langle uv + x^{3} + y^{3} + xyz
  \rangle \right)_{\langle x,y,z,u,v\rangle}
\]
is $\frac{4}{3} - \frac{5}{14 \sqrt{7}}$, and in particular not rational.
\end{conjecture}

\begin{example}
  Let $p > 2$ be a prime and consider the local rings
\[
R_{p,d} = \left( \overline{\mathbb{F}_{p}}[x_{0}, \ldots, x_{d}]/\langle \sum_{i=0}^{} x_{i}^{2}
  \rangle \right)_{\langle x_{0}, \ldots, x_{s} \rangle} \, \, .
\]
The numbers $e_{HK}(R_{p,d})$ have been explicitly computed
\cite{HanMonskySomeSurprisingHKFunctions} and -- even for a fixed~$d$
-- can depend on $p$ in a
complicated way.  For example, when $d = 4$ we have
\[
 \frac{29p^{2} + 15}{24 p^{2} + 12} \, \, .
\]
However, I. Gessel and P. Monsky have shown that
the $e_{HK}(R_{p,d})$ have a well-defined limit as $p \to \infty$ equal
to $1$ plus the coefficient of $z^{s}$ in the power series expansion of
$\sec z + \tan z$.

Perhaps one of the most interesting open problems aims to identify the
non-regular rings $R_{p,d}$ having the smallest Hilbert-Kunz multiplicity
possible.
\begin{conjecture}
\cite{watanabeyoshidaHK3dimlocalrings}
 Let $d \geq 1$ and $p > 2$  a prime number.  Let $R$ be $d$-dimensional
 characteristic $p$ unmixed local ring with residue field $
 \overline{\mathbb{F}_{p}}$. If $R$ is not regular, then $e_{HK}(R) \geq e_{HK}(R_{p,d})$ with
   equality if and only if it is formally isomorphic to $R_{p,d}$,
   {\itshape i.e.} their respective completions $\widehat{R} \simeq
   \widehat{R_{p,d}}$ are isomorphic.
\end{conjecture}
\noindent
This conjecture is known to be true when $R$ has dimension at most
six, see \cite{AberbachEnescuNewEstimatesForHK}, and also when $R$ is
a complete intersection in arbitrary dimension in \cite{EnescuShimomotoUpperSCofHK}.  Finally also see \cite{SinghFSignatureOfAffineSemigroup}.
\end{example}



Many topics from previous sections share a close relationship with
so-called relative Hilbert-Kunz multiplicities, {\itshape i.e.} the
differences $e_{HK}(I) - e_{HK}(J)$ for pairs of $\m$-primary
ideals $I \subseteq J$.  For example, as seen below, these differences may be used to
test for tight closure and are closely related to the $F$-signature.

\begin{theorem}
\cite[Theorem 8.17]{HochsterHunekeTC1}
Assume $R$ is a complete local domain.  If  $I \subseteq J$
are two $\m$-primary ideals, then $e_{HK}(I) = e_{HK}(J)$ if and only
if $I^{*} = J^{*}$.
\end{theorem}
\begin{proof}
The ($\Leftarrow$) direction is not difficult and is left as an exercise to the reader, along with the following hint:\\
\emph{Hint: } First show that there exists a $c \in R^{\circ}$ such that $c J^{[p^e]} \subseteq I^{[p^e]}$ for all $q \gg 0$.  Set $S = R/\langle c \rangle$ and consider its Hilbert-Kunz multiplicity with respect to $I$.  Finally show that there exists an integer $k$ such that $(S/(IS)^{[p^e]})^{\oplus k}$ can be mapped onto $J^{[p^e]} / I^{[p^e]}$ for all $e \geq 0$.
\end{proof}

\begin{theorem}
\cite[Proposition 15]{HunekeLeuschkeTwoTheoremsAboutMaximal}
If $R$ is a ring, then for any two
$\m$-primary ideals $I \subseteq J$
\begin{equation}
\label{eq:signaturerelativeHKinequality}
s(R) \leq \frac{e_{HK}(I) - e_{HK}(J)}{\length_{R}(J/I)} \, \, .
\end{equation}
\end{theorem}

\begin{question}
  Can one always find $\m$-primary ideals $I \subseteq J$ such that equality holds in
\eqref{eq:signaturerelativeHKinequality}?
\end{question}

\noindent
In many cases, such as when $R$ is $\Q$-Gorenstein, the above question
has a positive answer.  More generally, an affirmative response
would immediately imply Conjecture~\ref{conj:strong=weak} \cite{WatanabeYoshidaMinimal}.

\subsection{$F$-ideals, $F$-stable submodules, and $F$-pure centers}

Historically in commutative algebra, Frobenius has been used heavily to study local cohomology.
In particular, if $(R, \bm)$ is a local ring, the map $H^i_{\bm}(R) \to H^i_{\bm}(R^{1/p}) \cong H^i_{\bm}(R)$ is called the action of Frobenius on the local cohomology module $H^i_{\bm}(R)$ and denoted by $F$ (more generally, one also has a similar action on $H^i_{J}(R)$ for any ideal $J \subseteq R$).  Of course, one can iterate $F$, $e$-times, and obtain higher Frobenius actions $F^e : H^i_{\bm}(R) \to H^i_{\bm}(R)$.

\begin{remark}
If one is willing to use \Cech cohomology to write down specific elements of $H^i_J(R)$, then the Frobenius action can be understood as raising those elements to their $p$th power.  See \cite{SmithTestIdeals} for more details.
\end{remark}

Fix $F : H^i_{\bm}(R) \to H^i_{\bm}(R)$ consider now the following ascending chain of submodules.
\[
\ker F \subseteq \ker F^2 \subseteq \ker F^3 \subseteq \dots
\]
In \cite{HartshorneSpeiserLocalCohomologyInCharacteristicP} (also see \cite{LyubeznikFModulesApplicationsToLocalCohomology}, \cite{Gabber.tStruc} and \cite{BlickleBoeckleCartierModulesFiniteness}) it was shown that this ascending chain eventually stabilizes, even though the module in question is Artinian, and not generally Noetherian.  Set $N$ to be that stable submodule.
It is obvious that $F(N) \subseteq N$.  On the other hand, Karen Smith observed that $0^*_{H^{\dim R}_{\bm}(R)}$ is the unique largest submodule $M \subseteq H^{\dim R}_{\bm}(R)$, with non-zero annihilator such that $F(M) \subseteq M$, see \cite{SmithFRatImpliesRat}.  Motivated by this, she made the following definition:

\begin{definition}\cite{SmithTestIdeals}
An ideal $I \subseteq R$ is called an \emph{$F$-ideal} if $M_I = \Ann_{H^d_{\bm}(R)} I$ satisfies the condition $F(M_I) \subseteq M_I$.
\end{definition}

Suppose that $(R, \bm)$ is Gorenstein.  Then as in Subsection \ref{subsec.TestIdealsInGorRings}, we have a map $\Phi_R : R^{1/p} \to R$.  The Matlis dual of this map is $F : H^d_{\bm}(R) \to H^d_{\bm}(R)$ by local duality, see Theorem \ref{thm.LocalDuality}.

\begin{exercise}
Still assuming that $R$ is Gorenstein, prove that $I$ is an $F$-ideal if and only if $\Phi_R(I^{1/p}) \subseteq I$.
\end{exercise}

It turns out that this notion is very closely related to \emph{log canonical centers} in characteristic zero.  Motivated by this connection we define the following.

\begin{definition}\cite{SchwedeCentersOfFPurity} \cite{BrionKumarFrobeniusSplitting}
A prime ideal $Q \in \Spec R$ is called an \emph{$F$-pure center} if for every $e > 0$ and every $\phi \in \Hom_R(R^{1/p^e}, R)$, one has $\phi(Q^{1/p^e}) \subseteq Q$.  It is very common to also assume that $R_Q$ is $F$-pure.

More generally, given some fixed $\phi \in \Hom_R(R^{1/p^e}, R)$, an ideal $Q \in \Spec R$ is called an \emph{$F$-pure center of $(R, \phi)$}, if $\phi(Q^{1/p^e}) \subseteq Q$.
If additionally $\phi$ is a Frobenius splitting, then $Q$ (or the variety it defines) is called \emph{compatibly $\phi$-split}.
\end{definition}

\begin{exercise}
Suppose that $R$ is $F$-pure and $Q \in \Spec R$ is an \emph{$F$-pure center}.  Show that $R/Q$ is also $F$-pure.  Compare with \cite{KawamataSubadjunctionOne, AmbroSeminormalLocus, KollarKovacsLCImpliesDB} keeping in mind that $F$-pure singularities are closely related to log canonical singularities \cf Exercise \ref{ex.FpureImpliesLC}.
\end{exercise}

\begin{exercise} \cite{SchwedeCentersOfFPurity}, \cf \cite{AberbachEnescuStructureOfFPure}
Suppose that $R$ is $F$-pure and $Q \in \Spec R$ is an \emph{$F$-pure center} which is maximal with respect to inclusion.  Prove that $R/Q$ is strongly $F$-regular.  Compare with \cite{KawamataSubadjunction2}.  If $R$ is local, prove further that $Q$ is the splitting prime, Definition \ref{def.SplittingPrime}.
\end{exercise}

The structure of Frobenius stable submodules (and their annihilators) has been an important object of study in commutative algebra for several decades.  In particular, several questions about their finiteness have been asked, and also answered, see for example \cite{EnescuFInjectiveRingsAndFStablePrimes}, \cite{EnescuHochsterTheFrobeniusStructureOfLocalCohomology} and \cite{SharpGradedAnnihilatorsOfModulesOverTheFrobeniusSkewPolynomialRing}.  These questions are closely related to the finiteness of $F$-pure centers or compatibly $\phi$-split ideals.  See \cite{SchwedeFAdjunction, KumarMehtaFiniteness} for answers to this question and see \cite{SchwedeTuckerNumberOfFSplit, BlickleBoeckleCartierModulesFiniteness} for generalizations.

\begin{starexercise} \cite{SchwedeCentersOfFPurity}
 Suppose that $(X, \Delta)$ is a log $\bQ$-Gorenstein pair of any characteristic, see Definition \ref{def.logQGorenstein}.  Then a subscheme $Z \subseteq X$ is called a \emph{log canonical center} if there exists a proper birational map $\pi : \tld X \to X$ with $\tld X$ normal and a prime divisor $E$ on $\tld X$ such that $\pi(E) = Z$ and also such that the coefficient of $E$ in $K_{\tld X} - \pi^*(K_X + \Delta)$ is $-1$.

 Suppose now that $(X = \Spec R, \Delta)$ is a log $\bQ$-Gorenstein pair of characteristic $p > 0$ and that $\Delta = \Delta_{\phi}$ for some $\phi : R^{1/p^e} \to R^{1/p^e}$ as in (\ref{eq:divmaplbcorr}).  Show that every log canonical center of $(X, \Delta)$ is an $F$-pure center of $(R, \phi)$.
\end{starexercise}

\appendix \section{Canonical modules and duality}
\label{sec.appendixDuality}
{

\subsection{Canonical modules, Cohen-Macaulay and Gorenstein rings}

Throughout this section, we restrict ourselves to rings of finite type over a field $k$.  The generalization of the material in this section to rings of essentially finite type is obtained via localization, and so will be left to the reader as an exercise.  In particular, we assume that $R = k[x_1, \dots, x_n]/I = S/I$.

All the material in this section can be found in \cite{BrunsHerzog} or \cite{HartshorneResidues}.

\begin{definition}
Suppose that $R$ is as above and additionally that $R$ is equidimensional of dimension $d$.  Then we define the \emph{canonical module}, $\omega_R$ of $R$, to be the $R$-module
\[
\omega_R := \Ext^{n - d}_S(R, S).
\]
\end{definition}

We state several facts about canonical modules for the convenience of the reader.  Please see \cite{HartshorneLocalCohomology}, \cite{HartshorneResidues} or \cite{BrunsHerzog} for details and generalizations.

\begin{itemize}
\item[(1)]  If $R$ is a normal domain, then $\omega_R$ is isomorphic to an unmixed ideal of height one in $R$.  In particular, it can be identified with a divisor on $\Spec R$.  Any such divisor is called a \emph{canonical divisor}, see also Appendix \ref{sec.divisors}.
\item[(2)]  If $R = S/(f)$, then it is easy to check that $\omega_R \cong R = S/(f)$.  (Write down the long exact sequence computing $\Ext$).
\item[(3)]  The canonical module as defined seems to depend on the choice of generators and relations (geometrically speaking, it depends on the embedding).  In the context we are working in, $\omega_R$ is in fact unique up to isomorphism.  In greater generality, the canonical module is only unique up to tensoring with locally-free rank-one $R$-module.
\end{itemize}

\begin{definition}
With $R$ equidimensional of dimension $d$, we say that $R$ is \emph{Cohen-Macaulay} if $\Ext^i_S(R, S) = 0$ for all $i \neq n - d$.  We say that $R$ is \emph{Gorenstein} if $R$ is Cohen-Macaulay and additionally if for each maximal ideal $\bm \in \Spec R$ we have that $(\omega_R)_{\bm} \cong R_{\bm}$ (abstractly).
\end{definition}

For the reference of the reader we also recall some facts about Cohen-Macaulay and Gorenstein rings.

\begin{itemize}
\item[(i)]  In a Cohen-Macaulay ring, $\omega_R$ has finite injective dimension and so in a Gorenstein ring, $R$ has finite injective dimension.
\item[(ii)]  If $R = S/(f)$ then $R$ is Cohen-Macaulay and Gorenstein (this also holds if $R$ is a complete intersection).
\item[(iii)]  A regular ring is always Gorenstein, and in particular, it is Cohen-Macaulay.
\item[(iv)]  A local ring $(R, \bm)$ is Cohen-Macaulay if and only if $H^i_{\bm}(R) = 0$ for all $0 \leq i < \dim R$.
\item[(v)]  If $g \in R$ is a regular element and $R$ is local, then
  $R/\langle g \rangle$ is Gorenstein (respectively Cohen-Macaulay) if and only if $R$ is Gorenstein (respectively Cohen-Macaulay).\footnote{Notice that (v) + (iii) also implies (ii).}
\end{itemize}

Finally, we include one more definition.

\begin{definition}
A normal ring $R$ is called \emph{$\bQ$-Gorenstein} if the canonical module $\omega_R$, when viewed as a height-one fractional ideal $\omega_R \subseteq K(R)$, has a symbolic power $\omega_R^{(n)}$ which is locally free (for some $n > 0$).

Equivalently, after viewing $\omega_R \subseteq R$ as a fractional ideal, one may associate a divisor $K_R$ on $\Spec R$.  The symbolic power statement then is the same as saying that $nK_R$ is Cartier.
\end{definition}

\subsection{Duality}

In this section we discuss duality and transformation rules for canonical modules.  First we recall Matlis duality and the surrounding definitions.  Throughout this section, we restrict ourselves to rings essentially of finite type over a field.

\begin{definition}[Injective hull] \cite{BrunsHerzog, HartshorneLocalCohomology, HartshorneResidues, BrodmannSharpLocalCohomology}
Suppose that $(R, \bm)$ is a local ring.  An \emph{injective hull $E$ of the residue field $R/\bm$} is a an injective $R$-module $E \supseteq k$ that satisfies the following property:
\begin{itemize}
\item{}  For any non-zero submodule $U \subseteq E$ we have $U \cap k \neq \{ 0 \}$.
\end{itemize}
\end{definition}

Injective hulls of the residue field are unique up to non-unique isomorphism.  They are in a very precise sense the smallest injective module containing $E$, see \cite[Proposition 3.2.2]{BrunsHerzog} for additional discussion.

\begin{theorem}[Matlis Duality] \cite{BrunsHerzog, HartshorneLocalCohomology, HartshorneResidues, BrodmannSharpLocalCohomology}
 Suppose that $(R, \bm)$ is a local ring and that $E$ is the injective hull of the residue field.  Then the functor $\Hom_R(\blank, E)$ is a faithful exact functor on the category of Noetherian $R$-modules.  More-over, applying this functor twice is naturally isomorphic to the $\blank \tensor_R \hat{R}$ functor where $\hat{R}$ is the completion of $R$-along $\bm$.

Additionally and in particular, if $R$ is already complete, then $\Hom_R(\blank, E)$ induces an equivalence of categories between Artinian $R$-modules and Noetherian $R$-modules (and visa-versa).
\end{theorem}

Now we state a special case of \emph{local-duality}.

\begin{theorem}[Local Duality] \cite[Chapter V, Section 6]{HartshorneResidues}, \cite{HartshorneLocalCohomology}, \cite{BrodmannSharpLocalCohomology}
\label{thm.LocalDuality}
 Suppose that $(R, \bm)$ is a local ring and that $E$ is the injective hull of the residue field. Then for any finitely generated $R$-module $M$:
\[
 \Hom_R( \Hom_R(M, \omega_R), E) \cong H^{\dim R}_{\bm}(M).
\]
In particular $H^{\dim_R}_{\bm}(\omega_R) \cong E$ so that if $R$ is Gorenstein then $H^{\dim R}_{\bm}(R) \cong E$.

Furthermore, if $R$ is Cohen-Macaulay, then we have a natural isomorphism
\[
 \Hom_R( \Ext_R^i(M, \omega_R), E) \cong H^{\dim R - i}_{\bm}(M)
\]
for all $i \geq 0$.
\end{theorem}

\begin{remark}
 If one is willing to work in the bounded derived category with finitely generated cohomology $D^{b}_{\coherent}(R)$, then one obtains the more general statement (without any Cohen-Macaulay hypothesis):
\[
\Hom_R( \myR \Hom_R^{\mydot}(M, \omega_R^{\mydot}), E) \qis \myR \Gamma_{\bm}(M).
\]
where $\omega_R^{\mydot}$ is the dualizing complex of $R$.
\end{remark}

Finally, we remark on the following transformation rule for the canonical module
\begin{theorem}
\label{thm.TransformationRuleCanonicalModulesFiniteMaps}
 Suppose that $R \subseteq S$ is a finite extension of normal rings essentially of finite type over a field $k$.  Then
\[
 \Hom_R(S, \omega_R) \cong \omega_S.
\]
\end{theorem}
\begin{proof}
 This is contained in for example \cite[Theorem 3.3.7(b)]{BrunsHerzog} for Cohen-Macaulay local rings and the statement holds more generally for Cohen-Macaulay schemes.  In particular, both modules are automatically isomorphic (with a natural isomorphism) on the Cohen-Macaulay-locus.  But both modules are reflexive, and thus since the non-Cohen-Macaulay locus is of codimension at least 2, the modules are isomorphic.
\end{proof}

\begin{remark}
We will be applying Theorem \ref{thm.TransformationRuleCanonicalModulesFiniteMaps} to the case of the inclusion $R \subseteq R^{1/p}$.  While $R$ is finite type over $k$, $R^{1/p}$ is of finite type over $k^{1/p}$.  If $k$ is perfect, then $k^{1/p} = k$ and the inclusion $R \subseteq R^{1/p}$ can be interpreted as being $k$-linear (although with possibly different choices of generators and relations for $R^{1/p}$ over $k$).  If $k$ is not perfect, then $R^{1/p}$ need not be finite type over $k$, but it is if $[k^{1/p} : k] < \infty$ and again in this case the generators and relations for $R^{1/p}$ over $k$ maybe different than those over $k^{1/p}$.

However, as long as $[k^{1/p} : k] < \infty$, then it can be shown that $\omega_{R^{1/p}} \cong (\omega_R)^{1/p}$ or in other words that $\omega_{F_* R} \cong F_* \omega_R$.  Also see the discussion around condition ($\dagger$) in \cite[Page 921]{BlickleSchwedeTakagiZhang}.
\end{remark}

\begin{remark}
 If one is willing to work in the derived category $D^{b}_{\coherent}(R)$, then Theorem \ref{thm.TransformationRuleCanonicalModulesFiniteMaps} should be viewed as a generalization of the following special case of duality for a finite morphism where $M \in D^b_{\coherent}(S)$:
\[
 \myR \Hom_R^{\mydot}(M, \omega_R^{\mydot}) \cong \myR \Hom_S^{\mydot}(M, \omega_S^{\mydot}).
\]
Simply take $M = S$.
\end{remark}

In fact, there is the following generalization of the above remark.

\begin{theorem} [Grothendieck Duality]\cite{HartshorneResidues}
\label{thm.GrothendieckDuality}
Suppose that $f : Y \to X$ is a proper morphism of varieties over a field $k$.  Then $\omega_X^{\mydot}$ and $\omega_Y^{\mydot}$ exist and furthermore, for any coherent sheaf $\sM$ (or more generally object of $D^{b}_{\coherent}(X)$), we have a functorial isomorphism
\[
\myR \sHom_{\O_X}^{\mydot}( \myR f_* \sM, \omega_X^{\mydot}) \qis \myR f_* \myR \sHom_{\O_Y}^{\mydot}( \sM, \omega_Y^{\mydot})
\]
in $D^{b}_{\coherent}(X)$.
In particular, if we set $\sM = \O_Y$, we have an isomorphism:
\[
\myR \sHom_{\O_X}^{\mydot}( \myR f_* \O_Y, \omega_X^{\mydot}) \qis \myR f_* \omega_Y^{\mydot}.
\]
\end{theorem}

\begin{corollary}
\label{cor.FunctorialityOfOmega}
Suppose that $\pi : Y \to X$ is a proper morphism of varieties of the same dimension over a field $k$.  Then we have a natural map
\[
\pi_* \omega_Y \to \omega_X.
\]
\end{corollary}
\begin{proof}
Consider the natural map $\O_X \to \myR \pi_* \O_Y$ and apply the contra-variant Grothendieck-duality functor $\myR \sHom_{\O_X}(\blank, \omega_X^{\mydot})$, use the second half of Theorem \ref{thm.GrothendieckDuality} and then take cohomology (the fact that $\myH^{-\dim Y} \myR f_* \omega_Y^{\mydot} \cong \pi_* \omega_Y$ follows by analyzing the associated spectral sequence).
\end{proof}

\section{Divisors}
\label{sec.divisors}

In this section, we review divisors on normal algebraic
varieties.  Divisors are sometimes a stumbling block for commutative
algebraists trying to apply the techniques of algebraic geometry.  As such,
this appendix is designed to serve as a reference for divisors
for those already familiar with commutative algebra.
For those more geometrically inclined, please read \cite[Chapter II, Section 6]{Hartshorne} or \cite{HartshorneGeneralizedDivisorsOnGorensteinSchemes}.

We treat divisors here only in the case of an affine variety and
describe them using symbolic powers.  The generalization to non-affine
varieties is left to the reader.  Of course within the broader field of algebraic geometry, the formalism of divisors
is most useful in the study of projective (non-affine) varieties.

\begin{definition}
Suppose that $R$ is a normal domain of finite type over a field, and
let $X = \Spec R$ be the corresponding normal affine algebraic
variety.  Then a \emph{prime divisor} on $X$ is a  codimension 1
subvariety of $X$, and a \emph{Weil divisor on $X$} is a formal
$\mathbb{Z}$-linear combination of prime divisors.  In other words,
a Weil divisor is an element of the free Abelian group on the set of all
prime divisors.
\end{definition}

\begin{remark}
A prime divisor is exactly the same data as a height one prime ideal
$P \subseteq R$.  More generally, a Weil divisor can be viewed as the
combined data of a finite set of height-one prime ideals with formal
coefficients $n_i$.  In other words $D = \sum n_i V(P_i)$ where the
$P_i$ are prime ideals and $V(P_i) = \{ Q \in \Spec R \, | \, P_{i}
\subseteq Q \} = \Spec(R/P_{i}) \subseteq X$ is the vanishing locus of $P_i$.
\end{remark}

Given a prime $P \in \Spec R$, and an integer $n > 0$, we use $P^{(n)}
= (P^{n}R_{P}) \cap R$ to denote the $n$-th symbolic power of $P$.  If
$n = 0$, then $P^{(n)} = R$.  If $n < 0$ then $P^{(n)}$ is the
fractional ideal which is the inverse to $P^{(|n|)}$.  Explicitly,
$P^{(n)} = \{ \, x \in K(R) \,  | \, xP^{(|n|)} \subseteq R \, \}$.

\begin{definition}
Given a divisor $D = \sum n_i V(P_i)$ on an algebraic variety $X =
\Spec R$, the sheaf $\O_X(D)$ is simply the coherent sheaf of
$\O_{X}$-modules associated with the
fractional ideal $\bigcap_i P_i^{(-n_i)}$. In particular, note also
that $\O_X(-D)$ is determined by $\bigcap_i P_i^{(n_i)}$.
\end{definition}

Because the category of coherent sheaves of $\O_X$-modules is equivalent to the category of finitely generated $R$-modules, in what follows we will treat $\O_X(D)$ as if it was a fractional ideal and not a sheaf.

\begin{remark}
 A key property of $\O_X(D) = \bigcap_i P_i^{(-n_i)}$ is that it is S2 as an $R$-module.  Because it has full dimension and $R$ is normal, this means that it is also reflexive as an $R$-module (in other words, applying the functor $\Hom_R( \blank, R)$ twice yields an isomorphic module).
Therefore if $\O_X(D) = \bigcap_i P_i^{(-n_i)}$ and $\O_X(E) = \bigcap_i P_i^{(-m_i)}$ then $\O_X(D + E) = \bigcap_i P_i^{(-n_i - m_i)}$ is the largest ideal that agrees with $\O_X(D) \cdot \O_X(E)$ at all the height-one-primes of $R$.  See \cite{HartshorneGeneralizedDivisorsOnGorensteinSchemes} for additional discussion.
\end{remark}

\begin{definition}
\label{def.PropertiesOfWeilDivisors}
 We now list some common properties/prefixes associated to Weil divisors.  Suppose that $D = \sum n_i V(P_i)$ is a Weil divisor on $X = \Spec R$.
\begin{itemize}
 \item[(1)]  $D$ is called \emph{effective} if all of the $n_i$'s are non-negative.
 \item[(2)]  $D$ is called \emph{Cartier} if for every maximal (equivalently prime) ideal $\bm \in \Spec R$, $( \O_X(D))_{\bm}$ is a principal ideal.  In other words, if $\bigcap_i P_i^{(-n_i)}$ is locally principal.
 \item[(3)]  $D$ is called \emph{reduced} if all of the $n_i$'s are equal to $1$.
 \item[(4)]  $D$ is called \emph{$\bQ$-Cartier} if there exists an integer $n > 0$ such that $nD$ is Cartier.  Equivalently, this means $\left( \bigcap_i P_i^{(-n_i)} \right)^{(n)}$ is locally principal.  The \emph{index} of a $\bQ$-Cartier divisor is the smallest such $n$.
 \item[(5)]  $D$ is called a \emph{canonical divisor} if $\O_X(D)$ is (abstractly) isomorphic to a canonical module of $R$.
 \item[(6)]  A reduced divisor is said to have \emph{normal crossings} if it is Cartier, and for each $\bq \in \Spec R$ containing the ideal $\O_X(-D)$, $R_{\bq}$ is regular and $(\O_X(-D))_{\bq} = (\bigcap_i P_i^{(n_i)})_{\bq}$ is an ideal generated by a product of minimal generators of the maximal ideal of $R_P$.  If each $R/P_i$ is a regular ring, we then say that $D$ has \emph{simple normal crossings}.
 \item[(7)]  Two divisors $D$ and $E$ are said to be \emph{linearly equivalent} if $\O_X(D)$ is abstractly isomorphic to $\O_X(E)$.  In that case, we write $D \sim E$.
 \item[(8)]  The non-zero elements $x \in \Gamma(X, \O_X(D))$ are in bijective correspondence with effective divisors linearly equivalent to $D$.
\end{itemize}
\end{definition}

\begin{exercise}
 Prove that every divisor in a regular ring is Cartier and that every pair of Cartier divisors in a regular local ring are linearly equivalent.
\end{exercise}

\begin{exercise}
 Prove that $V( \langle x, y \rangle )$ on $\Spec k[x,y,z]/(x^2 - yz)$ is not Cartier but is $\bQ$-Cartier.
\end{exercise}

Finally we also describe $\mathbb{Q}$-divisors.

\begin{definition}
 A \emph{$\mathbb{Q}$-divisor} is a formal sum $\sum_i n_i D_i = \sum_i n_i V(P_i)$ of prime divisors with rational coefficients $n_i \in \mathbb{Q}$.  The set of $\bQ$-divisors also form a group under addition.
\end{definition}

\begin{remark}
 It is also very natural to define $\bR$-divisors.  We won't do that here however.
\end{remark}

\begin{definition}
 We now state some common terminology with $\bQ$-divisors.

Fix a $\mathbb{Q}$-divisor $\Delta = \sum n_i D_i$
\begin{itemize}
 \item[(i)]  $\Delta$ is called \emph{effective} if all of the $n_i$'s are non-negative.
 \item[(ii)]  We define $\lceil \Delta \rceil = \sum_i \lceil n_i \rceil D_i$, likewise $\lfloor \Delta \rfloor = \sum_i \lfloor n_i \rfloor D_i$.
 \item[(iii)]  When dealing with a $\bQ$-divisor $D$, we say that $D$ is \emph{integral} if $D$ is simultaneously a Weil-divisor and a $\bQ$-divisor (in other words, if all the $n_i$ are integers).
 \item[(iv)]  We say that $\Delta$ is a $\bQ$-Cartier divisor if there exists an integer $n > 0$ such that $n \Delta$ is an integral divisor and a Cartier divisor.  The \emph{index} of a $\bQ$-Cartier $\bQ$-divisor is the smallest such $n$.
 \item[(v)]  We say that two $\bQ$-divisors $\Delta_1$ and $\Delta_2$ are $\bQ$-linearly equivalent if there exists an integer $n > 0$ such that $n \Delta_1$ and $n \Delta_2$ are linearly equivalent integral Weil-divisors.  In this case we write $\Delta_1 \sim_{\bQ} \Delta_2$.
\end{itemize}
\end{definition}

\section{Glossary and diagrams on types of singularities}

{\small

We collect the various measures of and types of singularities in characteristic $p$.  Most of these are mentioned in the paper.
First we display a diagram explaining the relationship between the various singularity classes in characteristic zero and characteristic $p > 0$.
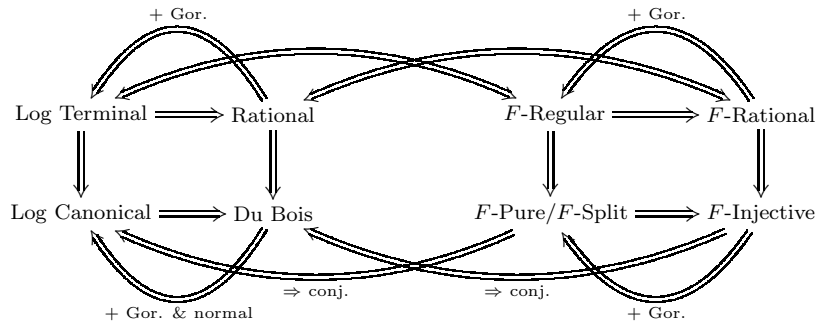
\[
\scriptsize
\xymatrix{
\text{Log Terminal} \ar@/^2.7pc/@{<=}[r]^-{\text{+ Gor.}} \ar@/^2pc/@{<=>}[rrr] \ar@{=>}[r] \ar@{=>}[d]& \text{Rational} \ar@{=>}[d]
\ar@/^2pc/@{<=>}[rrr] & & \text{
$F$-Regular} \ar@/^2.7pc/@{<=}[r]^-{\text{+ Gor.}} \ar@{=>}[r] \ar@{=>}[d] & \text{$F$-Rational} \ar@{=>}[d]\\
\text{Log Canonical} \ar@2{=>}[r] \ar@/_2pc/@{<=}[rrr]_{\text{$\Rightarrow$ conj.}} \ar@/_2.7pc/@{<=}[r]_-{\text{+ Gor. \& normal}} & \text{Du Bois}
\ar@/_2pc/@{<=}[rrr]_{\text{$\Rightarrow$ conj.}} & & \text{$F$-Pure/$F$-Split}
\ar@{=>}[r] \ar@/_2.7pc/@{<=}[r]_-{\text{+ Gor.}} & \text{$F$-Injective}\\
& & & & \\
}
\]
The left-square is classes of singularities in characteristic zero.  The right side is classes of singularities in characteristic $p > 0$.  An arrow $A \Rightarrow B$ between two classes of singularities means that all singularities of type $A$ are also of type $B$.  For example, the arrow $\text{Log terminal} \Rightarrow \text{Rational}$ means that all log terminal singularities are rational.

The connecting arrows between the two squares are via reduction to characteristic $p$, see \cite{SmithFRatImpliesRat,HaraRatImpliesFRat,MehtaSrinivasRatImpliesFRat,HaraWatanabeFRegFPure,SchwedeFInjectiveAreDuBois}.  There has been some progress on the conjectural direction between $F$-pure and log canonical singularities, \cite{HaraDimensionTwo, ShibutaTakagiLCThresholds, HernandezLCvsFPurity,MustataSrinivasOrdinary,MustataOrdinary2,TakagiAdjointIdealsAndACorrespondence}.

It should be noted that some of the implications on the characteristic zero side are highly non-trivial, see \cite{ElkikRationalityOfCanonicalSings,KovacsRat,SaitoMixedHodge,KollarKovacsLCImpliesDB}.
}

\subsection{Glossary of terms}
\label{subsec.Glossary}
{\small
\begin{flushdesc}
 \item[(Big) test ideal]  Given an $F$-finite reduced ring $R$ of characteristic $p$, the \emph{(big) test ideal} $\tau(R) = \tau_b(R)$ is defined to be the smallest ideal $I \subseteq R$, not contained in any minimal prime, such that $\phi(I^{1/p^e}) \subseteq I$ for all $e \geq 0$ and all $\phi \in \Hom_R(R^{1/p^e}, R)$.  It also coincides with $\bigcap_{M} (0 :_R 0^*_M)$ where $M$ runs over all $R$-modules.  The big test ideal is closely related to the multiplier ideal in characteristic zero, see \cite{SmithMultiplierTestIdeals,HaraInterpretation} for the original statements and see  \cite{HaraYoshidaGeneralizationOfTightClosure,TakagiInterpretationOfMultiplierIdeals} for generalizations to pairs.

\item[(Big/finitistic) test element]  An element $r \in R$ which is not contained in any minimal prime of $R$ is called a \emph{finitistic test element} if $z \in I^*$ implies that $cz^{p^e} \in I^{[p^e]}$ for all $e \geq 0$.  It is called a \emph{big test element} if $z \in 0^*_M$ implies that $0 = c^{1/p^e} \tensor z \in R^{1/p^e} \tensor M$ for all $R$-modules $M$.  It is an open question whether these two definitions are equivalent.  See \cite{HochsterHunekeTC1}, \cite{LyubeznikSmithCommutationOfTestIdealWithLocalization} and \cite{HochsterFoundations}.

\item[Completely stable (finitistic) test element]  An element $r$ in a local ring $R$, such that $r$ is not contained in any minimal prime of $R$, is called a \emph{completely stable test (finitistic) element} if it is a finitistic test element and if it remains a finitistic test element after both localization and completion.

 \item[Dense $F$-XXX type]  For a given class of singularities ``$F$-XXX'' in characteristic $p$, a characteristic zero scheme $X$ (or pair as appropriate) is said to have \emph{dense $F$-XXX type} if for all sufficiently large finitely generated $\bZ$-algebras $A$ and families $X_A \to \Spec A$ of characteristic $p$-modules of $X$ (in particular, the generic point of that family agrees with $X$ up to field-base-change), there exists a Zariski-dense set of maximal ideals $\bq \in \Spec A$ such that the fiber $X_{\bq}$ has $F$-XXX singularities.

 \item[Divisorially $F$-regular]  The combined information of a normal variety $X$ in characteristic $p > 0$ and an effective reduced divisor $D \subseteq X$ is called \emph{divisorially $F$-regular} if for every $c \in R$, not vanishing on any component of $D$, there is an $R$-linear map $\phi : R^{1/p^e} \to R$, for some $e > 0$, which sends $c^{1/p^e}$ to $1$.  Divisorially $F$-regular pairs were introduced in \cite{HaraWatanabeFRegFPure}.  Unfortunately for the terminology, divisorially $F$-regular pairs correspond to \emph{purely log terminal} singularities in characteristic zero, see \cite{TakagiPLTAdjoint}.

 \item[$F$-finite]  A reduced ring of characteristic $p > 0$ is called \emph{$F$-finite} if $R^{1/p}$ is a finite $R$-module.  Every ring essentially of finite type over a perfect field is $F$-finite.  See \cite{KunzOnNoetherianRingsOfCharP}.

 \item[$F$-injective]  A reduced $F$-finite ring of characteristic $p > 0$ is called \emph{$F$-injective} if the natural map $H^i_{\bm}(R) \to H^i_{\bm}(R^{1/p})$ is injective for every $i \geq 0$ and every maximal ideal $\bm \subseteq R$.  See \cite{FedderFPureRat}.  $F$-injective rings are closely related to rings with \emph{Du Bois singularities} in characteristic zero, \cite{SchwedeFInjectiveAreDuBois}.

 \item[$F$-ideal]  Suppose that $(R, \bm)$ is a local ring.  An ideal $I \subseteq R$ is called an \emph{$F$-ideal} if $M_I := \Ann_{H^d_{\bm}(R)} I$ satisfies the condition $F(M_I) = M_I$, see \cite{SmithTestIdeals}.

\item[$F$-jumping number]  Suppose that $R$ is an $F$-finite reduced ring and $\ba \subseteq R$ is an ideal.  The \emph{$F$-jumping numbers} of the pair $(R, \ba)$ are the real numbers $t \geq 0$ such that $\tau(R, \ba^t) \neq \tau(R, \ba^{t - \varepsilon})$ for all $\varepsilon > 0$.

\item[$F$-pure]  A reduced ring of characteristic $p > 0$ is called \emph{$F$-pure} if for every $R$-module $M$, the map $M \to M \tensor_R R^{1/p}$ is injective.  If $R$ is finite type over a perfect field, then this is equivalent to the condition that $R \to R^{1/p}$ splits as a map of $R$-modules.  $F$-purity was first introduced by Hochster and J. Roberts in \cite{HochsterRobertsFrobeniusLocalCohomology}.  $F$-pure rings are closely related to rings with \emph{log canonical singularities} in characteristic zero, \cite{HaraWatanabeFRegFPure}.  See \cite{HaraWatanabeFRegFPure, TakagiInversion,SchwedeTestIdealsInNonQGor} for generalizations to pairs/triples.

\item[$F$-pure center (a.k.a. center of $F$-purity)]  Suppose that $R$ is an $F$-finite ring.  A prime ideal $Q \in \Spec R$ is called an \emph{$F$-pure center} if for every $e > 0$ and every $\phi \in \Hom_R(R^{1/p^e}, R)$, one has $\phi(Q^{1/p^e}) \subseteq Q$.  It is very common to also assume that $R_Q$ is $F$-pure.  See \cite{SchwedeCentersOfFPurity}.

\item[$F$-pure threshold]  For a given pair $(R, \ba)$, the \emph{$F$-pure threshold of $(R, \ba)$} is the real number $\sup \{ s \geq 0 | (R, \ba^s) \text{ is $F$-pure} \}.$  It was introduced in \cite{TakagiWatanabeFPureThresh} and is an analog of the log canonical threshold.  Also compare with the $F$-threshold introduced \cite{MustataTakagiWatanabeFThresholdsAndBernsteinSato}.

\item[$F$-threshold]  For a given pair $(R, \ba)$ with $R$ is a local ring with maximal ideal $\bm$, the \emph{$F$-threshold} of the pair is the limit $\lim_{e \to \infty} { \max \{r | \ba^r \nsubseteq \bm^{[p^e]} \} \over p^e }.$  This coincides with the $F$-pure threshold when $R$ is regular, but is distinct otherwise.  It was introduced in \cite{MustataTakagiWatanabeFThresholdsAndBernsteinSato}.

\item[$F$-rational]  An $F$-finite reduced local ring $R$ is called \emph{$F$-rational} if it is Cohen-Macaulay and there is no proper non-zero submodule $J \subseteq \omega_R$ such that dual to the Frobenius map $F_* \omega_{R} \to \omega_R$ sends $F_* J$ back into $J$.  Equivalently, $R$ is $F$-rational if all ideals generated by a full system of parameters are tightly closed.  See \cite{FedderWatanabe} for the original definitions.  $F$-rational singularities are closely related to \emph{rational singularities} in characteristic zero, see \cite{SmithFRatImpliesRat}, \cite{HaraRatImpliesFRat} and \cite{MehtaSrinivasRatImpliesFRat}.

\item[$F$-regular]  A ring $R$ of characteristic $p > 0$ is called \emph{$F$-regular} if all ideals in all localizations of $R$ are tightly closed, see \cite{HochsterHunekeTC1}.

\item[$F$-split]  A scheme $X$ of finite type over a perfect field characteristic $p > 0$ is called \emph{$F$-split (or Frobenius split)} if the Frobenius map $\O_X \to \O_X^{1/p}$ splits.  If $X$ is affine and $F$-finite, this is the same as $F$-pure.  The use of $F$-splittings to study the global geometry of schemes was introduced in \cite{MehtaRamanathanFrobeniusSplittingAndCohomologyVanishing} and \cite{RamananRamanathanProjectiveNormality}.  Also see \cite{HaboushKempfVanishing} and \cite{BrionKumarFrobeniusSplitting}.  $F$-split projective schemes are closely related to \emph{log Calabi-Yau pairs} in characteristic zero \cite{SchwedeSmithLogFanoVsGloballyFRegular}.

\item[$F$-signature]  Suppose that $(R, \bm)$ is a local ring of characteristic $p > 0$ with perfect residue field $R/\bm$.  Then the \emph{$F$-signature} of $R$, denoted $s(R)$, is the limit $\lim_{e \to \infty} a_e / p^{ed}$ where $a_e$ is the number of free $R$-summands of $R^{1/p^e}$; in other words $R^{1/p^e} = R^{\oplus a_e} \oplus M$.  It was first defined explicitly in \cite{HunekeLeuschkeTwoTheoremsAboutMaximal}, also see \cite{SmithVanDenBerghSimplicityOfDiff}.  The limit was shown to exist by the second author in \cite{TuckerFSignatureExists}.

\item[$F$-stable submodule]  Suppose $M$ is an $R$-module with a Frobenius action $F : M \to M$ (an additive map such that $F(rx) = r^p F(x)$).  Most typically $M = H^{\dim R}_{\bm}(R)$ and $F$ is the induced action of Frobenius.  Then an $F$-stable submodule $N \subseteq M$ is a submodule $N \subseteq M$ such that $F(N) \subseteq N$.

\item[(FFRT) Finite $F$-representation type]  Suppose we are given an $F$-finite complete local ring $(R, \bm)$ and consider the Krull-Schmidt decomposition $R^{1/p^e} = M_{1, e} \oplus \dots \oplus M_{n_e, e}$ of $R^{1/p^e}$ for each $e$.  We say that $R$ has \emph{finite $F$-representation type}, or simply \emph{FFRT}, if the set of isomorphism classes of the $M_{i, e}$ is finite (as $e$ varies).  Rings with FFRT were introduced by \cite{SmithVanDenBerghSimplicityOfDiff}.  It is worth mentioning that tight closure commutes with localization in rings with FFRT \cite{YaoModulesWithFFRT}.

\item[(Finitistic) test ideal]  Given a ring $R$ of characteristic $p$, the \emph{finitistic test ideal} $\tau_{\textnormal{fg}}(R)$ is defined to be $\bigcap_{I \subseteq R} (I : I^*)$ where $I^*$ is the tight closure of $I$.  Classically, finitistic test ideals were known simply as test ideals, see \cite{HochsterHunekeTC1}.

\item[Frobenius action]  Suppose that $R$ is a ring and $M$ is an $R$-module.  A \emph{Frobenius action on $M$} (or simply an \emph{$F$-action}) is an additive map $f : M \to M$ satisfying the rule $f(r.m) = r^p.m$ for all $r \in R$ and $m \in M$.

\item[Generalized test ideal]  In the literature, the test ideal $\tau(R, \ba^t)$ (or more generally for more complicated pairs and triples) is often called the \emph{generalized test ideal}.  At one point it was believed that the generalized test ideal was not made up of ``test elements'', but it is made up of appropriately defined test elements, see Exercise \ref{ex.SharpTestElements} and \cite{SchwedeSharpTestElements}.

\item[Globally $F$-regular]  A scheme $X$ of characteristic $p > 0$ is called \emph{globally $F$-regular} if for every effective divisor $D$ there exists an $e > 0$ such that the Frobenius map $\O_X \to (\O_X(D))^{1/p^e}$ splits.  If $X$ is affine and $F$-finite, this is the same as strongly $F$-regular.  Globally $F$-regular varieties were introduced in \cite{SmithGloballyFRegular}.  Globally $F$-regular projective schemes are closely related to \emph{log Fano pairs} in characteristic zero \cite{SchwedeSmithLogFanoVsGloballyFRegular}.

\item[Hartshorne-Speiser-Lyubeznik (HSL)-number]  Given a local ring $(R, \bm)$, then the \emph{Hartshorne-Speiser-Lyubeznik number} of $R$ (or simply the \emph{HSL-number}) is the smallest natural number $e \geq 0$ such that the kernel of the local cohomology $e$-iterated Frobenius $\ker \left( H^{\dim R}_{\bm}(R) \to H^{\dim R}_{\bm}(F^e_* R) \right)$ is equal to the kernel $\ker \left( H^{\dim R}_{\bm}(R) \to H^{\dim R}_{\bm}(F^{e+1}_* R)\right)$.  It is always a finite number by \cite{HartshorneSpeiserLocalCohomologyInCharacteristicP, LyubeznikFModulesApplicationsToLocalCohomology}.  See \cite[Definition 3.14]{SharpGradedAnnihilatorsOfModulesOverTheFrobeniusSkewPolynomialRing} where this definition is generalized to any Artinian $R$-module with a Frobenius action, instead of simply $H^{\dim R}_{\bm}(R)$.

\item[Hilbert-Kunz(-Monsky) multiplicity]  Given a local ring $(R, \bm)$ of characteristic $p > 0$ and an $\bm$-primary ideal $\ba$, the \emph{Hilbert-Kunz multiplicity of $\ba \subseteq R$} is defined to be $\lim_{e \to \infty} \textnormal{length}_R(R/\ba^{[p^e]})$, it is denoted by $e_{\textnormal{HK}}(\ba, R)$.  Although originally thought not to always exist, \cite[Page 1011]{KunzOnNoetherianRingsOfCharP}, Monsky showed it did indeed always exist in \cite{MonskyHKFunction}.  It is however notoriously difficult to compute.  There are many connections between the Hilbert-Kunz multiplicity and tight closure of ideals as well as the $F$-signature.

\item[(Open) $F$-XXX type]  For a given class of singularities ``$F$-XXX'' in characteristic $p$, a characteristic zero scheme $X$ (or pair as appropriate) is said to have \emph{dense $F$-XXX type} if for all sufficiently large finitely generated $\bZ$-algebras $A$ and families $X_A \to \Spec A$ of characteristic $p$-modules of $X$ (in particular, the generic point of that family agrees with $X$ up to field-base-change), there exists an open and Zariski-dense $U \subseteq \Spec A$ such that for all maximal ideals $\bq \in U$, the fiber $X_{\bq}$ has $F$-XXX singularities.

\item[Sharply $F$-pure]  A variant of $F$-purity for pairs introduced in \cite{SchwedeSharpTestElements}.  This variant uses the $\lceil (p^e - 1)\Delta \rceil$ or $\lceil t(p^e - 1) \rceil$-exponent instead of $\lfloor t(p^e - 1) \rfloor$ or $\lceil tp^e \rceil$.  One aspect that distinguishes it from previous $F$-purity definitions for pairs, see \cite{HaraWatanabeFRegFPure} \cite{TakagiInversion}, is that the test ideal of a sharply $F$-pure pair is always a radical ideal.  See \cite[Corollaries 3.15 and 4.3]{SchwedeSharpTestElements} as well as \cite{SchwedeRefinementOfSharplyFPureAndStrongFReg} and \cite{HernandezNewNotionsOfFPurity} for further refinements.

\item[Splitting prime]  Given an $F$-finite $F$-pure local ring $(R, \bm)$ of characteristic $p > 0$, the \emph{splitting prime} $\sP \subseteq R$ is defined to be the set $\{ c \in R | \phi(c^{1/p^e}) \in \bm, \forall e > 0, \forall \phi \in \Hom_R(R^{1/p^e}, R) \}$.  The splitting prime was introduced by Aberbach-Enescu in \cite{AberbachEnescuStructureOfFPure}.  It closely related to a \emph{minimal log canonical center} in characteristic zero, see \cite{SchwedeCentersOfFPurity}.  As an ideal, it is the largest \emph{$F$-pure center}.

\item[Strongly $F$-regular]  An $F$-finite reduced ring of characteristic $p > 0$ is called \emph{strongly $F$-regular} if for every $c \in R$ not contained in a minimal prime, there is an $R$-linear map $\phi : R^{1/p^e} \to R$, for some $e > 0$, which sends $c^{1/p^e}$ to $1$, see \cite{HochsterHunekeTightClosureAndStrongFRegularity}.  Strongly $F$-regular rings are closed related to rings with \emph{log terminal singularities} in characteristic zero, see \cite{HaraWatanabeFRegFPure}.  See \cite{HaraWatanabeFRegFPure, TakagiInversion,SchwedeTestIdealsInNonQGor} for generalizations to pairs/triples.  For generalizations to the non-$F$-finite setting, see \cite{HochsterFoundations}.

\item[Strong test ideal]  A \emph{strong test ideal} for a ring $R$ is any ideal $J$ (not contained in any minimal prime) such that $JI^* = JI$ for all ideals $I \subseteq R$.  They were introduced in \cite{HunekeStrongTestIdeals}.  The strong test ideal coincides with the test ideal in many cases, see \cite{HaraSmithTheStrongTestIdeal} and \cite{VraciuStrongTestIdeals}.  Additional generalizations to modules can be found in \cite{EnescuStrongTestModulesAndMultiplierIdeals}.

\item[Strongly $F$-pure]  A variant of $F$-purity for pairs introduced in \cite{HaraWatanabeFRegFPure}.  A pair $(R, \Delta)$ is strongly $F$-pure if there exists $\phi \in \Hom_R(F^e_* R(\lfloor p^e \Delta \rfloor), R) \subseteq \Hom_R(F^e_* R, R)$ such that $\phi(F^e_* R) = R$.

\item[Test ideal]  See \emph{big test ideal} or \emph{finitistic test ideal}.

\item[Test element]  See \emph{big/finitistic test element}.

\item[Tight closure]  The \emph{tight closure $I^*$} of an ideal $I \subseteq R$ is the set of $z \in R$ such that there exists $c \in R$, but not in any minimal prime of $R$, so that $c z^{p^e} \in I^{[p^e]}$ for all $e \geq 0$.  See \cite{HochsterHunekeTC1}.

\item[$q$-Weak test element]  An element $c \in R$ not in any minimal prime is called a \emph{$q$-weak test element} if for all $p^e \geq q = p^d$, we have that $c x^{p^e} \in I^{[p^e]}$ for all $x \in I^*$ as $I$ ranges over all ideals of $R$, see \cite[Section 6]{HochsterHunekeTC1}.

\item[Weakly $F$-regular]  A ring $R$ of characteristic $p > 0$ is called \emph{weakly $F$-regular} if all ideals are tightly closed, see \cite{HochsterHunekeTC1}.  It is an open question whether weakly $F$-regular rings are strongly $F$-regular.
\end{flushdesc}

}

\bibliographystyle{skalpha}
\bibliography{CommonBib}

\def\cprime{$'$} \def\cprime{$'$}
  \def\cfudot#1{\ifmmode\setbox7\hbox{$\accent"5E#1$}\else
  \setbox7\hbox{\accent"5E#1}\penalty 10000\relax\fi\raise 1\ht7
  \hbox{\raise.1ex\hbox to 1\wd7{\hss.\hss}}\penalty 10000 \hskip-1\wd7\penalty
  10000\box7}
\providecommand{\bysame}{\leavevmode\hbox to3em{\hrulefill}\thinspace}
\providecommand{\MR}{\relax\ifhmode\unskip\space\fi MR}
\providecommand{\MRhref}[2]{%
  \href{http://www.ams.org/mathscinet-getitem?mr=#1}{#2}
}
\providecommand{\href}[2]{#2}
\begin{thebibliography}{KKMSD73}

\bibitem[AE03]{AberbachEnescuTestIdealsAndBaseChange}
{\sc I.~M. Aberbach and F.~Enescu}: \emph{Test ideals and base change problems
  in tight closure theory}, Trans. Amer. Math. Soc. \textbf{355} (2003), no.~2,
  619--636 (electronic).

\bibitem[AE05]{AberbachEnescuStructureOfFPure}
{\sc I.~M. Aberbach and F.~Enescu}: \emph{The structure of {$F$}-pure rings},
  Math. Z. \textbf{250} (2005), no.~4, 791--806.

\bibitem[AE11]{AberbachEnescuNewEstimatesForHK}
{\sc I.~M. Aberbach and F.~Enescu}: \emph{New estimates of {H}ilbert-{K}unz
  multiplicities for local rings of fixed dimension}, arXiv:1101.5078.

\bibitem[AHH93]{AberbachHochsterHunekeLocalofTCandModsofFinProjDim}
{\sc I.~M. Aberbach, M.~Hochster, and C.~Huneke}: \emph{Localization of tight
  closure and modules of finite phantom projective dimension}, J. Reine Angew.
  Math. \textbf{434} (1993), 67--114.

\bibitem[AL03]{AberbachLeuschke}
{\sc I.~M. Aberbach and G.~J. Leuschke}: \emph{The {$F$}-signature and strong
  {$F$}-regularity}, Math. Res. Lett. \textbf{10} (2003), no.~1, 51--56.

\bibitem[AM99]{AberbachMacCrimmonSomeResultsOnTestElements}
{\sc I.~M. Aberbach and B.~MacCrimmon}: \emph{Some results on test elements},
  Proc. Edinburgh Math. Soc. (2) \textbf{42} (1999), no.~3, 541--549.

\bibitem[Abh66]{AbhyankarResolutionOfSingularitiesOfEmbedded}
{\sc S.~S. Abhyankar}: \emph{Resolution of singularities of embedded algebraic
  surfaces}, Pure and Applied Mathematics, Vol. 24, Academic Press, New York,
  1966.

\bibitem[Amb98]{AmbroSeminormalLocus}
{\sc F.~Ambro}: \emph{{The locus of log canonical singularities}},
  arXiv:math.AG/9806067.

\bibitem[And00]{AndersonElementaryLFunctions}
{\sc G.~W. Anderson}: \emph{An elementary approach to {$L$}-functions mod
  {$p$}}, J. Number Theory \textbf{80} (2000), no.~2, 291--303.

\bibitem[Art75]{ArtinWildlyRamifiedZ2Actions}
{\sc M.~Artin}: \emph{Wildly ramified {$Z/2$} actions in dimension two}, Proc.
  Amer. Math. Soc. \textbf{52} (1975), 60--64.

\bibitem[Bli04]{BlickleMultiplierIdealsAndModulesOnToric}
{\sc M.~Blickle}: \emph{Multiplier ideals and modules on toric varieties},
  Math. Z. \textbf{248} (2004), no.~1, 113--121.

\bibitem[Bli09]{BlickleAlgebras}
{\sc M.~Blickle}: \emph{Test ideals via algebras of $p^{-e}$-liner maps},
  arXiv:0912.2255, to appear in J. Algebraic Geom.

\bibitem[BB09a]{BlickleBoeckleCartierModulesFiniteness}
{\sc M.~Blickle and G.~B\"ockle}: \emph{Cartier modules: finiteness results},
  to appear in J. Reine Angew. Math., arXiv:0909.2531.

\bibitem[BB09b]{BlickleBoeckle.MinimalCartier}
{\sc M.~Blickle and G.~B{\"o}ckle}: \emph{Minimal cartier modules and
  applications to local cohomology}, manuscript in preparation, 2009.

\bibitem[BL04]{BlickleLazarsfeldMultiplier}
{\sc M.~Blickle and R.~Lazarsfeld}: \emph{An informal introduction to
  multiplier ideals}, Trends in commutative algebra, Math. Sci. Res. Inst.
  Publ., vol.~51, Cambridge Univ. Press, Cambridge, 2004, pp.~87--114.

\bibitem[BMS08]{BlickleMustataSmithDiscretenessAndRationalityOfFThresholds}
{\sc M.~Blickle, M.~Musta{\c{t}}{\u{a}}, and K.~Smith}: \emph{Discreteness and
  rationality of {F}-thresholds}, Michigan Math. J. \textbf{57} (2008), 43--61.

\bibitem[BMS09]{BlickleMustataSmithFThresholdsOfHypersurfaces}
{\sc M.~Blickle, M.~Musta{\c{t}}{\u{a}}, and K.~E. Smith}:
  \emph{{$F$}-thresholds of hypersurfaces}, Trans. Amer. Math. Soc.
  \textbf{361} (2009), no.~12, 6549--6565.

\bibitem[BSTZ10]{BlickleSchwedeTakagiZhang}
{\sc M.~Blickle, K.~Schwede, S.~Takagi, and W.~Zhang}: \emph{Discreteness and
  rationality of {$F$}-jumping numbers on singular varieties}, Math. Ann.
  \textbf{347} (2010), no.~4, 917--949.

\bibitem[BST11]{BlickleSchwedeTuckerFSignaturePairs}
{\sc M.~Blickle, K.~Schwede, and K.~Tucker}: \emph{{$F$}-signature of pairs},
  in preparation.

\bibitem[Bre03]{BrennerTightClosureOfProjectiveBundles}
{\sc H.~Brenner}: \emph{Tight closure and projective bundles}, J. Algebra
  \textbf{265} (2003), no.~1, 45--78.

\bibitem[Bre05]{BrennerComputingTightClosureInDim2}
{\sc H.~Brenner}: \emph{Computing the tight closure in dimension two}, Math.
  Comp. \textbf{74} (2005), no.~251, 1495--1518 (electronic).

\bibitem[Bre06]{BrennerRationalityOfHKMultiplicityInDim2}
{\sc H.~Brenner}: \emph{The rationality of the {H}ilbert-{K}unz multiplicity in
  graded dimension two}, Math. Ann. \textbf{334} (2006), no.~1, 91--110.

\bibitem[BHV08]{BrennerHerzogVillamayorThreeLectures}
{\sc H.~Brenner, J.~Herzog, and O.~Villamayor}: \emph{Three lectures on
  commutative algebra}, University Lecture Series, vol.~42, American
  Mathematical Society, Providence, RI, 2008, Lectures from the Winter School
  on Commutative Algebra and Applications held in Barcelona, January
  30--February 3, 2006, Edited by Gemma Colom{\'e}-Nin, Teresa Cortadellas
  Ben{\'{\i}}tez, Juan Elias and Santiago Zarzuela.

\bibitem[BM10]{BrennerMonsky}
{\sc H.~Brenner and P.~Monsky}: \emph{Tight closure does not commute with
  localization}, Ann. of Math. (2) \textbf{171} (2010), no.~1, 571--588.

\bibitem[BK05]{BrionKumarFrobeniusSplitting}
{\sc M.~Brion and S.~Kumar}: \emph{Frobenius splitting methods in geometry and
  representation theory}, Progress in Mathematics, vol. 231, Birkh\"auser
  Boston Inc., Boston, MA, 2005.

\bibitem[BS98]{BrodmannSharpLocalCohomology}
{\sc M.~P. Brodmann and R.~Y. Sharp}: \emph{Local cohomology: an algebraic
  introduction with geometric applications}, Cambridge Studies in Advanced
  Mathematics, vol.~60, Cambridge University Press, Cambridge, 1998.

\bibitem[BH98]{BrunsHerzog}
{\sc W.~Bruns and J.~Herzog}: \emph{Cohen-{M}acaulay rings}, second edition
  (july 28, 1998) ed., Cambridge Studies in Advanced Mathematics, vol.~39,
  Cambridge University Press, Cambridge, 1998.

\bibitem[CP08]{CossartPiltantResolutionOfThreefolds1}
{\sc V.~Cossart and O.~Piltant}: \emph{Resolution of singularities of
  threefolds in positive characteristic. {I}. {R}eduction to local
  uniformization on {A}rtin-{S}chreier and purely inseparable coverings}, J.
  Algebra \textbf{320} (2008), no.~3, 1051--1082.

\bibitem[CP09]{CossartPiltantResolutionOfThreefolds2}
{\sc V.~Cossart and O.~Piltant}: \emph{Resolution of singularities of
  threefolds in positive characteristic. {II}}, J. Algebra \textbf{321} (2009),
  no.~7, 1836--1976.

\bibitem[Cut09]{CutkoskyResolutionOfSingularitiesFor3Folds}
{\sc S.~D. Cutkosky}: \emph{Resolution of singularities for 3-folds in positive
  characteristic}, Amer. J. Math. \textbf{131} (2009), no.~1, 59--127.

\bibitem[dFH09]{deFernexHaconSingularitiesOnNormal}
{\sc T.~de~Fernex and C.~D. Hacon}: \emph{Singularities on normal varieties},
  Compos. Math. \textbf{145} (2009), no.~2, 393--414.

\bibitem[DEL00]{DemaillyEinLazSubadditivity}
{\sc J.-P. Demailly, L.~Ein, and R.~Lazarsfeld}: \emph{A subadditivity property
  of multiplier ideals}, Michigan Math. J. \textbf{48} (2000), 137--156,
  Dedicated to William Fulton on the occasion of his 60th birthday.

\bibitem[Ein97]{EinMultiplierIdealsVanishing}
{\sc L.~Ein}: \emph{Multiplier ideals, vanishing theorems and applications},
  Algebraic geometry---Santa Cruz 1995, Proc. Sympos. Pure Math., vol.~62,
  Amer. Math. Soc., Providence, RI, 1997, pp.~203--219.

\bibitem[ELS01]{EinLazSmithSymbolic}
{\sc L.~Ein, R.~Lazarsfeld, and K.~E. Smith}: \emph{Uniform bounds and symbolic
  powers on smooth varieties}, Invent. Math. \textbf{144} (2001), no.~2,
  241--252.

\bibitem[ELSV04]{EinLazSmithVarJumpingCoeffs}
{\sc L.~Ein, R.~Lazarsfeld, K.~E. Smith, and D.~Varolin}: \emph{Jumping
  coefficients of multiplier ideals}, Duke Math. J. \textbf{123} (2004), no.~3,
  469--506.

\bibitem[Eis95]{EisenbudCommAlgWithAView}
{\sc D.~Eisenbud}: \emph{Commutative algebra}, Graduate Texts in Mathematics,
  vol. 150, Springer-Verlag, New York, 1995, With a view toward algebraic
  geometry.

\bibitem[Eis10]{EisensteinGeneralizationsOfRestrictionFormula}
{\sc E.~Eisenstein}: \emph{Generalizations of the restriction theorem for
  multiplier ideals}, arXiv:1001.2841.

\bibitem[Elk81]{ElkikRationalityOfCanonicalSings}
{\sc R.~Elkik}: \emph{Rationalit\'e des singularit\'es canoniques}, Invent.
  Math. \textbf{64} (1981), no.~1, 1--6.

\bibitem[Ene03a]{EnescuFInjectiveRingsAndFStablePrimes}
{\sc F.~Enescu}: \emph{F-injective rings and {F}-stable primes}, Proc. Amer.
  Math. Soc. \textbf{131} (2003), no.~11, 3379--3386 (electronic).

\bibitem[Ene03b]{EnescuStrongTestModulesAndMultiplierIdeals}
{\sc F.~Enescu}: \emph{Strong test modules and multiplier ideals}, Manuscripta
  Math. \textbf{111} (2003), no.~4, 487--498.

\bibitem[EH08]{EnescuHochsterTheFrobeniusStructureOfLocalCohomology}
{\sc F.~Enescu and M.~Hochster}: \emph{The {F}robenius structure of local
  cohomology}, Algebra Number Theory \textbf{2} (2008), no.~7, 721--754.

\bibitem[ES05]{EnescuShimomotoUpperSCofHK}
{\sc F.~Enescu and K.~Shimomoto}: \emph{On the upper semi-continuity of the
  {H}ilbert-{K}unz multiplicity}, J. Algebra \textbf{285} (2005), no.~1,
  222--237.

\bibitem[EV92]{EsnaultViehwegLecturesOnVanishing}
{\sc H.~Esnault and E.~Viehweg}: \emph{Lectures on vanishing theorems}, DMV
  Seminar, vol.~20, Birkh\"auser Verlag, Basel, 1992.

\bibitem[Fed83]{FedderFPureRat}
{\sc R.~Fedder}: \emph{{$F$}-purity and rational singularity}, Trans. Amer.
  Math. Soc. \textbf{278} (1983), no.~2, 461--480.

\bibitem[FW89]{FedderWatanabe}
{\sc R.~Fedder and K.~Watanabe}: \emph{A characterization of {$F$}-regularity
  in terms of {$F$}-purity}, Commutative algebra (Berkeley, CA, 1987), Math.
  Sci. Res. Inst. Publ., vol.~15, Springer, New York, 1989, pp.~227--245.

\bibitem[Gab04]{Gabber.tStruc}
{\sc O.~Gabber}: \emph{Notes on some {$t$}-structures}, Geometric aspects of
  Dwork theory. Vol. I, II, Walter de Gruyter GmbH \& Co. KG, Berlin, 2004,
  pp.~711--734.

\bibitem[GR70]{GRVanishing}
{\sc H.~Grauert and O.~Riemenschneider}: \emph{Verschwindungss\"atze f\"ur
  analytische {K}ohomologiegruppen auf komplexen {R}\"aumen}, Invent. Math.
  \textbf{11} (1970), 263--292.

\bibitem[Hab80]{HaboushKempfVanishing}
{\sc W.~J. Haboush}: \emph{A short proof of the {K}empf vanishing theorem},
  Invent. Math. \textbf{56} (1980), no.~2, 109--112.

\bibitem[HM93]{HanMonskySomeSurprisingHKFunctions}
{\sc C.~Han and P.~Monsky}: \emph{Some surprising {H}ilbert-{K}unz functions},
  Math. Z. \textbf{214} (1993), no.~1, 119--135.

\bibitem[Han03]{HanesNotesOnHKFunction}
{\sc D.~Hanes}: \emph{Notes on the {H}ilbert-{K}unz function}, J. Algebra
  \textbf{265} (2003), no.~2, 619--630.

\bibitem[Har98a]{HaraRatImpliesFRat}
{\sc N.~Hara}: \emph{A characterization of rational singularities in terms of
  injectivity of {F}robenius maps}, Amer. J. Math. \textbf{120} (1998), no.~5,
  981--996.

\bibitem[Har98b]{HaraDimensionTwo}
{\sc N.~Hara}: \emph{Classification of two-dimensional {$F$}-regular and
  {$F$}-pure singularities}, Adv. Math. \textbf{133} (1998), no.~1, 33--53.

\bibitem[Har01]{HaraInterpretation}
{\sc N.~Hara}: \emph{Geometric interpretation of tight closure and test
  ideals}, Trans. Amer. Math. Soc. \textbf{353} (2001), no.~5, 1885--1906
  (electronic).

\bibitem[Har06]{HaraMonskyFPureThresholdsAndFJumpingExponents}
{\sc N.~Hara}: \emph{F-pure thresholds and {F}-jumping exponents in dimension
  two}, Math. Res. Lett. \textbf{13} (2006), no.~5-6, 747--760, With an
  appendix by Paul Monsky.

\bibitem[HS01]{HaraSmithTheStrongTestIdeal}
{\sc N.~Hara and K.~E. Smith}: \emph{The strong test ideal}, Illinois J. Math.
  \textbf{45} (2001), no.~3, 949--964.

\bibitem[HT04]{HaraTakagiOnAGeneralizationOfTestIdeals}
{\sc N.~Hara and S.~Takagi}: \emph{On a generalization of test ideals}, Nagoya
  Math. J. \textbf{175} (2004), 59--74.

\bibitem[HW02]{HaraWatanabeFRegFPure}
{\sc N.~Hara and K.-I. Watanabe}: \emph{F-regular and {F}-pure rings vs. log
  terminal and log canonical singularities}, J. Algebraic Geom. \textbf{11}
  (2002), no.~2, 363--392.

\bibitem[HY03]{HaraYoshidaGeneralizationOfTightClosure}
{\sc N.~Hara and K.-I. Yoshida}: \emph{A generalization of tight closure and
  multiplier ideals}, Trans. Amer. Math. Soc. \textbf{355} (2003), no.~8,
  3143--3174 (electronic).

\bibitem[Har66]{HartshorneResidues}
{\sc R.~Hartshorne}: \emph{Residues and duality}, Lecture notes of a seminar on
  the work of A. Grothendieck, given at Harvard 1963/64. With an appendix by P.
  Deligne. Lecture Notes in Mathematics, No. 20, Springer-Verlag, Berlin, 1966.

\bibitem[Har67]{HartshorneLocalCohomology}
{\sc R.~Hartshorne}: \emph{Local cohomology}, A seminar given by A.
  Grothendieck, Harvard University, Fall, vol. 1961, Springer-Verlag, Berlin,
  1967.

\bibitem[Har77]{Hartshorne}
{\sc R.~Hartshorne}: \emph{Algebraic geometry}, Springer-Verlag, New York,
  1977, Graduate Texts in Mathematics, No. 52.

\bibitem[Har94]{HartshorneGeneralizedDivisorsOnGorensteinSchemes}
{\sc R.~Hartshorne}: \emph{Generalized divisors on {G}orenstein schemes},
  Proceedings of Conference on Algebraic Geometry and Ring Theory in honor of
  Michael Artin, Part III (Antwerp, 1992), vol.~8, 1994, pp.~287--339.

\bibitem[Har07]{HartshonreGeneralizedDivisorsAndBiliaison}
{\sc R.~Hartshorne}: \emph{Generalized divisors and biliaison}, Illinois J.
  Math. \textbf{51} (2007), no.~1, 83--98 (electronic).

\bibitem[HS77]{HartshorneSpeiserLocalCohomologyInCharacteristicP}
{\sc R.~Hartshorne and R.~Speiser}: \emph{Local cohomological dimension in
  characteristic {$p$}}, Ann. of Math. (2) \textbf{105} (1977), no.~1, 45--79.

\bibitem[Her11a]{HernandezLCvsFPurity}
{\sc D.~Hern\'andez}: \emph{Log canonical singularities and $f$-purity for
  polynomials over $\bc$}, preprint, {\tt
  http://www-personal.umich.edu/{\textasciitilde}dhernan/daniel/Research.html}.

\bibitem[Her11b]{HernandezNewNotionsOfFPurity}
{\sc D.~Hern\'andez}: \emph{New notions of {$F$}-purity}, work in progress.

\bibitem[Hoc77]{HochsterCyclicPurity}
{\sc M.~Hochster}: \emph{Cyclic purity versus purity in excellent {N}oetherian
  rings}, Trans. Amer. Math. Soc. \textbf{231} (1977), no.~2, 463--488.

\bibitem[Hoc07]{HochsterFoundations}
{\sc M.~Hochster}: \emph{Foundations of tight closure theory}, lecture notes
  from a course taught on the University of Michigan Fall 2007 (2007), {\tt
  http://www.math.lsa.umich.edu/{\textasciitilde}hochster/711F07/fndtc.pdf}.

\bibitem[HH89]{HochsterHunekeTightClosureAndStrongFRegularity}
{\sc M.~Hochster and C.~Huneke}: \emph{Tight closure and strong
  {$F$}-regularity}, M\'em. Soc. Math. France (N.S.) (1989), no.~38, 119--133,
  Colloque en l'honneur de Pierre Samuel (Orsay, 1987).

\bibitem[HH90]{HochsterHunekeTC1}
{\sc M.~Hochster and C.~Huneke}: \emph{Tight closure, invariant theory, and the
  {B}rian\c con-{S}koda theorem}, J. Amer. Math. Soc. \textbf{3} (1990), no.~1,
  31--116.

\bibitem[HH94]{HochsterHunekeFRegularityTestElementsBaseChange}
{\sc M.~Hochster and C.~Huneke}: \emph{{$F$}-regularity, test elements, and
  smooth base change}, Trans. Amer. Math. Soc. \textbf{346} (1994), no.~1,
  1--62.

\bibitem[HH02]{HochsterHunekeComparisonOfSymbolic}
{\sc M.~Hochster and C.~Huneke}: \emph{Comparison of symbolic and ordinary
  powers of ideals}, Invent. Math. \textbf{147} (2002), no.~2, 349--369.

\bibitem[HH06]{HochsterHunekeTightClosureInEqualCharactersticZero}
{\sc M.~Hochster and C.~Huneke}: \emph{Tight closure in equal characteristic
  zero}, A preprint of a manuscript, 2006.

\bibitem[HH07]{HochsterHunekeFineBehaviorOfSymbolicPowers}
{\sc M.~Hochster and C.~Huneke}: \emph{Fine behavior of symbolic powers of
  ideals}, Illinois J. Math. \textbf{51} (2007), no.~1, 171--183 (electronic).

\bibitem[HR76]{HochsterRobertsFrobeniusLocalCohomology}
{\sc M.~Hochster and J.~L. Roberts}: \emph{The purity of the {F}robenius and
  local cohomology}, Advances in Math. \textbf{21} (1976), no.~2, 117--172.

\bibitem[Hun96]{HunekeTightClosureBook}
{\sc C.~Huneke}: \emph{Tight closure and its applications}, CBMS Regional
  Conference Series in Mathematics, vol.~88, Published for the Conference Board
  of the Mathematical Sciences, Washington, DC, 1996, With an appendix by
  Melvin Hochster.

\bibitem[Hun97]{HunekeStrongTestIdeals}
{\sc C.~Huneke}: \emph{Tight closure and strong test ideals}, J. Pure Appl.
  Algebra \textbf{122} (1997), no.~3, 243--250.

\bibitem[HL02]{HunekeLeuschkeTwoTheoremsAboutMaximal}
{\sc C.~Huneke and G.~J. Leuschke}: \emph{Two theorems about maximal
  {C}ohen-{M}acaulay modules}, Math. Ann. \textbf{324} (2002), no.~2, 391--404.

\bibitem[HS06]{HunekeSwansonIntegralClosure}
{\sc C.~Huneke and I.~Swanson}: \emph{Integral closure of ideals, rings, and
  modules}, London Mathematical Society Lecture Note Series, vol. 336,
  Cambridge University Press, Cambridge, 2006.

\bibitem[HY02]{HunekeYaoMinimalHKMultiplicity}
{\sc C.~Huneke and Y.~Yao}: \emph{Unmixed local rings with minimal
  {H}ilbert-{K}unz multiplicity are regular}, Proc. Amer. Math. Soc.
  \textbf{130} (2002), no.~3, 661--665 (electronic).

\bibitem[Kat98]{KatzmanTheComplexityofFrobPowers}
{\sc M.~Katzman}: \emph{The complexity of {F}robenius powers of ideals}, J.
  Algebra \textbf{203} (1998), no.~1, 211--225.

\bibitem[Kat08]{KatzmanParameterTestIdealOfCMRings}
{\sc M.~Katzman}: \emph{Parameter-test-ideals of {C}ohen-{M}acaulay rings},
  Compos. Math. \textbf{144} (2008), no.~4, 933--948.

\bibitem[Kat10]{KatzmanANonFGAlgebraOfFrob}
{\sc M.~Katzman}: \emph{A non-finitely generated algebra of {F}robenius maps},
  Proc. Amer. Math. Soc. \textbf{138} (2010), no.~7, 2381--2383.

\bibitem[KLZ09]{KatzmanLyubeznikZhangOnDiscretenessAndRationality}
{\sc M.~Katzman, G.~Lyubeznik, and W.~Zhang}: \emph{On the discreteness and
  rationality of {$F$}-jumping coefficients}, J. Algebra \textbf{322} (2009),
  no.~9, 3238--3247.

\bibitem[Kaw07]{KawakitaInversion}
{\sc M.~Kawakita}: \emph{Inversion of adjunction on log canonicity}, Invent.
  Math. \textbf{167} (2007), no.~1, 129--133.

\bibitem[Kaw82]{KawamataVanishing}
{\sc Y.~Kawamata}: \emph{A generalization of {K}odaira-{R}amanujam's vanishing
  theorem}, Math. Ann. \textbf{261} (1982), no.~1, 43--46.

\bibitem[Kaw97]{KawamataSubadjunctionOne}
{\sc Y.~Kawamata}: \emph{Subadjunction of log canonical divisors for a
  subvariety of codimension {$2$}}, Birational algebraic geometry ({B}altimore,
  {MD}, 1996), Contemp. Math., vol. 207, Amer. Math. Soc., Providence, RI,
  1997, pp.~79--88.

\bibitem[Kaw98]{KawamataSubadjunction2}
{\sc Y.~Kawamata}: \emph{Subadjunction of log canonical divisors. {II}}, Amer.
  J. Math. \textbf{120} (1998), no.~5, 893--899.

\bibitem[KKMSD73]{KempfToroidalEmbeddings}
{\sc G.~Kempf, F.~F. Knudsen, D.~Mumford, and B.~Saint-Donat}: \emph{Toroidal
  embeddings. {I}}, Lecture Notes in Mathematics, Vol. 339, Springer-Verlag,
  Berlin, 1973.

\bibitem[Kol97]{KollarSingularitiesOfPairs}
{\sc J.~Koll{\'a}r}: \emph{Singularities of pairs}, Algebraic geometry---Santa
  Cruz 1995, Proc. Sympos. Pure Math., vol.~62, Amer. Math. Soc., Providence,
  RI, 1997, pp.~221--287.

\bibitem[KK10]{KollarKovacsLCImpliesDB}
{\sc J.~Koll{\'a}r and S.~J. Kov{\'a}cs}: \emph{Log canonical singularities are
  {D}u {B}ois}, J. Amer. Math. Soc. \textbf{23} (2010), no.~3, 791--813.

\bibitem[KM98]{KollarMori}
{\sc J.~Koll{\'a}r and S.~Mori}: \emph{Birational geometry of algebraic
  varieties}, Cambridge Tracts in Mathematics, vol. 134, Cambridge University
  Press, Cambridge, 1998, With the collaboration of C. H. Clemens and A. Corti,
  Translated from the 1998 Japanese original.

\bibitem[Kov00]{KovacsRat}
{\sc S.~J. Kov{\'a}cs}: \emph{A characterization of rational singularities},
  Duke Math. J. \textbf{102} (2000), no.~2, 187--191.

\bibitem[KS11]{KovacsSchwedeDuBoisSurvey}
{\sc S.~J. Kov\'acs and K.~Schwede}: \emph{Hodge theory meets the minimal model
  program: a survey of log canonical and {D}u {B}ois singularities}, Topology
  of Stratified Spaces (G.~Friedman, E.~Hunsicker, A.~Libgober, and L.~Maxim,
  eds.), Math. Sci. Res. Inst. Publ., vol.~58, Cambridge Univ. Press,
  Cambridge, 2011, pp.~51--94.

\bibitem[KSS10]{KovacsSchwedeSmithLCImpliesDuBois}
{\sc S.~J. Kov{\'a}cs, K.~Schwede, and K.~E. Smith}: \emph{The canonical sheaf
  of {D}u {B}ois singularities}, Adv. Math. \textbf{224} (2010), no.~4,
  1618--1640.

\bibitem[KM09]{KumarMehtaFiniteness}
{\sc S.~Kumar and V.~B. Mehta}: \emph{Finiteness of the number of compatibly
  split subvarieties}, Int. Math. Res. Not. IMRN (2009), no.~19, 3595--3597.

\bibitem[Kun69]{KunzCharacterizationsOfRegularLocalRings}
{\sc E.~Kunz}: \emph{Characterizations of regular local rings for
  characteristic {$p$}}, Amer. J. Math. \textbf{91} (1969), 772--784.

\bibitem[Kun76]{KunzOnNoetherianRingsOfCharP}
{\sc E.~Kunz}: \emph{On {N}oetherian rings of characteristic {$p$}}, Amer. J.
  Math. \textbf{98} (1976), no.~4, 999--1013.

\bibitem[Kun86]{KunzKahlerDifferentials}
{\sc E.~Kunz}: \emph{K\"ahler differentials}, Advanced Lectures in Mathematics,
  Friedr. Vieweg \& Sohn, Braunschweig, 1986.

\bibitem[Laz04]{LazarsfeldPositivity2}
{\sc R.~Lazarsfeld}: \emph{Positivity in algebraic geometry. {II}}, Ergebnisse
  der Mathematik und ihrer Grenzgebiete. 3. Folge. A Series of Modern Surveys
  in Mathematics [Results in Mathematics and Related Areas. 3rd Series. A
  Series of Modern Surveys in Mathematics], vol.~49, Springer-Verlag, Berlin,
  2004, Positivity for vector bundles, and multiplier ideals.

\bibitem[LL07]{LazarsfeldLeeLocalSyzygiesOfMultiplierIdeals}
{\sc R.~Lazarsfeld and K.~Lee}: \emph{Local syzygies of multiplier ideals},
  Invent. Math. \textbf{167} (2007), no.~2, 409--418.

\bibitem[Lec64]{LechInequalitiesRelatedToCertainCouples}
{\sc C.~Lech}: \emph{Inequalities related to certain couples of local rings},
  Acta Math. \textbf{112} (1964), 69--89.

\bibitem[Lip78]{LipmanTwoDimensionalDesingularization}
{\sc J.~Lipman}: \emph{Desingularization of two-dimensional schemes}, Ann.
  Math. (2) \textbf{107} (1978), no.~1, 151--207.

\bibitem[Lip94]{LipmanAdjointsOfIdeals}
{\sc J.~Lipman}: \emph{Adjoints of ideals in regular local rings}, Math. Res.
  Lett. \textbf{1} (1994), no.~6, 739--755, With an appendix by Steven Dale
  Cutkosky.

\bibitem[Lyu97]{LyubeznikFModulesApplicationsToLocalCohomology}
{\sc G.~Lyubeznik}: \emph{{$F$}-modules: applications to local cohomology and
  {$D$}-modules in characteristic {$p>0$}}, J. Reine Angew. Math. \textbf{491}
  (1997), 65--130.

\bibitem[LS99]{LyubeznikSmithStrongWeakFregularityEquivalentforGraded}
{\sc G.~Lyubeznik and K.~E. Smith}: \emph{Strong and weak {$F$}-regularity are
  equivalent for graded rings}, Amer. J. Math. \textbf{121} (1999), no.~6,
  1279--1290.

\bibitem[LS01]{LyubeznikSmithCommutationOfTestIdealWithLocalization}
{\sc G.~Lyubeznik and K.~E. Smith}: \emph{On the commutation of the test ideal
  with localization and completion}, Trans. Amer. Math. Soc. \textbf{353}
  (2001), no.~8, 3149--3180 (electronic).

\bibitem[McD03]{McDermottTestIdealsInDiagonalHypersurfaces}
{\sc M.~A. McDermott}: \emph{Test ideals in diagonal hypersurface rings. {II}},
  J. Algebra \textbf{264} (2003), no.~1, 296--304.

\bibitem[MR85]{MehtaRamanathanFrobeniusSplittingAndCohomologyVanishing}
{\sc V.~B. Mehta and A.~Ramanathan}: \emph{Frobenius splitting and cohomology
  vanishing for {S}chubert varieties}, Ann. of Math. (2) \textbf{122} (1985),
  no.~1, 27--40.

\bibitem[MS97]{MehtaSrinivasRatImpliesFRat}
{\sc V.~B. Mehta and V.~Srinivas}: \emph{A characterization of rational
  singularities}, Asian J. Math. \textbf{1} (1997), no.~2, 249--271.

\bibitem[Mon83]{MonskyHKFunction}
{\sc P.~Monsky}: \emph{The {H}ilbert-{K}unz function}, Math. Ann. \textbf{263}
  (1983), no.~1, 43--49.

\bibitem[Mon08]{MonskyRationalityOfHKMultCounterExampleLikely}
{\sc P.~Monsky}: \emph{Rationality of {H}ilbert-{K}unz multiplicities: a likely
  counterexample}, Michigan Math. J. \textbf{57} (2008), 605--613, Special
  volume in honor of Melvin Hochster.

\bibitem[Mon09]{MonskyTranscendenceOfSomeHKMultiplicities}
{\sc P.~Monsky}: \emph{Transcendence of some hilbert-kunz multiplicities
  (modulo a conjecture)}, arXiv:0908.0971.

\bibitem[Mus02]{MustataMultiplierIdealofASum}
{\sc M.~Musta{\c{t}}{\v{a}}}: \emph{The multiplier ideals of a sum of ideals},
  Trans. Amer. Math. Soc. \textbf{354} (2002), no.~1, 205--217 (electronic).

\bibitem[Mus10]{MustataOrdinary2}
{\sc M.~Musta{\c{t}}{\v{a}}}: \emph{Ordinary varieties and the comparison
  between multiplier ideals and test ideals ii}, arXiv:1012.2915, to appear in
  Proc. Amer. Math. Soc.

\bibitem[MS10]{MustataSrinivasOrdinary}
{\sc M.~Musta{\c{t}}{\v{a}} and V.~Srinivas}: \emph{Ordinary varieties and the
  comparison between multiplier ideals and test ideals}, arXiv:1012.2818, to
  appear in Nagoya Math. J.

\bibitem[MTW05]{MustataTakagiWatanabeFThresholdsAndBernsteinSato}
{\sc M.~Musta{\c{t}}{\v{a}}, S.~Takagi, and K.-i. Watanabe}: \emph{F-thresholds
  and {B}ernstein-{S}ato polynomials}, European Congress of Mathematics, Eur.
  Math. Soc., Z\"urich, 2005, pp.~341--364.

\bibitem[MY09]{MustataYoshidaTestIdealVsMultiplierIdeals}
{\sc M.~Musta{\c{t}}{\u{a}} and K.-I. Yoshida}: \emph{Test ideals vs.
  multiplier ideals}, Nagoya Math. J. \textbf{193} (2009), 111--128.

\bibitem[Nad89]{NadelMultiplierIdealSheavesAndExistence}
{\sc A.~M. Nadel}: \emph{Multiplier ideal sheaves and existence of
  {K}\"ahler-{E}instein metrics of positive scalar curvature}, Proc. Nat. Acad.
  Sci. U.S.A. \textbf{86} (1989), no.~19, 7299--7300.

\bibitem[PS73]{PeskineSzpiroDimensionProjective}
{\sc C.~Peskine and L.~Szpiro}: \emph{Dimension projective finie et cohomologie
  locale. {A}pplications \`a la d\'emonstration de conjectures de {M}.
  {A}uslander, {H}. {B}ass et {A}. {G}rothendieck}, Inst. Hautes \'Etudes Sci.
  Publ. Math. (1973), no.~42, 47--119.

\bibitem[RR85]{RamananRamanathanProjectiveNormality}
{\sc S.~Ramanan and A.~Ramanathan}: \emph{Projective normality of flag
  varieties and {S}chubert varieties}, Invent. Math. \textbf{79} (1985), no.~2,
  217--224.

\bibitem[Sai00]{SaitoMixedHodge}
{\sc M.~Saito}: \emph{Mixed {H}odge complexes on algebraic varieties}, Math.
  Ann. \textbf{316} (2000), no.~2, 283--331.

\bibitem[Sch08]{SchwedeSharpTestElements}
{\sc K.~Schwede}: \emph{Generalized test ideals, sharp {$F$}-purity, and sharp
  test elements}, Math. Res. Lett. \textbf{15} (2008), no.~6, 1251--1261.

\bibitem[Sch09a]{SchwedeFAdjunction}
{\sc K.~Schwede}: \emph{{$F$}-adjunction}, Algebra Number Theory \textbf{3}
  (2009), no.~8, 907--950.

\bibitem[Sch09b]{SchwedeFInjectiveAreDuBois}
{\sc K.~Schwede}: \emph{{$F$}-injective singularities are {D}u {B}ois}, Amer.
  J. Math. \textbf{131} (2009), no.~2, 445--473.

\bibitem[Sch10a]{SchwedeCentersOfFPurity}
{\sc K.~Schwede}: \emph{Centers of {$F$}-purity}, Math. Z. \textbf{265} (2010),
  no.~3, 687--714.

\bibitem[Sch10b]{SchwedeRefinementOfSharplyFPureAndStrongFReg}
{\sc K.~Schwede}: \emph{A refinement of sharply {$F$}-pure and strongly
  {$F$}-regular pairs}, J. Commut. Algebra \textbf{2} (2010), no.~1, 91--109.

\bibitem[Sch11a]{SchwedeDiscretenessQGorenstein}
{\sc K.~Schwede}: \emph{A note on discreteness of ${F}$-jumping numbers},
  arXiv:1004.1377 to appear in Proc. Amer. Math. Soc.

\bibitem[Sch11b]{SchwedeTestIdealsInNonQGor}
{\sc K.~Schwede}: \emph{Test ideals in non-{$\mathbb{{Q}}$}-{G}orenstein
  rings}, Trans. Amer. Math. Soc. \textbf{363} (2011), no.~11, 5925--5941.

\bibitem[SS10]{SchwedeSmithLogFanoVsGloballyFRegular}
{\sc K.~Schwede and K.~E. Smith}: \emph{Globally {$F$}-regular and log {F}ano
  varieties}, Adv. Math. \textbf{224} (2010), no.~3, 863--894.

\bibitem[ST08]{SchwedeTakagiRationalPairs}
{\sc K.~Schwede and S.~Takagi}: \emph{Rational singularities associated to
  pairs}, Michigan Math. J. \textbf{57} (2008), 625--658.

\bibitem[ST10a]{SchwedeTuckerTestIdealFinite}
{\sc K.~Schwede and K.~Tucker}: \emph{On the behavior of test ideals under
  finite morphisms.}, arXiv:1003.433v1.

\bibitem[ST10b]{SchwedeTuckerNumberOfFSplit}
{\sc K.~Schwede and K.~Tucker}: \emph{On the number of compatibly {F}robenius
  split subvarieties, prime {$F$}-ideals, and log canonical centers}, Ann.
  Inst. Fourier (Grenoble) \textbf{60} (2010), no.~5, 1515--1531.

\bibitem[Sha07]{SharpGradedAnnihilatorsOfModulesOverTheFrobeniusSkewPolynomialRing}
{\sc R.~Y. Sharp}: \emph{Graded annihilators of modules over the {F}robenius
  skew polynomial ring, and tight closure}, Trans. Amer. Math. Soc.
  \textbf{359} (2007), no.~9, 4237--4258 (electronic).

\bibitem[ST09]{ShibutaTakagiLCThresholds}
{\sc T.~Shibuta and S.~Takagi}: \emph{Log canonical thresholds of binomial
  ideals}, Manuscripta Math. \textbf{130} (2009), no.~1, 45--61.

\bibitem[Sin98]{SinghComputationofTightClosureDiagonalHypersurfaces}
{\sc A.~K. Singh}: \emph{A computation of tight closure in diagonal
  hypersurfaces}, J. Algebra \textbf{203} (1998), no.~2, 579--589.

\bibitem[Sin99a]{SinghDeformationOfFPurity}
{\sc A.~K. Singh}: \emph{Deformation of {$F$}-purity and {$F$}-regularity}, J.
  Pure Appl. Algebra \textbf{140} (1999), no.~2, 137--148.

\bibitem[Sin99b]{SinghFregularityDoesNotDeform}
{\sc A.~K. Singh}: \emph{{$F$}-regularity does not deform}, Amer. J. Math.
  \textbf{121} (1999), no.~4, 919--929.

\bibitem[Sin05]{SinghFSignatureOfAffineSemigroup}
{\sc A.~K. Singh}: \emph{The {$F$}-signature of an affine semigroup ring}, J.
  Pure Appl. Algebra \textbf{196} (2005), no.~2-3, 313--321.

\bibitem[Siu05]{SiuMultiplierIdealSheavesSurvey}
{\sc Y.-T. Siu}: \emph{Multiplier ideal sheaves in complex and algebraic
  geometry}, Sci. China Ser. A \textbf{48} (2005), no.~suppl., 1--31.

\bibitem[Siu09]{SiuDynamicMultiplierSurvey}
{\sc Y.-T. Siu}: \emph{Dynamic multiplier ideal sheaves and the construction of
  rational curves in {F}ano manifolds}, Complex analysis and digital geometry,
  Acta Univ. Upsaliensis Skr. Uppsala Univ. C Organ. Hist., vol.~86, Uppsala
  Universitet, Uppsala, 2009, pp.~323--360.

\bibitem[Smi94]{SmithTightClosureParameter}
{\sc K.~E. Smith}: \emph{Tight closure of parameter ideals}, Invent. Math.
  \textbf{115} (1994), no.~1, 41--60.

\bibitem[Smi95]{SmithTestIdeals}
{\sc K.~E. Smith}: \emph{Test ideals in local rings}, Trans. Amer. Math. Soc.
  \textbf{347} (1995), no.~9, 3453--3472.

\bibitem[Smi97]{SmithFRatImpliesRat}
{\sc K.~E. Smith}: \emph{{$F$}-rational rings have rational singularities},
  Amer. J. Math. \textbf{119} (1997), no.~1, 159--180.

\bibitem[Smi00a]{SmithGloballyFRegular}
{\sc K.~E. Smith}: \emph{Globally {F}-regular varieties: applications to
  vanishing theorems for quotients of {F}ano varieties}, Michigan Math. J.
  \textbf{48} (2000), 553--572, Dedicated to William Fulton on the occasion of
  his 60th birthday.

\bibitem[Smi00b]{SmithMultiplierTestIdeals}
{\sc K.~E. Smith}: \emph{The multiplier ideal is a universal test ideal}, Comm.
  Algebra \textbf{28} (2000), no.~12, 5915--5929, Special issue in honor of
  Robin Hartshorne.

\bibitem[Smi01]{SmithTightClosure}
{\sc K.~E. Smith}: \emph{Tight closure and vanishing theorems}, School on
  Vanishing Theorems and Effective Results in Algebraic Geometry (Trieste,
  2000), ICTP Lect. Notes, vol.~6, Abdus Salam Int. Cent. Theoret. Phys.,
  Trieste, 2001, pp.~149--213.

\bibitem[SVdB97]{SmithVanDenBerghSimplicityOfDiff}
{\sc K.~E. Smith and M.~Van~den Bergh}: \emph{Simplicity of rings of
  differential operators in prime characteristic}, Proc. London Math. Soc. (3)
  \textbf{75} (1997), no.~1, 32--62.

\bibitem[Stu08]{StubbsThesis}
{\sc J.~F. Stubbs}: \emph{Potent elements and tight closure in {A}rtinian
  modules}, ProQuest LLC, Ann Arbor, MI, 2008, Thesis (Ph.D.)--University of
  Michigan.

\bibitem[Tak04a]{TakagiInversion}
{\sc S.~Takagi}: \emph{F-singularities of pairs and inversion of adjunction of
  arbitrary codimension}, Invent. Math. \textbf{157} (2004), no.~1, 123--146.

\bibitem[Tak04b]{TakagiInterpretationOfMultiplierIdeals}
{\sc S.~Takagi}: \emph{An interpretation of multiplier ideals via tight
  closure}, J. Algebraic Geom. \textbf{13} (2004), no.~2, 393--415.

\bibitem[Tak06]{TakagiFormulasForMultiplierIdeals}
{\sc S.~Takagi}: \emph{Formulas for multiplier ideals on singular varieties},
  Amer. J. Math. \textbf{128} (2006), no.~6, 1345--1362.

\bibitem[Tak08]{TakagiPLTAdjoint}
{\sc S.~Takagi}: \emph{A characteristic {$p$} analogue of plt singularities and
  adjoint ideals}, Math. Z. \textbf{259} (2008), no.~2, 321--341.

\bibitem[Tak11]{TakagiAdjointIdealsAndACorrespondence}
{\sc S.~Takagi}: \emph{Adjoint ideals and a correspondence between log
  canonicity and {F}-purity}, arXiv:1105.0072.

\bibitem[TT08]{TakagiTakahashiDModulesOverRingsWithFFRT}
{\sc S.~Takagi and R.~Takahashi}: \emph{{$D$}-modules over rings with finite
  {$F$}-representation type}, Math. Res. Lett. \textbf{15} (2008), no.~3,
  563--581.

\bibitem[TW04]{TakagiWatanabeFPureThresh}
{\sc S.~Takagi and K.-i. Watanabe}: \emph{On {F}-pure thresholds}, J. Algebra
  \textbf{282} (2004), no.~1, 278--297.

\bibitem[TY08]{TakagiYoshidaGeneralizedTestIdealsAndSymbolicPowers}
{\sc S.~Takagi and K.-i. Yoshida}: \emph{Generalized test ideals and symbolic
  powers}, Michigan Math. J. \textbf{57} (2008), 711--724, Special volume in
  honor of Melvin Hochster.

\bibitem[Tuc11]{TuckerFSignatureExists}
{\sc K.~Tucker}: \emph{{$F$}-signature exists}, arXiv:1103.4173.

\bibitem[Vas98]{VassilevTestIdeals}
{\sc J.~C. Vassilev}: \emph{Test ideals in quotients of {$F$}-finite regular
  local rings}, Trans. Amer. Math. Soc. \textbf{350} (1998), no.~10,
  4041--4051.

\bibitem[Vie82]{ViehwegVanishingTheorems}
{\sc E.~Viehweg}: \emph{Vanishing theorems}, J. Reine Angew. Math. \textbf{335}
  (1982), 1--8.

\bibitem[Vra02]{VraciuStrongTestIdeals}
{\sc A.~Vraciu}: \emph{Strong test ideals}, J. Pure Appl. Algebra \textbf{167}
  (2002), no.~2-3, 361--373.

\bibitem[Vra08]{VraciuNewatTightClosure}
{\sc A.~Vraciu}: \emph{A new version of {${\mathfrak{a}}$}-tight closure},
  Nagoya Math. J. \textbf{192} (2008), 1--25.

\bibitem[WY00]{WatanabeYoshidaHKMultAndInequality}
{\sc K.-i. Watanabe and K.-i. Yoshida}: \emph{Hilbert-{K}unz multiplicity and
  an inequality between multiplicity and colength}, J. Algebra \textbf{230}
  (2000), no.~1, 295--317.

\bibitem[WY04]{WatanabeYoshidaMinimal}
{\sc K.-i. Watanabe and K.-i. Yoshida}: \emph{Minimal relative {H}ilbert-{K}unz
  multiplicity}, Illinois J. Math. \textbf{48} (2004), no.~1, 273--294.

\bibitem[WY05]{watanabeyoshidaHK3dimlocalrings}
{\sc K.-i. Watanabe and K.-i. Yoshida}: \emph{Hilbert-{K}unz multiplicity of
  three-dimensional local rings}, Nagoya Math. J. \textbf{177} (2005), 47--75.

\bibitem[Yao05]{YaoModulesWithFFRT}
{\sc Y.~Yao}: \emph{Modules with finite {$F$}-representation type}, J. London
  Math. Soc. (2) \textbf{72} (2005), no.~1, 53--72.

\end{thebibliography}

\end{document}